\documentclass[12pt,reqno]{amsart}

\usepackage{amsthm,amsmath,amsfonts,epsfig,lscape}

\usepackage{caption}
\usepackage{float}
\usepackage{placeins}

\usepackage{color}

\theoremstyle{plain}
\newtheorem{thm}{Theorem}
\newtheorem{proposition}{Proposition}

\newtheorem{lemma}{Lemma}

\theoremstyle{definition}

\theoremstyle{remark}

\newcommand{\pen}{\operatorname{pen}}
\newcommand{\X}{\mathbf{X}}

\newcommand{\f}{\mathbf{f}}
\newcommand{\g}{\mathbf{g}}

\newcommand{\w}{\mathbf{w}}
\newcommand{\error}{\boldsymbol{\epsilon}}

\newcommand{\EventLambda}{\mathcal{E}_\lambda}
\newcommand{\EventEta}{\mathcal{E}_\eta}
\newcommand{\EventNu}{\mathcal{E}_\nu}
\newcommand{\DeltaThm}{\Delta_n}
\newcommand{\DeltaOracle}{\Delta_n^{*}}
\allowdisplaybreaks[4]

\begin{document}
\title[Statistical inference in sparse additive models]{Statistical inference in sparse high-dimensional additive models}
\author{Karl Gregory$^1$, Enno Mammen$^{2}$ and Martin Wahl$^{3}$}
\thanks{$^1$ Department of Statistics, University of South Carolina, 216 LeConte College, 1523 Greene St, Columbia, SC, 29201, USA}
\thanks{
$^2$ Institute of Applied Mathematics, Heidelberg University, Im Neuenheimer Feld 205, 
69120 Heidelberg, Germany}
\thanks{
$^3$ 
Institute of Mathematics, Humboldt-Universit\"at zu Berlin,
Unter den Linden 6,
10099 Berlin, Germany}

\date{}
\maketitle
\begin{abstract}
In this paper we discuss the estimation of a nonparametric component $f_1$ of a nonparametric additive model $Y=f_1(X_1) + ...+ f_q(X_q) + \epsilon$. We allow the number $q$ of additive components to grow to infinity and we make sparsity assumptions about the number of nonzero additive components. We compare this estimation problem with that of estimating $f_1$ in the oracle model $Z= f_1(X_1) + \epsilon$, for which the additive components $f_2,\dots,f_q$ are known. We construct a two-step presmoothing-and-resmoothing estimator of $f_1$ and state finite-sample bounds for the difference between our estimator and some smoothing estimators $\hat f_1^{\text{(oracle)}}$ in the oracle model. In an asymptotic setting these bounds can be used to show asymptotic equivalence of our estimator and the oracle estimators; the paper thus shows that, asymptotically, under strong enough sparsity conditions, knowledge of $f_2,\dots,f_q$ has no effect on estimation accuracy. Our first step is to estimate $f_1$ with an undersmoothed estimator based on near-orthogonal projections with a group Lasso bias correction. We then construct pseudo responses $\hat Y$ by evaluating a debiased modification of our undersmoothed estimator of $f_1$ at the design points. In the second step the smoothing method of the oracle estimator $\hat f_1^{\text{(oracle)}}$ is applied to a nonparametric regression problem with ``responses'' $\hat Y$ and covariates $X_1$. 
Our mathematical exposition centers primarily on establishing properties of the presmoothing estimator.  We present simulation results demonstrating close-to-oracle performance of our estimator in practical applications. 
\end{abstract}

\section{Introduction}

In this paper we study the estimation of an additive component in a high-dimensional sparse additive model. We compare this estimation problem with estimation in a nonparametric sub-model that contains only a single nonparametric component and we show that the two estimation problems are asymptotically equivalent. Our central argument is based on the construction of a class of two-step estimators that achieve the operating characteristics achieved by arbitrarily chosen smoothing estimators in the model with a single nonparametric component. We will prove finite-sample bounds for the difference between these two estimators. In an asymptotic framework, these bounds imply asymptotic equivalence of the two estimators under weak conditions. In addition to their theoretical value, these estimators are also of direct practical value, which we illustrate in simulations.

Our approach is analogous to that by which it is shown, in semiparametric modeling, that optimal estimation of a finite-dimensional parameter $\theta$ is asymptotically equivalent to optimal estimation in the hardest parametric sub-model containing only the parameter $\theta$.  This corresponds to our studying the estimation of an additive component $f_1$ in an additive model with additive components $f_1,\dots,f_q$ as compared to the estimation of $f_1$ in the classical nonparametric regression model in which $f_1$ is the sole component.  We refer to the latter model as the oracle model because estimation in this model is equivalent to estimation in the additive model when the functions $f_2,\dots,f_q$ are known. 

When we study estimation in semiparametric models, we typically have at our disposal an estimator for the  parametric sub-model which is asymptotically normal and unbiased and of which the asymptotic covariance matrix achieves a lower bound.  Thus, in order to establish the asymptotic efficiency of an estimator for the parametric component of a semiparametric model, it suffices to show that it is asymptotically normal and unbiased and that its asymptotic covariance matrix achieves the same lower bound as that achieved by the estimator in the  parametric sub-model.

In contrast, in nonparametric estimation, we typically do not have any single asymptotically optimal estimator for the sub-model containing only $f_1$.  This is because there are many different types of smoothing estimators, such as regression splines, kernel estimators, smoothing splines, and orthogonal series, which are not naturally comparable to one another and which have distinct asymptotic variances and biases, where the biases, moreover, are typically non-vanishing.   Thus, there is no benchmark optimal estimator of $f_1$ in the single-component sub-model to which we can compare estimators of $f_1$ in the additive model.

We circumvent this problem by showing that for every smoothing estimator $\hat f_1^{\text{(oracle)}}$ in the oracle model, there exists a corresponding estimator $\hat f_1$ in the additive model such that $\|\hat f_1 - \hat f_1^{\text{(oracle)}}\|_{\infty} $ is small compared to the typical values of $\hat f_1^{\text{(oracle)}} - f_1$ with high probability,  where $\|\cdot\|_{\infty}$ is the supremum norm. We call an estimator $\hat f_1$ with this property oracle adaptive to $\hat f_1^{\text{(oracle)}}$ or asymptotically equivalent to $\hat f_1^{\text{(oracle)}}$ if we compare the estimators in an asymptotic setting.  For this result we make some weak assumptions on $\hat f_1^{\text{(oracle)}}$ that hold for all classical smoothers. 
We prescribe a two-step construction of the estimator $\hat f_1$. In the first step all the components of the additive model are estimated with undersmoothing--that is with low bias and high variance--resulting in a pilot estimator $\hat f_1^{\text{(pre)}}$ of $f_1$ that is intentionally too wiggly. In the second step we apply the smoothing operation used in the calculation of $\hat f_1^{\text{(oracle)}}$ to the nonparametric regression problem where $\hat f_1^{\text{(pre)}}$ is regressed on the values of the first covariate $X_1^i$, $i=1,...,n$. The resulting resmoothed estimator $\hat f_1$ is our proposed estimator for $f_1$.
 
Our main result will state finite-sample properties for the presmoothing or pilot estimator $\hat f_1^{\text{(pre)}}$. This estimator is compared with a presmoothing estimator $\hat f_1^{\text{(oracle, pre)}}$ in the oracle model and finite-sample bounds are established. It is important that this is needed only for one specification of $\hat f_1^{\text{(pre)}}$ and $\hat f_1^{\text{(oracle, pre)}}$. We will argue that these finite-sample bounds imply that for a large class of smoothing operators (kernel smoothing, orthogonal series estimation, splines) resmoothing of $\hat f_1^{\text{(pre)}}$ by application of this smoothing method approximates well the corresponding smoothing estimator in the oracle model.
Thus, we get oracle adaptive estimators for a whole class of estimators $\hat f_1^{\text{(oracle)}}$. This can be judged as an asymptotic optimality theory for sparse high-dimensional additive models. 

The theoretical program has been carried out in \cite{HKM} for additive models with a fixed number of functions $q$. In this paper we will go far beyond this restriction and allow the total number of functions $q$ as well as the number $s_0$ of nonzero functions to grow with $n$, allowing also the case in which $q > n$. In \cite{HKM} the preliminary estimator was chosen as the first component  of a least-squares projection onto a sieve space of additive functions. Because of the high-dimensionality this approach cannot be applied here. We use a near orthogonal projection and apply a group Lasso estimator  for  correcting the bias due to the near orthogonality. In parametric high-dimensional models this type of estimator has also been called the desparsified or debiased Lasso. The Lasso estimator cannot be used directly because it is not comparable to an estimator in the oracle model. Note that for standard smoothers, the presmoothing  estimator $\hat f_1^{\text{(oracle, pre)}}$ will have a pointwise normal limit, whereas as for linear models the asymptotic distribution theory of the Lasso estimator is expected to be more complex. Thus $\hat f_1^{\text{(oracle, pre)}}$ cannot be approximated by a Lasso estimator in the additive model. 
 
We now formally express our estimation problem.  Let
$
Y=f(X)+\epsilon=\sum_{j=1}^qf_j(X_j)+\epsilon
$
with response $Y$ and covariates $X=(X_1,\dots,X_q)$ taking values in $[0,1]^q$. For identifiability, we assume that $\mathbf{E}[f_j(X_j)]= 0$ for $j=2,\dots,q$.
We assume that $\epsilon$ is a Gaussian random variable independent of $X$ with expectation $0$ and variance $\sigma^2$. Moreover, we assume that we observe $n$ independent copies $(Y^1,X^1),\dots,(Y^n,X^n)$ of $(Y,X)$, i.e.,
\begin{equation}
\label{eqn:model}
Y^i=  \sum_{j=1}^qf_j(X^i_j) +\epsilon^i,\ \ \ i=1,\dots,n.
\end{equation} 
We aim at estimating $f_1$ globally as well as locally at some point $x_0$. We compare the additive model \eqref{eqn:model} with the oracle model 
\begin{equation}
\label{eqn:ormodel}
Z^i=
 f_1(X^i_1) +\epsilon^i,\ \ \ i=1,\dots,n,
\end{equation} 
where one observes $X_1^i$ and $Z^i = Y^i-  \sum_{j=2}^qf_j(X^i_j) $. Here,  $X_1^i,\dots, X_q^i$  and $\epsilon_i$ are the same variables as in the additive model \eqref{eqn:model}. 
Note that the oracle model is equivalent to the submodel of the additive model where the additive nuisance components $f_2,\dots,f_q$ are known.

The discussion of additive models goes back to the influential work of Stone \cite{Sto}, who pointed out that additive nonparametric models efficiently circumvent the poor accuracy of high-dimensional regression functions and yet still provide high flexibility for statistical modeling. 
In recent years, estimation of nonparametric high-dimensional sparse additive models has been considered in a series of papers. Earlier references are \cite{LinZha}, \cite{AvaGranAmb}, \cite{Yuan}, \cite{HHW}, and \cite{RLLW}, where $L^1$-penalty based methods have been used for variable selection in additive models. For a related paper on model choice in nonparametric regression, see \cite{BerLec}. For sparse models in functional linear regression see \cite{LePa}, and for sparse models in varying coefficient models, see \cite{NoPa}. Rates of convergence for a fixed number of non-zero components have been discussed in \cite{LinZha} and \cite{HHW}. Rates of convergence for settings that allow for an increasing number of non-zero components were studied in \cite{MGB}, \cite{RWY}, \cite{SuSu}, and \cite{KolYua}. The latter paper also includes more general additive models where the summands are not necessarily functions of differing one-dimensional arguments. The paper \cite{Kat} proposes a two-step procedure in which variables are selected in a first step and a rate-optimal estimator is implemented in the second step. In \cite{FFS} sure independence screening is proposed for ultra-high dimensional additive models. 

All of these papers discuss only variable selection and/or optimal rates of convergence.  None of them presents any asymptotic distribution results for the proposed estimators, which severely restricts their range of statistical application.  In particular, there are no procedures in the current literature for the construction of valid confidence regions or tests of hypotheses in the high-dimensional sparse additive model. An asymptotic distribution theory for the Lasso estimator is complex because model choice is implicitly embedded in the construction of the estimator. For high-dimensional parametric models modifications for the Lasso estimator have been proposed that allow a complete asymptotic distribution theory. The method is to replace in the least-squares estimator each orthogonal projection of a covariate onto the other covariates with a projection of relaxed orthogonality (using the Lasso) and then to subtract an estimate of the resulting bias, which is constructed with Lasso estimates of the parameters.  The result is a non-sparse estimator, and it has been called the debiased Lasso for this reason. Influential discussions of this method include \cite{BellCherHan} and \cite{BellCher}, \cite{ZhaZha}, \cite{GeeBueRitDez}, and \cite{JavMont}. As said above, we will use this method in the nonparametric context of additive models. Debiasing has also been used in nonparametrics in \cite{LKL} for the discussion of undersmoothing estimators in additive models. Their estimator of a component $f_1$ is based on fits of the model $f_1(x_1) + f_2(x_2,x_1) + ...+ f_q(x_q,x_1)$ with $\mathbb{E}[f_k(X_k,x_1)]=0$ for all $x_1$. The paper \cite{MZ16} studies a high-dimensional linear regression model with group structure. The number and the size of groups is allowed to depend on $n$ which makes their model comparable to ours. However, they pursue a different goal (namely chi-squared type inference), while we are interested in the non-parametric problem of finding an estimator having comparable performance to an oracle estimator for which the other components are known. For this reason, the estimation procedures and the mathematical derivations differ considerably. The same is true of \cite{GeStu}, in which $\ell_2$ confidence sets for groups of variables are also constructed in high-dimensional regression and of \cite{StuGe}, in which confidence sets are proposed under general models of structured sparsity.   It should be remarked that as in most of the papers cited we do not make any so-called beta-min assumption. Thus, arbitrarily small additive components are allowed. Furthermore, in our model we assume that error variables are homoscedastic. Things change in the case of heteroscedastic errors, as has been pointed out in \cite{Efro2} for fixed~$q$.

The paper is organized as follows. In Section \ref{sec:genmod} we present a general result for the performance of presmoothers based on preliminary estimates of the nuisance components. In this section we do not assume that the nuisance component has an additive structure. We come to sparse additive models in Section \ref{sec:spaaddmod} where the result of Section  \ref{sec:genmod}  is applied to debiased Lasso estimators in sparse high-dimensional additive models. Section \ref{sec:resmooth} discusses the properties of various resmoothing estimators based on the presmoother. A simulation study is shown in Section \ref{sec:sim} and Section \ref{sec:math} outlines the proofs of the main results. Complete proofs and further simulation results are collected in the Supplementary Material. In addition, the Supplementary Material contains a detailed implementation of our two step procedure with Nadaraya-Watson smoothing, as well as a discussion of the construction of adaptive estimators using Lepski-type methods.

\section{Result for a generic presmoothing estimator} \label{sec:genmod}

In this section we express model \eqref{eqn:model} conditional on the design as
$$\mathbf{Y} = \mathbf{f}_1 + \mathbf{f}_{-1} + \boldsymbol{\epsilon}$$
where $\mathbf{Y}=(Y^1,\dots,Y^n)^T$ is an observation vector with values in $\mathbb{R}^n$, $\boldsymbol{\epsilon}=(\epsilon^1,\dots,\epsilon^n)^T$ is a random error vector
{with i.i.d. Gaussian elements with mean $0$ and variance $\sigma^2$ }and  $\mathbf{f}_1$ and $ \mathbf{f}_{-1}$ are unknown, fixed elements of $\mathbb{R}^n$. We assume that $\mathbf{f}_1$ and $ \mathbf{f}_{-1}$ can be approximated by elements of linear subspaces $V_1$ or $V_{-1}$ of $\mathbb{R}^n$, respectively. In this section we do not assume that $ \mathbf{f}_{-1}$ is the sum of additive components. This will be done from the next section on where we will chose these two spaces as $V_j=\{ (g(X^1),\dots,g(X^n))^T: g \in P_j\}$ with $j \in \{1,-1\}$ for some approximation spaces $ P_1$ and $P_{-1}$.

As in the following sections, the statistical aim is estimation of $\mathbf{f}_1$.  More specifically, we want to find an estimator  $\hat {\mathbf{f}}_1^{\text{(pre)}}$ in the model $V_1$ that approximates 
\[
\hat \f_1^{\operatorname{(oracle, pre)}}= \hat \Pi_1 (\mathbf{Y}-\mathbf{f}_{-1})
= \hat \Pi_1 (\mathbf{f}_1 + \boldsymbol{\epsilon}),
\]
where $\hat \Pi_1$ is the orthogonal projection from $\mathbb{R}^n$ onto $V_1$. Note that $\hat \f_1^{\operatorname{(oracle, pre)}}$ can be interpreted as an oracle estimator that makes use of knowing the true value of the nuisance part $\mathbf{f}_{-1}$ of the model. 
For the construction of  an approximation of $\hat \f_1^{\operatorname{(oracle, pre)}}$ we assume that an estimator 
$\hat {\mathbf{f}}_{-1}^{\text{(init)}}$ with values in $V_{-1}$
of the nuisance parameter $\mathbf{f}_{-1}$ can be constructed that has some properties that will be specified below. In this section we will choose $\hat {\mathbf{f}}_1^{\text{(pre)}}$ as
\begin{equation} \label{eq:defpre} \hat {\mathbf{f}}_1^{\text{(pre)}} =(I-A)^{-1} (\hat \Pi_1 -A) (\mathbf{Y} - \hat {\mathbf{f}}_{-1}^{\text{(init)}}) \end{equation} 
with $
A= (\hat \Pi_{-1}\hat \Pi_1 )^T$ (provided that $I-A$ is invertible, ensured by (A1) below), where $\hat \Pi_{-1}$ is a linear map from $V_1$ to $V_{-1}$. We will also denote the $n \times n$ matrices defined by the maps $\hat \Pi_{1}$ and $A$ as $\hat \Pi_{1}$ and $A$, respectively.

In the case that $\hat\Pi_{-1}$ is the orthogonal projection from $\mathbb{R}^n$ onto $V_{-1}$ we have that  $ (\hat \Pi_1 -A) \hat {\mathbf{f}}_{-1}^{\text{(init)}} \equiv 0$ and the estimator $\hat {\mathbf{f}}_1^{\text{(pre)}}$ is equal to the first summand of the least-squares estimator of the model $V_1+V_{-1}$. For this case equation  \eqref{eq:defpre} is also called  Frisch-Waugh-Lovell formula or Neyman orthogonalisation. If $\hat \Pi_{-1}$ differs from the orthogonal projection onto $V_{-1}$, then the estimator $(I-A)^{-1} (\hat \Pi_1 -A) \mathbf{Y}$ has a bias equal to $(I-A)^{-1} (\hat \Pi_1 -A) \mathbf{f}_{-1}$, which the estimator $\hat {\mathbf{f}}_1^{\text{(pre)}}$ from \eqref{eq:defpre} attempts to correct by subtracting the bias estimator $(I-A)^{-1} (\hat \Pi_1 -A) \hat {\mathbf{f}}_{-1}^{\text{(init)}}$. We may thus interpret
 $\hat {\mathbf{f}}_1^{\text{(pre)}}$ as a debiased estimator. If $\hat \Pi_{-1}$ only approximates the orthogonal projection from $\mathbb{R}^n$ onto $V_{-1}$, we may refer to this property as \textit{near-orthogonality}. A general discussion of near-orthogonality in semiparametrics and high-dimensional parametrics can be found in \cite{CheCheDemDuf,NinLiu}. For nonparametric models see also \cite{BelCheCheWei}. In high-dimensional sparse linear models  it has been proposed to estimate the nuisance variables by Lasso estimation. With this choice of $\hat {\mathbf{f}}_{-1}^{\text{(init)}}$ the estimator 
 $\hat {\mathbf{f}}_1^{\text{(pre)}}$ has also been called debiased Lasso estimator or with a different motivation  desparsified Lasso-estimator. 

For the comparison of $\hat {\mathbf{f}}_1^{\text{(pre)}}$ with $\hat \f_1^{\operatorname{(oracle, pre)}}$ we assume that for some constants $\Delta_1,\Delta_2,\Delta_3 > 0$,  $\rho_1 < 1$ and some function $\xi: V_{-1} \to \mathbb{R}^+$:
\begin{itemize}
\item [(A1)] $\|\hat \Pi_{-1}\mathbf{g}_{1}\|_{n,2} \leq \rho_1 \|\mathbf{g}_{1}\| _{n,2}$ for all $\mathbf{g}_{1}\in V_1$,
\item [(A2)]  $\|(\hat\Pi_1-\hat \Pi_{-1}\hat \Pi_1 )^T \mathbf{g}_{-1}\|_{n,\infty} \leq \Delta_1 \xi(\mathbf{g}_{-1}) $  for all $\mathbf{g}_{-1} \in V_{-1}$,
\item [(A3)] $\|\hat \Pi_{-1}\hat \Pi_1 \mathbf{e}_i\|_{n,2}\leq  \Delta_2/n$ for all $i=1,\dots,n$, where $\mathbf{e}_i$ is the $ith$ standard basis vector in $\mathbb{R}^n$,
\item[(A4)] $\|\hat\Pi_1 y\|_{n,\infty}\leq \Delta_3\|y\|_{n,\infty}$ for all $y\in\mathbb{R}^n$,
\end{itemize} 

Here, for a vector $x \in \mathbb{R}^n$ we define $\|x\|_{n,\infty} = \max_{1 \leq i \leq n} |x_i|$ and $\|x\|^2_{n,2} = n^{-1} \sum_{1 \leq i \leq n} x_i^2$.
 We have the following theorem:
\begin{thm} \label{mainthm} Make the Assumptions (A1)--(A4). Then, for all $\delta>0$, with probability $\geq 1 - \delta$,  
\begin{align*}
&\| \hat {\mathbf{f}}_1^{\operatorname{(pre)}} -  \hat \f_1^{\operatorname{(oracle, pre)}}\|_{n,\infty} \\
&\leq\Big(1+\frac{\Delta_2}{1-\rho_1}\Big)\bigg({\sigma}\sqrt{\frac{2\log(2n)+2\log(1/\delta)}{n}}
+\| \mathbf{f}_1 - \mathbf{g}_1 \|_{n,\infty}\nonumber\\
&\hspace{3cm}+(\Delta_2+\Delta_3)\| \mathbf{f}_{-1} - \mathbf{g}_{-1}\|_{n,\infty}+\Delta_1\xi(\hat {\mathbf{f}}^{\text{(init)}}_{-1} - \mathbf{g}_{-1})\bigg)\nonumber
\end{align*}
for all $\mathbf{g}_{-1} \in V_{-1}$ and $\mathbf{g}_{1} \in V_{1}$.
\end{thm}
Ignoring the factor $\Delta_2/(1-\rho_1)$, we see that the upper bound is given by the sum of a nearly parametric $n^{-1/2}$-rate, two bias terms, and a higher-order error term combining the near-orthogonality from (A2) with the performance of the estimator $\hat {\mathbf{f}}_{-1}^{\text{(init)}}$.

The theorem holds for arbitrary linear maps $\hat \Pi_1: \mathbb{R}^n \to V_1$ satisfying $\|\hat\Pi_1 y\|_{n,2}\leq \Delta^\prime _1\|y\|_{n,2}$ for all $y\in\mathbb{R}^n$, with  $\Delta^\prime _1=1$. For $\Delta^\prime _1 >1$ the theorem holds after a slight modification of the upper bound.

\section{A choice of presmoothing estimator}
\label{sec:spaaddmod}

In this section we will apply Theorem \ref{mainthm} to the estimation of a component in the additive model \eqref{eqn:model}. Without loss of generality we assume that the aim is to estimate the first additive component $f_1$, locally at a fixed point and globally, and that the other components $f_2$,...,$f_q$ are nuisance components. We will allow for very large values for the number $q$ of components but we will make a sparsity assumption  that a large unknown fraction of the additive components is equal to zero. For the application of  Theorem \ref{mainthm} we will choose the subspaces $V_1$ or $V_{-1}$ of $\mathbb{R}^n$  equal to $V_j=\{ (g(X^1),\dots,$ $ g(X^n))^T: g \in P_j\}$ with $j \in \{1,-1\}$ for linear subspaces $ P_j \subset \{ g:[0,1]^q\rightarrow \mathbb{R}: g(x) = g_j(x_j)$ for some function $g_j: [0,1] \to \mathbb{R}\} $ for $j=1,...,q$ and $ P_{-1} = P_2 \oplus ... \oplus P_q\subset \{ g:[0,1]^q\rightarrow \mathbb{R}: g(x) = g_2(x_2)+\dots+ g_q(x_q)$ for some functions $g_j: [0,1] \to \mathbb{R}$ for $j=2,...,q\} $ that will be defined in a moment in Subsection \ref{ssec:addmod} as spaces of piecewise polynomials.
We will choose $\hat\Pi_{-1}$ as a ``projection of relaxed orthogonality'' and we will select $\hat {\mathbf{f}}_{-1}^{\text{(init)}}$ as a group Lasso estimator of the nuisance parameter $\mathbf{f}_{-1}$.  We will make use of Theorem \ref{mainthm} to show that the resulting group debiased Lasso estimator $\hat {\mathbf{f}}_1^{\text{(pre)}}$ differs from the projection estimator in the oracle model by a uniform bound that is asymptotically negligible in an asymptotic framework. 

In what follows we will often consider the spaces $P_1$ and $P_{-1}$ as subspaces of $L^2(\mathbb{P}^X)$. The space $L^2(\mathbb{P}^X)$ is a Hilbert space with the inner product $\langle g, h\rangle=\mathbb{E}[g(X)h(X)]$ and the corresponding norm $\|g\|=\sqrt{\langle g, g\rangle}$. Let $\|\cdot\|_\infty$ denote the supremum norm on $L^\infty(\mathbb{P}^X)$. Let $\langle\cdot,\cdot\rangle_n$ denote the empirical inner product defined by
$
\langle g,h\rangle_n=n^{-1}\sum_{i=1}^n g(X^i)h(X^i)
$
and let $\|\cdot\|_{n}=\|\cdot\|_{n,2}$ denote the corresponding empirical norm. For a function $g:[0,1]^q\rightarrow \mathbb{R}$,  we abbreviate $\g = (g(X^1) , \dots,g(X^n))^T$.


\subsection{A choice of the initial estimator}\label{ssec:initest}
We choose  $\hat \f_{-1}^{\text{(init)}}$ as a nonparametric group Lasso estimator defined as follows. For some tuning parameter $\lambda > 0$, let
\[
(\hat{f}_1^L,\dots,\hat{f}_q^L)\in\operatorname{argmin}_{g_j\in P_j}\limits\bigg\{\frac{1}{n}\sum_{i=1}^n\bigg[Y^i - \sum_{j=1}^q g_j(X_j^i)\bigg]^2+2 \lambda \sum_{j=1}^q\|g_j\|_n\bigg\}.
\]
Setting $\hat f_{-1}^L=\sum_{j=2}^q\hat f_j^L$, we choose $\hat \f_{-1}^{\text{(init)}}$ as $\hat \f^{L}_{-1} = (\hat f_{-1}^L(X^1) , \dots,\hat f_{-1}^L(X^n))^T$.

\subsection{A choice of sieve function spaces} \label{ssec:addmod}
We now introduce our choices of the spaces $P_1$ and $P_{-1}$ as spaces of piecewise polynomials. We conjecture that our theory would go through with alternative spaces based on a localized basis but keep to piecewise polynomials for simplicity.
 
For $j=1,\dots,q$, let $t_j\geq 0$ and $m_j\geq 1$ be integers and let $B_j$ be the space of piecewise polynomials in the variable $x_j\in [0,1]$ of maximal degree $t_j$ defined on the intervals $ I_{jk}=(k/m_j,(k+1)/m_j]$, $k=0,\dots,m_j-1$. Thus each function $g_j\in B_j$ has the property that, restricted to each interval $I_{jk}$, it is a polynomial of degree at most $t_j$. Let $Q_l$, $l\geq 0$ be the sequence of the Legendre polynomials (see, e.g., the book by Whittaker and Watson \cite{WW} for the definition and fundamental properties of the Legendre polynomials). Then the shifted and rescaled polynomials $R_l(x)=\sqrt{2l+1} Q_l(2x-1)$, $x\in [0,1]$, are orthonormal with respect to the inner product induced by the Lebesgue measure on $[0,1]$. For $k=0,\dots,m_j-1$ and $l=1,\dots,t_j+1$, we now define 
\begin{equation*}
b_{j,k(t_j+1)+l}(x_j)=\sqrt{m_j}R_{l-1}\left( m_j\left( x_j-\frac{k}{m_j}\right) \right)
\end{equation*}
for $x\in I_{jk}$ (and equal to zero otherwise). Hence
$b_{j,k(t_j+1)+1},\dots,b_{j,k(t_j+1)+t_j+1}
$
is an orthonormal basis of the functions in $B_j$ which are zero outside the interval $I_{jk}$, and we conclude that 
$
 b_{j,1},\dots,b_{j,m_j(t_{j}+1)}
$
is an orthonormal basis of $B_j$ with respect to the Lebesgue measure.  

In the following we suppose that $m_2=\dots=m_q$ and that $t_2=\dots=t_q$,
that is, we suppose that $B_2,\dots,B_q$ are defined with piecewise polynomials of the same order and {on the same intervals}. We let $m=\max(m_1,m_2)$ and $t=\max(t_1,t_2)$. Moreover, let 
\[
P_1=B_1\quad\text{and}\quad P_j=\left\lbrace  g_j\in B_j:\mathbb{E}\left[ g_j(X_j)\right]=0 \right\rbrace,\quad j=2,\dots,q.
\]
Note that in practice, the spaces $P_j^n=\lbrace  g_j\in B_j:n^{-1}\sum_{i=1}^ng_j(X_j^i)=0\rbrace$ will be used instead of $P_j$, which is achieved by centering each basis function $b_{jk}$ by its empirical mean. However, we choose, in our analysis, to proceed using the spaces $P_j$ instead of $P_j^n$ in order to avoid cumbersome technicalities. 

In what follows we will regard the spaces $P_1,\dots,P_q$ as subspaces of $L^2(\mathbb{P}^X)\cap L^\infty(\mathbb{P}^X)$.
We let $d_j=\dim P_j$ and $d=\max_j d_j$.  Hence (under Assumption (B1) below), we have $d_1=m_1(t_1+1)$ and $d_2=\dots=d_p=m_2(t_2+1)-1$. We abbreviate $P_{-1}=\sum_{j=2}^qP_j$ for the space of additive functions with components coming from $P_2,\dots,P_q$, and let $\Pi_{-1}:L^2(\mathbb{P}^X) \rightarrow P_{-1}$ be the orthogonal projection from $L^2(\mathbb{P}^X)$ to $P_{-1}$ given by 
\[
\Pi_{-1}h=\operatorname{argmin}_{g\in P_{-1}}\limits\|h-g\|^2.
\]
Similarly, for $J\subseteq \{1,\dots,q\}$, we abbreviate $P_J=\sum_{j\in J}P_j$, and let $\Pi_J:L^2(\mathbb{P}^X) \rightarrow P_J$ be the orthogonal projection from $L^2(\mathbb{P}^X)$ to $P_J$.

\subsection{A projection of relaxed orthogonality}\label{ssec:relorth}
We choose for the projection $\hat \Pi_{-1}$ a Lasso-based projection defined as follows. For $k=1,\dots,d_1$, we define the nonparametric group Lasso projection 
$
\hat{\Pi}_{-1}^L b_{1k}=\sum_{j=2}^q(\hat{\Pi}_{-1}^L b_{1k})_j\in P_{-1}
$
of $\Pi_{-1}b_{1k}$ by
\begin{equation*}
\Big((\hat{\Pi}_{-1}^L b_{1k})_2,\dots,(\hat{\Pi}_{-1}^L b_{1k})_q\Big)  \in\operatorname{argmin}_{g_j\in P_j}\limits\bigg\{\Big\|b_{1k}-\sum_{j=2}^qg_j\Big\|_n^2+2\eta\sum_{j=2}^q\left\|g_{j}\right\|_n\bigg\},
\end{equation*}
where $\eta>0$ is a tuning parameter. We extend $\hat{\Pi}_{-1}^L$ linearly to all of $P_1$ as follows:
\[
\hat{\Pi}_{-1}^L: P_1\rightarrow P_{-1}, \quad \sum_{k=1}^{d_1}\alpha_k b_{1k}\mapsto \sum_{k=1}^{d_1}\alpha_k\hat{\Pi}_{-1}^L b_{1k},
\]
which can be seen as an empirical version of $\Pi_{-1}$ restricted to $P_1$. Note that $\hat{\Pi}_{-1}^L$ can also be defined as a map from $V_1$ to $V_{-1}$, provided that the $\mathbf{b}_{1k}$ are linearly independent. The latter holds on the empirical norm approximation event $\EventNu$ defined in Section \ref{SecEvents}.

\subsection{Assumptions for the choice of presmoothing estimator} We here give the assumptions needed in order that the choices of $\hat \f_{-1}^{\operatorname{(init)}}$ as $\hat \f_{-1}^L$, of $\hat \Pi_{-1}$ as $\hat \Pi_{-1}^L$, and of $V_j$ as $\{\g: g \in P_j\}$ with $j \in \{1,-1\}$ will possess the properties (A1)--(A4) for some sequences $\Delta_1,\Delta_2,\Delta_3$ and a constant $\rho_1$.

\begin{itemize}
\item[(B1)]
For $j=1,\dots,q$, $X_j$ takes values in $[0,1]$ and has a density $p_j$ with respect to the Lebesgue measure on $[0,1]$ which satisfies $c_1\leq p_j\leq 1/c_1$ for some constant $c_1>0$. Moreover, for $j=2,\dots,q$, $(X_1,X_j)$ has a density $p_{1j}$ with respect to the Lebesgue measure on $[0,1]^2$ which is bounded from above by $1/c_1$.\end{itemize}

\begin{itemize}
\item[(B2)]
There is a constant $0\leq \rho_0<1$ such that $\|\Pi_{-1}g_1\|\leq \rho_0\|g_1\|$ for all $g_1\in P_1$.
\end{itemize}

\begin{itemize}
\item[(B3)]
There exist some $r_1,r_2 > 0$ and a subset $J_0\subseteq \lbrace 1,\dots,q\rbrace$ with $1\in J_0$ and $|J_0|\leq s_0$ such that for each $j\in J_0$ there is a $g_j^*\in P_j$ satisfying
\[
\|f_1-g^*_1\|_\infty\leq C_0d_1^{-r_1}\quad\text{and}\quad
\|f_j-g^*_j\|_\infty\leq C_0d_{2}^{-r_2},\quad j=2,\dots,q,
\]
for some constant $C_0>0$. Moreover, for $g^*=\sum_{j\in J_0}g_j^*$, we have
\[
\|f-g^*\|_\infty\leq C_0( d_1^{-r_1}+s_0d_2^{-r_2}).
\]
\end{itemize}

\begin{itemize}
\item[(B4)]
	For each $k=1,\dots,d_1$, there is a subset $J_k\subseteq \lbrace 2,\dots,q\rbrace$ with $|J_k|\leq s_1$, such that there is a decomposition
$
\Pi_{J_k}b_{1k}-\Pi_{-1}b_{1k}=\sum_{j=2}^qp_j
$
with $p_j \in P_j$ satisfying
\[
\sum_{j=2}^q\|p_j\|\leq C_1\sqrt{\frac{s_1d}{n}}
\]
for some constant $C_1>0$. Moreover, $d\leq n$.
\end{itemize}

\begin{itemize}
\item[(B5)] There is a real number $0<\phi\leq 1$ such that
\[
\sum_{j\in J_0}\|g_j\|^2\leq\|\sum_{j=1}^q g_j\|^2/\phi^2
\]
for all $(g_1,\dots,g_q)\in (P_1,\dots,P_q)$ with 
$
\sum_{j=1}^q\|g_j\|\leq 8\sqrt{3}\sum_{j\in J_0}\|g_j\|
$.
\end{itemize}
\begin{itemize}
\item[(B6)]
	For $k=1,\dots,d_1$,
\[
\sum_{j\in J_k}\|g_j\|^2\leq\|\sum_{j=2}^q g_j\|^2/\phi^2
\]
for all $(g_2,\dots,g_q)\in (P_2,\dots,P_q)$ with 
$
\sum_{j=2}^q\|g_j\|\leq 8\sqrt{3}\sum_{j\in J_k}\|g_j\|$.
\end{itemize}
Finally, we denote by $0< \psi\leq 1$  the largest number such that $ \sum_{j\in J_k} \|g_j\|^2\leq \| \sum_{j\in J_k}g_j \|^2 \big/\psi^2$ for all $g_j\in P_j$, $j\in J_k$ and all $k=1,\dots,d_1$. By Assumption (B6),  we know that $\psi\geq \phi>0$. 

Assumption (B2) introduces a geometric quantity $\rho_0$ which governs the degree of collinearity between the spaces $P_1$ and $P_{-1}$. The closer $\rho_0$ is to $1$, the harder it is to distinguish the effects of $X_1$ from those of $X_2,\dots,X_q$.  Note that $\rho_0$ equals the cosine of the minimal angle between $P_1$ and $P_{-1}$ and that (B2) is implied by standard boundedness conditions on joint densities; see e.g. \cite{Wahl} for more discussion. Moreover, $\rho_0$ can be upper-bounded by the cosine of the minimal angle between $L^2(\mathbb{P}^{X_1})$ and $\{g\in L^2(\mathbb{P}^{X_{-1}}):\mathbb{E}[g(X_{-1})]=0\}$, yielding a stronger assumption on the joint distribution of the covariates irrespective of the choice of the approximation spaces $P_1$ and $P_{-1}$.  Assumption (B3) imposes an approximate sparsity assumption on our regression function $f$. 
Assumption (B4) assumes that the projection of each basis function $b_{1k}$ onto the space $P_{-1}$ may be approximated sufficiently well by its projection onto a subspace of $P_{-1}$ of $s_1$ or fewer additive components. It corresponds to the sparsity assumption on the precision matrix for the high-dimensional linear model; see e.g. \cite{ZhaZha,GeeBueRitDez}. The following lemma shows that  (B4) holds under a mixing type condition on the covariables. \begin{lemma} \label{lemB4check}
Suppose that (B1) applies and that after some permutation of the indices $2,...,q$, it holds that $|p_{il}(x_i,x_l) - p_i(x_i) p_l(x_l) | \leq c^* \rho^{|i-l|}$ for $1 \leq i,l \leq q$, with $0 < \rho < 1$ and a constant $c^*>0$ that is small enough. Then (B4) holds with $s_1= C \log n$, $C>0$ a constant. 
\end{lemma}
The proof of the lemma is given in Section \ref{sectdiscussB4}  in the Supplementary Material. We conjecture that the lemma also holds for arbitrary $c^*>0$ and that it also remains valid  under weaker mixing conditions for choices of $s_1$ that are large enough.

 Assumptions (B5) and (B6) are nonparametric versions of the theoretical compatibility condition introduced in \cite{MGB}. They are implied by a stronger assumption on the joint distribution of the covariates by replacing the approximation spaces by their corresponding (centered) $L^2$ spaces. The (weaker) constant $\psi$ will be used in the analysis of the norms of the $\Pi_{-1}b_{1k}$.

\subsection{A bound for the choice of presmoothing estimator}\label{boundsforthepresmoothingestimator}

For $a\geq 1$, we choose the tuning parameters as
\begin{align}\label{EqChoiceLambdaEta}
\lambda&=2\sigma\sqrt{\frac{d}{n}}+2\sigma\sqrt{\frac{2L}{n}},\qquad\eta=C\bigg( \sqrt{\frac{Ld}{n}}+\frac{\sqrt{s_1}Ld}{\psi n} \bigg),
\end{align}
where $L=a\log n+\log q$. We make the convention that $C$ denotes a constant that is chosen large enough depending only on the quantities $t$, $c_1$, $C_0$, and $C_1$ and that $C$ is not necessarily the same at each occurrence.  The choice of the constant $C$ in the definition of $\eta$ can be found in Appendix \ref{AppendixE} in the Supplementary Material.
We now state our main result.
\begin{thm} \label{ourchoicesthm}
Suppose that Assumptions (B1)--(B6) are satisfied. For $a\geq 1$, choose $\hat \f_{-1}^{\operatorname{(init)}}$ as $\hat \f_{-1}^L$ and  $\hat \Pi_{-1}$ as $\hat \Pi_{-1}^L$ with $\lambda$ and $\eta$ as in \eqref{EqChoiceLambdaEta} and  $V_j=\{ \mathbf{g}: g \in P_j\}$, $j \in \{1,-1\}$. Then there are constants $c,C>0$ such that the following holds.  If
\begin{align} \label{EqThmLCond}
s_1\sqrt{\frac{d(d_1+ s_1)L}{n}}&\leq c(1-\rho_0)^2\psi^2\phi, \qquad(s_0+s_1)\sqrt{\frac{d (d+L)}{n}} \leq c\phi^2,
\end{align} 
where $L=a\log n+\log q$, then with probability $\geq 1 - 2 n^{-a}$ we have 
\[
\|\hat \f_1^{\operatorname{(pre)}} - \hat \f_1^{\operatorname{(oracle,pre)}}\|_{n,\infty} \leq  \DeltaThm,
\] 
where 
\begin{align} \label{eq:Delta_n^*}
 \DeltaThm&= \frac{C}{(1-\rho_0)\psi\phi^2}\bigg[{\sigma}\sqrt{\frac{s_1a\log n }{n}}+s_1d_1^{-r_1} + s_1s_0 d_2^{-r_2} \\ \nonumber
& \qquad +   \sigma^{-1}\sqrt{d_1s_1L}(d_1^{-2r_1} + s_0^2 d_2^{-2r_2})+ \frac{{\sigma}s_0 \sqrt{s_1 d_1d  (d +L)L}}{n}\bigg]\nonumber.
\end{align}
\end{thm}

The upper bound $ \DeltaThm$ in the statement of the theorem shows five terms. The first term  is nearly of parametric $\sqrt{n}$-order, at least if $s_1$ are not too large. The next two terms correspond to bias terms (see Assumption (B3)) and the last two terms reflect quadratic expressions in the expansions showing up in the proof of Theorem \ref{ourchoicesthm}.  

In the next section we will use the estimator $\hat \f_1^{\operatorname{(pre)}}$ as pilot estimator in a two step procedure where a large class of resmoothing operations can be applied in the second step. We will show that up to a negligible error one can achieve the same accuracy as that achieved by a large class of smoothers in the oracle model. For a proof of this claim it will be essential that for the pilot estimator $\hat \f_1^{\operatorname{(pre)}}$
the error $\|\hat \f_1^{\operatorname{(pre)}} - \hat \f_1^{\operatorname{(oracle,pre)}}\|_{n,\infty}$ is negligibly small. 

We remark again that because $\hat \f_1^{\operatorname{(oracle,pre)}}$ has a pointwise limiting normal distribution this can only be achieved by pilot estimators that also have a pointwise limiting normal distribution. This means we cannot simply use the group Lasso as the presmoothing estimator.


Theorem \ref{ourchoicesthm} is also of importance when applied directly to the estimator $\hat \f_1^{\operatorname{(pre)}}$. If $d_1$ is chosen large enough the estimator $ \hat \f_1^{\operatorname{(oracle,pre)}}$ is an undersmoothing estimator with bias of smaller order than its standard deviation. Hence, Theorem \ref{ourchoicesthm} implies that $\hat \f_1^{\operatorname{(pre)}}$ has pointwise the same mean zero normal limit as $ \hat \f_1^{\operatorname{(oracle,pre)}}$. The asymptotics of $ \hat \f_1^{\operatorname{(oracle,pre)}}$ are well understood and thus by our result we get that many asymptotic results on statistical methods for inference on the function $f_1$ carry over to statistical methods on $f_1$ in the additive model based on $\hat \f_1^{\operatorname{(pre)}}$. This includes pointwise and uniform confidence intervals for $f_1$ and testing procedures concerning the shape of $f_1$.

\subsection{An alternative choice of projection of relaxed orthogonality}
We finish Section \ref{sec:spaaddmod} by discussing another, computationally less feasible, choice for $\hat \Pi_{-1}$ that is directly based on (A1)--(A3). In order to find a linear map $\hat\Pi_{-1}:V_1\rightarrow V_{-1}$ satisfying (A1)--(A3), a simple construction consists in minimizing (A1) under the constraint that (A2) and (A3) are satisfied. As a slight modification, let $\hat \Pi_{-1}^M=M:V_1\rightarrow V_{-1}$ be a linear map that is the solution of 
\begin{align}
&\text{minimize}\quad \|M\hat\Pi_1\|_{\operatorname{op}}\label{EqConstrJMInfNorm}\\
&\text{subject to}\quad \text{(i)}
\quad\sup_{\|g_j\|_n\leq 1}\|(\hat \Pi_1-M\hat \Pi_1 )^T\mathbf{g}_j\|_{n,\infty} \leq \Delta_1',\quad j=2,\dots,q,\nonumber\\
&\hspace{2cm}\text{(ii)}\quad\|M\mathbf{b}_{1k}\|_{n}\leq \Delta_2',\quad k=1,\dots,d_1\nonumber. 
\end{align}
Here, for all linear maps on $\mathbb{R}^n$, $\|\cdot\|_{\operatorname{op}}$ denotes the operator norm with respect to the empirical inner product. 
This procedure can be further simplified by replacing in (i) the supremum norm by the empirical norm:
\[
\hspace{1,4cm}\text{(i')}
\quad \|(\hat \Pi_1-M\hat \Pi_1 )^T\hat\Pi_j\|_{\operatorname{op}} \leq \Delta_1'',\quad j=2,\dots,q,
\]
with the relation $\Delta_1'=C\sqrt{d_1}\Delta_1''$, cf. \eqref{eq:asd} below. A similar minimization scheme has been studied Mitra and Zhang \cite{MZ16} in the high-dimensional linear regression model with group structure. In fact, \eqref{EqConstrJMInfNorm} can be seen as an extension of Equation (2.33) of Mitra and Zhang \cite{MZ16} to the non-parametric setting. A main difference is Condition (ii). It reflects a smoothing property when the true orthogonal projection $\Pi_{-1}$ is applied to a localized basis. This property is crucial for the asymptotic equivalence results described in the introduction.

If one is willing to make additional assumptions on the joint distribution of the covariates one can construct projections based on density estimates; for a recent contribution see \cite{GuoZha}, who use parametric dependence models,
 and for additive models with a fixed number of additive components see also \cite{HenSpe,KimLinHen}, who use higher order smoothness assumptions on joint densities.

One can show that Theorem \ref{ourchoicesthm} also holds with $\hat\Pi_{-1}^L$ replaced by $\hat\Pi_{-1}^M$ from \eqref{EqConstrJMInfNorm}, provided that $\Delta_1'$ and $\Delta_2'$ are chosen appropriately. Moreover, there is another version of Theorem \ref{ourchoicesthm} with slightly weaker restrictions on $s_1$. For this, we replace (B4) by  
\begin{itemize} 
\item[(B4')]
Suppose that there is a subset $J_1\subseteq \lbrace 2,\dots,q\rbrace$ with $|J_1|\leq s_1$, such that $\|(\Pi_{-1}-\Pi_{J_1})\Pi_1\|_{\operatorname{op}}\leq (1-\rho_0)/4$ and for all $k=1,\dots,d_1$,
\[
 \|\Pi_{-1}b_{1k}-\Pi_{J_1}b_{1k}\|\leq C_1\sqrt{\frac{d}{n}}.
\]
\end{itemize}
\begin{thm} \label{altchoicesthm}
Suppose that Assumptions (B1)-(B3),(B4)',(B5) are satisfied. For $a\geq 1$, choose $\hat \f_{-1}^{\operatorname{(init)}}$ as $\hat \f_{-1}^{L}$ with $\lambda$ from \eqref{EqChoiceLambdaEta},  $\hat \Pi_{-1}^M$ as a solution of \eqref{EqConstrJMInfNorm} with $\Delta_1'=C\sqrt{d_1}\eta$ and $\Delta_2'=(C/\psi)\sqrt{s_1/d_1}$, with $\eta$ from \eqref{EqChoiceLambdaEta} and $C$ sufficiently large, and  $V_j=\{ \mathbf{f}: f \in P_j\}$, $j \in \{1,-1\}$. Then there are constants $c,C>0$ such that the following holds. If
\begin{equation} \label{eq:ThmCond}
\sqrt{\frac{s_1d(a\log n+\log q)}{n}}\leq c\psi(1-\rho_0),\quad s_0\sqrt{\frac{d (d+ a\log n+\log q)}{n}} \leq c\phi^2,
\end{equation} 
then, with probability $\geq 1 -  2n^{-a}$, we have 
$
\|\hat \f_1^{\operatorname{(pre)}} - \hat \f_1^{\operatorname{(oracle,pre)}}\|_{n,\infty} \leq  \DeltaThm,
$
with $\DeltaThm$ from \eqref{eq:Delta_n^*}.
\end{thm}

\section{The resmoothing estimator}
\label{sec:resmooth}

We now consider the resmoothing step discussed in the introduction which makes use of the presmoothed data.  Under the conditions of Theorem \ref{ourchoicesthm}, with probability greater than or equal to $1 - 2 n^{-a}$, we have
\begin{equation}
\label{eqn:smoothadd}
\|\hat \f_1^{\text{(pre)}} -\hat \f_1^{\text{(oracle,pre)}}\|_{n,\infty} \leq \DeltaThm,
\end{equation}
where $\DeltaThm$ is defined in \eqref{eq:Delta_n^*}. We now consider several classes of estimators $\hat f_1^{\text{(oracle)}}$ for the oracle model. 
For a discussion of the resmoothing idea let us write the estimator as $ \hat {{f}}_1^{\text{(oracle)}} ={\mathcal S} \mathbf{Z}$, where $\mathbf{Z}=(Z_1,...,Z_n)^T$ are the responses in the oracle model \eqref{eqn:ormodel} and ${\mathcal S}$ is the applied smoothing operator that maps an element of  $\mathbb{R}^n$ to a real-valued function.
In this section we will discuss the performance of the resmoothing estimator $ \hat f_1$ defined as
   $ \hat {{f}}_1 ={\mathcal S}\hat {\mathbf{f}}_1^{\text{(pre)}}$. It is  the estimator obtained from applying the smoothing operation $\mathcal S$ of $\hat f_1^{\text{(oracle)}}$ to the regression problem with covariate values $X_1^i$ and ``response'' values $\hat f_1^{\text{(pre)}}(X_1^i)$. Similarly, in the oracle model $ {\mathcal S} \hat {\mathbf{f}}_1^{\text{(oracle,pre)}}$ is defined as the estimator resulting from the smoothing operation $\mathcal{S}$  applied to the regression problem with covariate values $X_1^i$ and response values $\hat f_1^{\text{(oracle, pre)}}(X_1^i)$. Our main assumption on $\hat f_1^{\text{(oracle)}}$ is that 
\begin{equation}
\label{eqn:smoothundersmooth}
\|{\mathcal S} \hat  {\mathbf{f}}_1^{\text{(oracle,pre)}}-\hat  {{f}}_1^{\text{(oracle)}}\|_{\infty} \leq  \DeltaOracle
\end{equation} 
for some small values $\DeltaOracle$ with high probability. This is a natural assumption that is valid for many smoothing estimators. It says that a smoothing operation applied after an undersmoothing of the data is approximately  equivalent to a single application of the smoothing. For our next argument we need that the smoothing operation $\mathcal S$ has the following continuity property for all $\delta > 0$ and a constant $C_*>0$:
\begin{eqnarray}
\nonumber
&&\text{A change in the responses by a maximal amount less than}\\ \nonumber &&\text{$\delta$ does not lead to a change larger than $C_*\delta$ in the resulting}\\ 
&&\text{smoother.}\label{eqn:smoothcontin}
\end{eqnarray} 
Note that \eqref{eqn:smoothcontin} refers to the smoother's evaluations across all points. We also call this property of the smoothing operation small error robustness.
From \eqref{eqn:smoothadd}--\eqref{eqn:smoothcontin}
 we get that 
\begin{eqnarray*}
&&\|\hat f_1 - {\hat f}_1^{\text{(oracle)}}\|
 _{\infty} \leq 
  \| {\mathcal S} \hat {\mathbf{f}}_1^{\text{(pre)}} - {\mathcal S} \hat {\mathbf{f}}_1^{\text{(oracle,pre)}}\| _{\infty} \\
 && \qquad \qquad +
  \|{\mathcal S} \hat {\mathbf{f}}_1^{\text{(oracle,pre)}}-\hat f_1^{\text{(oracle)}}\|_{\infty}
 \leq  C_* \DeltaThm +  \DeltaOracle
\end{eqnarray*} 
with high probability. If $C_* \DeltaThm +  \DeltaOracle$ is small compared to the typical values of $\hat f_1^{\text{(oracle)}} - f_1$ we get that $\hat f_1 $ has approximately the same statistical performance as $\hat f_1^{\text{(oracle)} }$. 

The resmoothing idea has also been used in \cite{HKM} for additive models with a fixed number of additive components. See also the discussion in the introduction. The main contribution of this paper lies in the construction of a preliminary estimator for sparse additive models with an increasing number of additive components.  Compared to the preliminary estimator in \cite{HKM} the structure of this estimator is rather complex. It consists of a nonlinear near-orthogonal projection  followed by a debiasing step using group Lasso.  In  \cite{HKM} the preliminary estimator was a backfitting estimator with clear and well-understood asymptotic properties. 
Having obtained our bound in Theorem \ref{ourchoicesthm}, we can proceed  similarly as in \cite{HKM}. The paper  \cite{HKM} also contains results on the asymptotic equivalence of  orthogonal series estimators and smoothing splines with respect to the $L^2$ norm. In this section we concentrate on $L^\infty$ asymptotic equivalence because this has more important implications for the statistical applications of the estimators. We discuss least-squares splines, piecewise polynomial fits, and local polynomials. We use these classes of estimators to show that our bounds in Theorem \ref{ourchoicesthm} and Theorem \ref{altchoicesthm} are sharp enough for important applications of the estimators. We start by discussing this program for least-squares splines and piecewise polynomial fits, respectively. For these classes we have the following result, for which a proof is given in Appendix \ref{AppendixD} of the Supplementary Material.

\begin{thm}\label{theo:splinepoly} Suppose that the assumptions of Theorem \ref{ourchoicesthm} or Theorem \ref{altchoicesthm} hold and let $\hat f_1^{\operatorname{(pol)}}$ and $\hat f_1^{\operatorname{(spl)}}$ be two-step estimators for which least-squares polynomial and B-spline fitting, respectively, with an equidistant partition of $m^*$ intervals has been used in the second step.  Furthermore, suppose that the order of the polynomials or splines used in the second step is the same as that of the polynomials used in the first step and suppose that the number of intervals $m_1$ used in the first step is a multiple of $m^*$. Then  
\begin{eqnarray*}\|\hat f_1^{\operatorname{(pol)}}- \hat f_1^{\operatorname{(oracle,pol)}}\|_\infty &\leq& C \DeltaThm, \\
\|\hat f_1^{\operatorname{(spl)}}- \hat f_1^{\operatorname{(oracle,spl)}}\|_\infty& \leq& C \DeltaThm,
\end{eqnarray*}
with probability greater than or equal to $ 1 - 2 n^{-a}$. Here $\hat f_1^{\operatorname{(oracle,pol)}}$ and $\hat f_1^{\operatorname{(oracle,spl)}}$ are the one-step estimators in the oracle model based on least-squares polynomial and B-spline fitting, respectively.
\end{thm}

The assumption that the grid used in the first step is based on subdividing the grid used in the second step greatly simplifies the proof. But by using more refined arguments, it can be shown that this assumption is not necessary. 

For an asymptotic interpretation let us assume that 
\begin{eqnarray} \label{ratesplineA1} &&\log \log q = o(\log n),\quad r_1 > \frac 1 2,  \quad s_0 = O(n ^{\gamma_0}), \quad s_1 = O(n ^{\gamma_1})\end{eqnarray}
for some constants  $0 \leq \gamma_0 < 1/2$ and $ 0 \leq \gamma_1\leq 1/4$. 
Suppose that we want to check whether an estimator in the additive model exists that is asymptotically equivalent to a rate-optimal spline or polynomial estimator in the oracle model.  For rate-optimal estimation in the oracle model, the number of intervals  should be a constant times $n^{1/(2r_1 +1)}$, which results in a pointwise rate of $n^{-\beta}$ with $\beta ={r_1/(2r_1 +1)}$. Thus we have to check if $m_1$ and $m_2$ can be chosen such that
\begin{equation} \label{ratesplineC} \DeltaThm = o(n^{-\beta})\end{equation} 
holds for this choice of $\beta$, $m_1$ and $m_2$. In Appendix \ref{AppendixD} in the Supplementary Material we will show that this is the case if $\gamma_0,\gamma_1 \geq 0$ are small enough and 
$r_2 >2  r_1 /(2r_1 +1)$. In particular, we see that values of $r_2$ such that $r_2 < r_1$ are allowed. Thus we do not require the nuisance additive components $f_2,...,f_q$ to be as smooth as $f_1$. For a short discussion of the case $0 < \gamma_0 < 1/2$ see also Appendix \ref{AppendixD} in the Supplementary Material.
\par 
We now discuss local polynomial estimators. The degree of the local polynomial estimator is denoted by $k$. Define $(\tilde{a}_0,...,\tilde{a}_{k})$ as the minimum of 
$$\sum_{i=1} ^n \left [ Z^i  - a_0 - ...- a_{k} (X_1^i-x)^{k}\right ]^2 K_h(X_1^i-x)$$
over $(a_0 , ..., a_{k}) \in \mathbb{R}^{k+1}$
and set $\hat f_1 ^{j,\text{(oracle,lpol)}}(x) = \tilde{a}_j$. This is an estimator of the $j$-th derivative of $f_1$ in the oracle model. Here, $K_h(u) = h^{-1} K(h^{-1} u)$ is a kernel with kernel function $K$ and bandwidth $h$. Similarly, we define $\hat f_1 ^{j,\text{(lpol)}}(x) = \hat{a}_j$, where now $(\hat{a}_0,...,\hat{a}_{k})$ minimizes 
\[
\sum_{i=1} ^n \left [ \hat Y^i - a_0 - ...- a_{k} (X_1^i-x)^{k}\right ]^2 K_h(X_1^i-x),
\]
where $\hat Y^i$ is the $i$th element of $\hat \f_1^{\operatorname{(pre)}}$. For this class of estimators we have the following result.

\begin{thm}\label{theo:locpoly} Suppose that the assumptions of Theorem \ref{ourchoicesthm} or Theorem \ref{altchoicesthm} hold. Suppose further that the kernel $K$ is a probability density function with bounded support, $[-1,1]$ say, that it has an absolutely bounded derivative and that the bandwidth $h$ fulfills $c_- n^{-\eta_1} \leq h \leq c_+ n^{-\eta_2}$ for some $c_-, c_+ > 0$ and $ 0 < \eta_2 \leq \eta_1 < 1/3$. Furthermore, assume that for a value $\rho^* \leq k+1$ the function $f$ has an absolutely bounded derivative of order $\rho^*$.  Then it holds for $j=0, \dots, k$ 
that 
\begin{align*}h^j \|\hat f_1 ^{j,\operatorname{(oracle,lpol)}}-\hat f_1 ^{j,\operatorname{(lpol)}}\|_\infty &\leq& C \big [\DeltaThm + d_1^{-\rho^*} + (d_1 h)^{-1} (nh)^{-1/2}\sqrt{ \log n}\big ]
\end{align*} uniformly  for all $h$ with $c_- n^{-\eta_1} \leq h \leq c_+ n^{-\eta_2}$
with probability greater than or equal to $ 1 - 3 n^{-a} $. 
\end{thm}

Applying this theorem with $j=0$ and $k=r_1-1$ and with a choice of $h$ of optimal order $n^{-1/(2r_1 +1)}$, we can show that the two-step estimator is asymptotically equivalent to a local polynomial estimator in the oracle model if  \eqref{eqn:resmoothbound1}--\eqref{eqn:resmoothbound8} holds with $\beta = r_1/(2r_1 +1)$. Furthermore, one can argue in the same way as in the discussion after Theorem \ref{theo:splinepoly} to get asymptotic oracle equivalence of the two-step estimator and the oracle estimator. 

We remark that the last two theorems have an important implication for practical applications.  The asymptotic behaviour of classical smoothing estimators in the oracle model is well understood. Classical smoothing theory offers results on pointwise asymptotic normal limits and global approximations by Gaussian processes. These results can be used for the implementation of pointwise and global confidence regions. Furthermore  testing procedures on the shape of the function $f_1$, e.g. on the validity of parametric specifications for $f_1$, can be constructed by using the distance of the estimator of $f_1$ from the hypothesis. There exists a rich literature on such approaches in the oracle model. By application of our results, this theory carries over to the additive model as long as the performance of the statistical procedure is robust against small changes of the values of the response variables. Then theoretical results and the validity of statistical methods in the oracle model also  apply  in the additive model for statistical methods based on our two-step procedure.  In particular, for pointwise and uniform confidence regions of the function $f_1$, this can be easily verified if the smoother fulfills the small error robustness property \eqref{eqn:smoothcontin}. 

In Subsection \ref{subsec:NaWa} in the supplementary material we explain this in detail for two-step estimators with Nadaraya-Watson smoothing used in the second step.  For such estimators the subsection contains a discussion of construction of pointwise and global confidence regions, nonparametric testing, the estimation of the error variance and a short remark on global bandwidth selection. A detailed discussion of data adaptive local bandwidth selectors for kernel smoothing can be found in 
Subsection \ref{subsec:Lepski} in the supplementary material. There, we consider for kernel smoothing Lepski-type methods for adaptive pointwise estimation (cf.~\cite{LepSpok, LepMamSpok,Gaif}). We show that Lepski-type bandwidth selection in the oracle model and in the additive model leads to asymptotically equivalent kernel smoothers. That means that the difference between the kernel smoother in the oracle model with bandwidth selector based on oracle data is asymptotically equivalent to the two-step estimator with kernel estimation in the second step when the bandwidth choice is based on the pseudo-data coming from the first step of the procedure.

\section{Simulation studies}\FloatBarrier
\label{sec:sim}
We assess in simulation the coverage of pointwise confidence intervals for $f_1$ based on the resmoothed estimator $\hat f_1$. We construct the presmoothing estimators $\hat f_1^{\operatorname{(pre)}},\dots,\hat f_q^{\operatorname{(pre)}}$ choosing $\hat \f_{-1}^{(\text{init})}$ as $\hat \f^L_{-1}$ and $\hat \Pi_{-1}$ as $\hat\Pi^L_{-1}$ as described in Section \ref{sec:spaaddmod}, and we select the group Lasso tuning parameters $\lambda$ and $\eta$ via crossvalidation.  In place of the Legendre polynomial basis described in Section \ref{ssec:addmod}, however, we use a cubic B-spline basis with $\lfloor 2n^{1/2} \rfloor$ basis functions (a large number to promote undersmoothing).  For the resmoothing estimator we consider a least-squares splines estimator $\hat f_1^{(\text{spl})}$ with cubic $B$-splines. Under this choice of resmoothing estimator we have for each $x_0\in[0,1]$ the representation $\hat f_1^{(\text{spl})}(x_0) =  \w_{x_0}^T\hat{\f}_1^{(\operatorname{pre})}$ for some $n\times 1$ vector $\w_{x_0}$. Assuming that the dominating term in  $\hat f_1^{(\text{spl})}(x_0) - f_1(x_0)$ is $\w_{x_0}^T(I - A)^{-1}(\hat \Pi_1 - A) \error$, we consider the confidence interval
\begin{table}[h]
\centering\setlength{\tabcolsep}{3pt}
\begin{tabular}{crr|rrrrrrrrrr}
  \multicolumn{3}{c}{ }& \multicolumn{ 2 }{c}{ $f_1$ } & \multicolumn{ 2 }{c}{ $f_2$ } & \multicolumn{ 2 }{c}{ $f_3$ } & \multicolumn{ 2 }{c}{ $f_4$ } & \multicolumn{ 2 }{c}{ $f_5$ } \\$n$ & $q$ & $\zeta$ & -1.0 & 0.5 & -1.0 & 0.5 & -1.0 & 0.5 & -1.0 & 0.5 & -1.0 & 0.5 \\ 
  \hline
100 & 50 & 0.0 & 93 & 92 & 90 & 85 & 92 & 93 & 94 & 92 & 94 & 94 \\ 
   &  & 0.1 & 91 & 91 & 92 & 88 & 89 & 93 & 95 & 90 & 94 & 95 \\ 
   &  & 0.3 & 85 & 93 & 93 & 86 & 84 & 92 & 92 & 90 & 95 & 94 \\ 
   &  & 0.5 & 79 & 91 & 92 & 82 & 74 & 78 & 95 & 85 & 92 & 94 \\ 
   & 150 & 0.0 & 89 & 91 & 87 & 83 & 89 & 93 & 93 & 93 & 95 & 94 \\ 
   &  & 0.1 & 93 & 92 & 91 & 85 & 87 & 91 & 95 & 95 & 95 & 95 \\ 
   &  & 0.3 & 89 & 91 & 91 & 90 & 87 & 91 & 93 & 91 & 94 & 95 \\ 
   &  & 0.5 & 89 & 94 & 92 & 88 & 85 & 83 & 93 & 89 & 92 & 92 \\ 
   \hline
1000 & 50 & 0.0 & 91 & 92 & 95 & 94 & 95 & 93 & 96 & 96 & 94 & 95 \\ 
   &  & 0.1 & 92 & 92 & 95 & 95 & 93 & 94 & 96 & 94 & 94 & 95 \\ 
   &  & 0.3 & 89 & 91 & 94 & 94 & 90 & 92 & 94 & 96 & 92 & 93 \\ 
   &  & 0.5 & 88 & 94 & 94 & 91 & 86 & 84 & 92 & 93 & 94 & 93 \\ 
   & 150 & 0.0 & 92 & 93 & 95 & 94 & 94 & 94 & 94 & 97 & 95 & 93 \\ 
   &  & 0.1 & 91 & 93 & 95 & 95 & 93 & 94 & 96 & 94 & 94 & 94 \\ 
   &  & 0.3 & 92 & 92 & 95 & 92 & 88 & 92 & 95 & 94 & 95 & 93 \\ 
   &  & 0.5 & 87 & 93 & 93 & 92 & 80 & 82 & 89 & 91 & 93 & 94 \\ 
  \end{tabular}
\caption{Coverage $(\times 100)$ of confidence intervals in \eqref{ciresmoothed} for functions $f_1,\dots,f_5$ at $x = -1.0,0.5$.} 
\label{table_cov_diff}
\end{table}
\begin{equation}\label{ciresmoothed}
	\hat f_1^{\text{(spl)}}(x_0) \pm z_{\alpha/2} \|((I-A)^{-1}(\hat \Pi_1 - A))^T\w_{x_0}\|_2 \hat \sigma_1
\end{equation}
for $f_1(x_0)$, where $\|((I-A)^{-1}(\hat \Pi_1 - A))^T\w_{x_0}\|_2$ denotes the Euclidean norm, $z_{\alpha/2}$ is the upper $\alpha/2$ quantile of the standard Normal distribution and $\hat \sigma_1$ is an estimator of $\sigma$ which we define in Appendix \ref{simvarest} in the Supplementary Material.

\begin{figure}[h]
    \centering
    \includegraphics[width=\textwidth, height=140px]{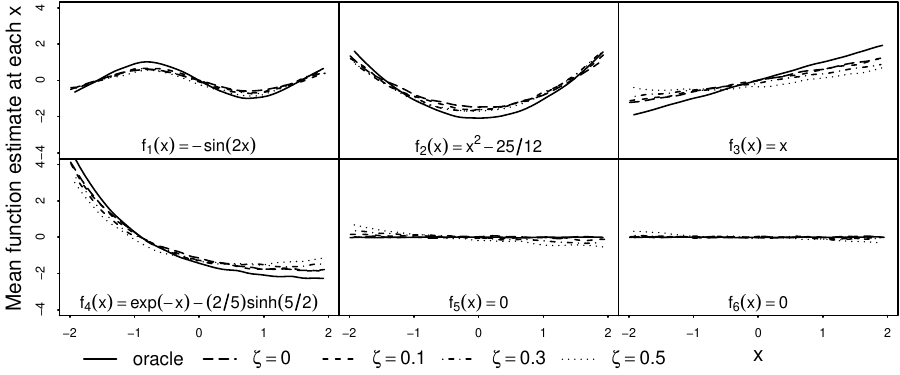}
	\caption{Average values of $\hat f_j^{(\operatorname{oracle,spl})}(x)$ and $\hat f_j^{\text{(spl)}}(x)$, $j=1,\dots,6,$ from $500$ simulated data sets over a range of $x$ values in the $n=100$, $q=150$ setting under each of the correlation settings $\zeta=0.0, 0.1, 0.3, 0.5$.}
    \label{fig:n100_q150_estplot}
\end{figure}

\begin{figure}[h]
    \centering
    \includegraphics[width=\textwidth, height=140px]{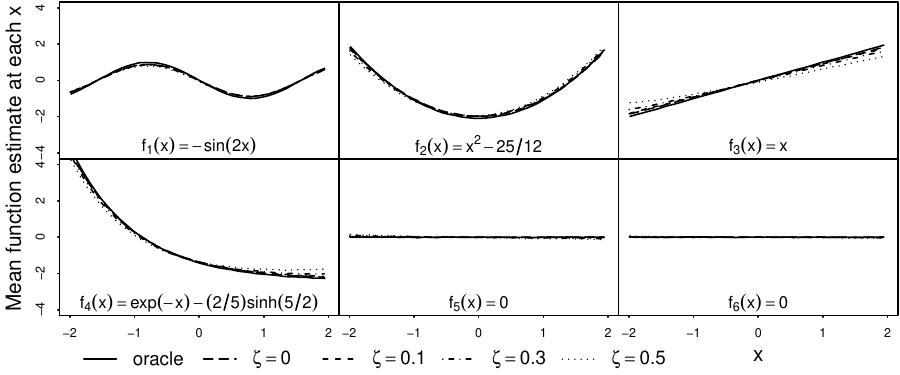}
	\caption{Average values of $\hat f_j^{(\operatorname{oracle,spl})}(x)$ and $\hat f_j^{\text{(spl)}}(x)$, $j=1,\dots,6,$ from $500$ simulated data sets over a range of $x$ values in the $n=1000$, $q=150$ setting under each of the correlation settings $\zeta=0.0, 0.1, 0.3, 0.5$.}
    \label{fig:n1000_q150_estplot}
\end{figure}
\par
\par
We consider a model with four active covariates among a large number of inactive covariates.  For $q = 50,150$ and sample sizes $n=100,1000$, we generate data from the model $Y^i = \sum_{j=1}^q f_j(X_j^i)+\epsilon^i$, $i=1,\dots,n$, where $\epsilon^1,\dots,\epsilon^n$ are independent Normal$(0,1)$ random variables, $f_5 = \dots=f_q = 0$ and
\[
	f_1(x) = -\sin(2x), \quad f_2(x) = x^2 - \frac{25}{12}, \quad f_3(x) = x, \quad f_4(x) = e^{-x} - \frac{2}{5}\text{sinh}(\frac{5}{2}),
\]
which we take from \cite{MGB}. The covariates $X^1,\dots,X^n$ are generated as independent realizations of the random variable $X = (X_1,\dots,X_q)^T$, where $X_j$ is marginally Uniform on $[-2.5,2.5]$ for $j=1,\dots,q$ and $\text{corr}(X_j,X_{j'}) = \zeta^{|j - j'|}$, $1\leq j,j'\leq q$. Under all combinations of the sample size $n$ and the number of covariates $q$ we conduct simulations under the values $\zeta = 0.0,0.1,0.3,0.5$.

\par
Table \ref{table_cov_diff} reports the coverage over $500$ simulated data sets of the confidence intervals for $f_1(x),\dots,f_5(x)$ given by \eqref{ciresmoothed} under $\alpha = 0.05$.  As we might expect, the coverage is closer to the nominal level for $n=1000$ than for $n=100$ and further from the nominal level for larger values of $\zeta$. When there is undercoverage, we find that it owes in large part to the bias of the point estimate, some part of which comes from the group Lasso sparsity penalization in the presmoothing step, and some part of which is the same bias to which the oracle estimator is subject. Results provided in Appendix \ref{additionaloutput} in the Supplementary Material show that the confidence interval in $\eqref{ciresmoothed}$ achieved almost exactly nominal coverage of $\mathbf{E} \hat f_1^{\text{(spl)}}(x_0)$, and that the oracle-based confidence intervals achieved sub-nominal coverage in some cases, which we ascribe to some amount of bias in the oracle point estimator.

\par

Figures \ref{fig:n100_q150_estplot} and  \ref{fig:n1000_q150_estplot} give impressions of the closeness of the resmoothing estimator to the oracle estimator.  Figure \ref{fig:n100_q150_estplot} shows under $n=100$ and $q=150$ the average values of the resmoothed estimators $\hat f_1^{\text{(spl)}}(x),\dots,\hat f_6(x)^{\text{(spl)}}$ as well as of the oracle estimators $\hat f_1^{(\operatorname{oracle,spl})}(x),\dots,\hat f_6^{(\operatorname{oracle,spl})}(x)$ under each of the settings $\zeta = 0.0, 0.1, 0.3, 0.5$.  We see that the resmoothing estimator is at this sample size not very close to the oracle estimator, exhibiting some shrinkage toward zero.  For the $n=1000$ and $q = 150$ case, for which Figure \ref{fig:n1000_q150_estplot} displays the results,  the resmoothing estimator is, in contrast, much closer to the oracle estimator.
Complete simulation details and additional output can be found in Appendix \ref{simdetails} in the Supplementary Material.
\FloatBarrier

\vspace{12pt}
\section{Proofs}\label{sec:math}

\subsection{Proof of Theorem \ref{mainthm}}
We abbreviate $B= \hat \Pi_1 -A$, so that $ \hat {\mathbf{f}}_1^{\text{(pre)}} =(I-A)^{-1} B (Y - \hat {\mathbf{f}}_{-1}^{\text{(init)}}) $. We make use of the following inequality:
\begin{eqnarray*} 
\| \hat  {\mathbf{f}}_1^{\text{(pre)}} - \hat  {\mathbf{f}}_1^{\text{(oracle, pre)}}\|_{\infty}
&\leq& \| (I-A)^{-1} B \boldsymbol{\epsilon} - \hat \Pi_1 \boldsymbol{\epsilon} \|_\infty\\
&&+
\| (I-A)^{-1} B( {\mathbf{f}}_{-1} -  {\mathbf{g}}_{-1}) \|_\infty\\&&
+ \| (I-A)^{-1} B ( {\mathbf{g}}_{-1} - \hat  {\mathbf{f}}_{-1}^{\text{(init)}})\|_\infty\\
&& + \| (I-A)^{-1} B  {\mathbf{f}}_1 - \hat \Pi_1  {\mathbf{f}}_1 \|_\infty\\
&=& T_1+ ...+ T_4.
\end{eqnarray*} 
By (A1), we get $\|A\|_{\operatorname{op}}=\|A^T\|_{\operatorname{op}}\leq \rho_1$. Hence, $(I-A)^{-1}=\sum_{j\geq 0}A^j$ and 
\begin{equation}\label{boundadd}
(I-A)^{-1} B=\hat{\Pi}_1-\sum_{j\geq 1}A^j(I-\hat \Pi_1).
\end{equation}
Applying \eqref{boundadd}, we get 
\[
T_1=\|\sum_{j\geq 1}A^j(I-\hat \Pi_1)\boldsymbol{\epsilon}\|_{n,\infty}=\max_{i=1,\dots,n}|U_i|,\qquad
U_i=\langle \mathbf{e}_i,\sum_{j\geq 1}A^j(I-\hat \Pi_1)\boldsymbol{\epsilon}\rangle.
\]
The random variable $U_i$ is Gaussian with expectation $0$ and variance $\sigma^2n\|\sum_{j\geq 1}(I-\hat \Pi_1)(A^T)^j\mathbf{e}_i\|_{n,2}^2$. Since 
\begin{align*}
\|\sum_{j\geq 1}(I-\hat \Pi_1)(A^T)^j\mathbf{e}_i\|_{n,2}&\leq \sum_{j\geq 1}\|(I-\hat \Pi_1)(A^T)^j\mathbf{e}_i\|_{n,2}\\
&\leq \sum_{j\geq 1}\rho_1^{j-1}\|A^T\mathbf{e}_i\|_{n,2}\leq \frac{1}{1-\rho_1}\frac{\Delta_2}{n},
\end{align*}
where we applied (A1) and (A3), we get that $T_1$ is a maximum of $n$ Gaussian random variables having variance bounded by $\sigma^2\Delta_2^2(1-\rho_1)^{-2}n^{-1}$. Thus, by the union bound and standard tail bounds for Gaussian random variables, we have with probability $\geq 1-\delta$, 
\begin{equation}\label{T1bound}
T_1\leq   \frac{\sigma\Delta_2}{1-\rho_1}\sqrt{\frac{2\log(2n)+2\log(1/\delta)}{n}}.
\end{equation}
For $T_2$ first note that 
\begin{equation}\label{EqAInfty}
\|Ay\|_{n,\infty}=\max_{i=1,\dots,n}|\langle \mathbf{e}_i,Ay\rangle|\leq n\max_{i=1,\dots,n}\|A^T\mathbf{e}_i\|_{n,2}\|y\|_{n,2}\leq \Delta_2\|y\|_{n,2}
\end{equation}
for all $y\in \mathbb{R}^n$.
Using this and (A4), we get
\begin{equation}\label{EqBInfty}
\|By\|_{n,\infty}\leq \|\hat \Pi_1 y\|_{n,\infty}+\|Ay\|_{n,\infty}\leq \Delta_3\|y\|_{n,\infty}+\Delta_2\|y\|_{n,2}
\end{equation}
for all $y\in \mathbb{R}^n$. Thus
\begin{align*}
T_2&\leq \sum_{j\geq 0}\|A^jB( {\mathbf{f}}_{-1} -  {\mathbf{g}}_{-1}) \|_{n,\infty}\\
&\leq \|B( {\mathbf{f}}_{-1} -  {\mathbf{g}}_{-1}) \|_{n,\infty}+\Delta_2\sum_{j\geq 1}\rho_1^{j-1}\|B( {\mathbf{f}}_{-1} -  {\mathbf{g}}_{-1}) \|_{n,2}\\
&\leq (1+\Delta_2(1-\rho_1)^{-1})(\Delta_2+\Delta_3)\| \mathbf{f}_{-1} - \mathbf{g}_{-1}\|_{n,\infty},
\end{align*}
where we applied that $\| A y\|_{n,2} \leq \rho_1\|y\| _{n,2}$ for all $y\in\mathbb{R}^n$ because of (A1), and where we made use of Equations \eqref{EqAInfty} and \eqref{EqBInfty}. Similarly, for $T_3$ we get that
\begin{align*} 
T_3
\leq \sum_{j\geq 0}\| A^j B ( {\mathbf{g}}_{-1} - \hat  {\mathbf{f}}_{-1}^{\text{(init)}})\|_{n,\infty} 
&\leq  (1+\Delta_2(1-\rho_1)^{-1})\| ( B ( {\mathbf{g}}_{-1} - \hat  {\mathbf{f}}_{-1}^{\text{(init)}})\|_{n,\infty}\\
&\leq (1+\Delta_2(1-\rho_1)^{-1})\Delta_1\xi(\hat {\mathbf{f}}^{\text{(init)}}_{-1} - \mathbf{g}_{-1}),
 \end{align*} 
where we applied (A3) in the last inequality. For the discussion of $T_4$ note that because of \eqref{boundadd}
\begin{align*} (I-A)^{-1} B  {\mathbf{f}}_1 - \hat \Pi_1  {\mathbf{f}}_1= - \sum _{j \geq 1} A^j (I-\hat \Pi_1)  {\mathbf{f}}_1
=  - \sum _{j \geq 1} A^j (I-\hat \Pi_1) ( {\mathbf{f}}_1-  {\mathbf{g}}_1).
\end{align*}
Thus
$
T_4=\|\sum _{j \geq 1} A^j (I-\hat \Pi_1) ( {\mathbf{f}}_1-  {\mathbf{g}}_1)\|_{n,\infty}\leq \Delta_2(1-\rho_1)^{-1}\| \mathbf{f}_{1} - \mathbf{g}_{1}\|_{n,\infty}$.
This completes the proof.\qed

\subsection{Events}\label{SecEvents}
In this section, we define several events upon which the proof of Theorem \ref{ourchoicesthm} will hold.  First, we define
\begin{equation*}
\EventLambda=\Big\{2\max_{j=1,\dots,q}\sup_{0\neq g_j\in V_j} \frac{|\langle \epsilon,g_j\rangle_n|}{\|g_j\|_n} \leq \lambda\Big\}\cap \mathcal{E}_{\phi,J_0},
\end{equation*}
where $\mathcal{E}_{\phi,J_0}$ is the compatibility condition event defined as the event on which 
\[
\sum_{j\in J_0}\|g_j\|_n^2\leq 3\|\sum_{j=1}^q g_j\|_n^2/\phi^2
\] 
for all $(g_1,\dots,g_q)\in (P_1,\dots,P_q)$ satisfying $\sum_{j=1}^q\|g_j\|_n\leq 8\sum_{j\in J_0}\|g_j\|_n$. Note that $\EventLambda$ is needed in the analysis of the group Lasso estimator of $f$. 

We also define
\begin{equation*}
\EventEta= \bigcap_{k=1}^{d_1}\Big\{2\max_{j=2,\dots,q}\sup_{0\neq g_j\in V_j} \frac{|\langle b_{1k}-\Pi_{-1}b_{1k},g_j\rangle_n|}{\|g_j\|_n} \leq\eta\Big\}\cap \mathcal{E}_{\phi,J_k},
\end{equation*}
where $ \mathcal{E}_{\phi,J_k}$ is the compatibility condition event defined as the event on which 
\[
\sum_{j\in J_k}\|g_j\|_n^2\leq 3\|\sum_{j=2}^q g_j\|_n^2/\phi^2
\]
for all $(g_2,\dots,g_q)\in (P_2,\dots,P_q)$ satisfying $\sum_{j=2}^q\|g_j\|_n\leq 8\sum_{j\in J_k}\|g_j\|_n$.
Note that $\EventEta$ is needed in the analysis of $\hat\Pi_{-1}^L$.

Finally, for $\nu\in(0,1/2)$, we define the empirical norm approximation event
\begin{align*}
\EventNu&=\mathcal{E}_{\nu,1}\cap\bigcap_{k,l=1}^{d_1}\left\lbrace \left|\langle \Pi_{J_k}b_{1k},\Pi_{J_l}b_{1l}\rangle-\langle \Pi_{J_k}b_{1k},\Pi_{J_l}b_{1l}\rangle_n\right|\leq \nu \|\Pi_{J_k}b_{1k}\|\|\Pi_{J_l}b_{1l}\| \right\rbrace\\
&\text{with}\quad\mathcal{E}_{\nu,1} = \bigcap_{j=1}^q\left\{(1-\nu)\|g_j\|^2\leq\|g_j\|^2_n\leq(1+\nu)\|g_j\|^2\text{ for all }g_j\in P_j\right\}.
\end{align*}
Note that if $\EventNu$ holds, then each $g_1\in P_1$ is uniquely determined by $\mathbf{g}_1$, meaning that we do not have to distinguish between those objects. For instance, if $\EventNu$ holds, then $\hat\Pi_1 y$, $y\in\mathbb{R}^n$, can also be considered as a function in $P_1$ and we do not have to distinguish between $\hat\Pi_{-1}^L\mathbf{g}_1$ and $\hat\Pi_{-1}^Lg_1$.

The following concentration result gives a lower bound on the probability of the events defined in this section.  It is proven in Appendix \ref{AppendixBadEvents} in the Supplementary Material.

\begin{proposition}\label{mthmlowbounds}
Suppose that Assumptions (B1), (B4), (B5) and (B6) hold. Then there are constants $c,C>0$ such that the following holds. For $a> 1$, let $\lambda$ and $\eta$ be as defined in \eqref{EqChoiceLambdaEta}, and let 
\begin{equation}\label{EqChoiceNu}
\nu=\frac{C}{\psi}\sqrt{\frac{s_1d(a\log n+ \log q)}{n}}.
\end{equation}
If the second inequality in \eqref{EqThmLCond} holds, then we have
$\mathbb{P}\left(\EventLambda\cap\EventEta\cap\EventNu\right)\geq 1- n^{-a}$.
\end{proposition}

\subsection{The nonparametric group Lasso estimator} In this section, we state a risk bound for the group Lasso estimator $\hat{f}^L$ in Section \ref{ssec:initest} which is suitable to our purposes. In order to bound the approximation error terms we need a risk bound for the group Lasso estimator in the undersmoothed case. From now on we will denote the nonparametric group Lasso penalty by
\[
\pen_{\lambda}(g) = 2 \lambda \sum_{j=1}^{q}\|g_j\|_n, \hbox{ for any } g \in P_{\{1,\dots,q\}}.
\]
Applying the work by Bickel, Ritov, and Tsybakov \cite{BRT}, we obtain:
\begin{proposition}\label{lassoest1} Suppose that Assumption (B3) holds.
If $\EventLambda$ holds, then we have for each $g\in P_{J_0}$,
\[
\|\hat{f}^L-f\|^2_n+\operatorname{pen}_\lambda(\hat{f}^L-g) \leq 4\|f-g\|_n^2+240s_0\lambda^2/\phi^2.
\]
In particular, if we choose $g^*$ from Assumption (B3), then we have on $\EventLambda$,
\[
\|\hat{f}^L-f\|^2_n+\operatorname{pen}_\lambda(\hat{f}^L-g^*) \leq 4C_0^2(d_1^{-r_1}+s_0d_2^{-r_2})^2+240 s_0\lambda^2/\phi^2.
\]
\end{proposition}

\subsection{The Lasso projection of relaxed orthogonality} 
In this section, we state risk bounds for the Lasso estimators $\hat{\Pi}_{-1}^Lb_{1l}$ of $\Pi_{-1}b_{1l}$. The analysis is analogous to that of the Lasso estimator $\hat f^L$ of $f$.  We only have to replace $Y$ by $b_{1l}$, $f$ by $\Pi_{-1}b_{1l}$, and $\epsilon$ by $b_{1l}-\Pi_{-1}b_{1l}$. Note that for all $g \in P_{-1}$, we have
$
\left\langle b_{1l}-\Pi_{-1}b_{1l},g \right\rangle=0$.
Let
\[
\pen_{\eta}(g) = 2 \eta \sum_{j=2}^{q}\|g_{j}\|_n, \hbox{ for any } g \in P_{-1}.
\]
The following result is similar to the result above. A proof of Propositions \ref{lassoest1} and \ref{lassoest} is given in Appendix \ref{AppendixLasso} in the Supplementary Material.
\begin{proposition}\label{lassoest} Suppose that Assumption (B4) holds. Let $l\in \lbrace 1,\dots,d_1\rbrace$. If $\EventEta$, then for each $g\in P_{J_l}$,
\[
\|\hat{\Pi}_{-1}^Lb_{1l}-\Pi_{-1}b_{1l}\|^2_n +\pen_\eta\left(\hat\Pi^L_{-1}b_{1l}-g\right)\leq 4\|\Pi_{-1}b_{1l}-g\|_n^2 +240s_1 \eta^2/\phi^2.
\]
In particular, choosing $g=\Pi_{J_l}b_{1l}$ gives on $\EventEta\cap\EventNu$ with $
\eta\geq \sqrt{d/n}$,
\[
\|\hat{\Pi}_{-1}^Lb_{1l}-\Pi_{-1}b_{1l}\|^2_n +\pen_\eta\left(\hat\Pi^L_{-1}b_{1l}-\Pi_{J_l}b_{1l}\right)\leq \left(4C_1^2+(240/\phi^2)\right)  s_1\eta^2.
\]
\end{proposition}
Finally, let us state the {(nonparametric)} KKT conditions in the case of the empirical Lasso projection. 
\begin{lemma}\label{KKT}
For all $l\in\{1,\dots,d_1\}$ and all $g\in P_{-1}$, we have
\[
2|\langle g,b_{1l}-\hat{\Pi}_{-1}^Lb_{1l}\rangle_n|\leq \pen_\eta(g).
\]
\end{lemma}
A proof of Lemma \ref{KKT} is given in Appendix \ref{proofLemmaKKT} in the Supplementary Material.

\subsection{Fundamental properties of $\hat{\Pi}_{-1}^L$}\label{fpo}
In this section, we state some properties of $\hat{\Pi}_{-1}^L$ needed to verify (A1) and (A3). Recall the convention that $C$ denotes a constant depending only on the quantities $t$, $c_1$, and $C_1$ and that $C$ is not necessarily the same at each occurrence. 

 Using the smoothness of the joint densities $p_{1j}$ from (B1), it is easy to see that $\|\Pi_{\{j\}}b_{1k}\|\leq C/\sqrt{d_1}$, from which it follows that $\|\Pi_{J_k}b_{1k}\|\leq (C/\psi)\sqrt{(s_1/d_1)}$. The following proposition shows that a similar smoothing property also holds when $\hat{\Pi}_{-1}^L$ is applied to functions having local support.

\begin{proposition}\label{prem} Suppose that Assumptions (B1) and (B4) hold. If $\EventEta\cap \EventNu$ holds with $\eta\geq \sqrt{d/n}$, then for each $g_1\in V_1$ satisfying $\operatorname{supp}(g_{1})\subseteq I_{1k'}$ for some $k'\in\{1,\dots,m_1\}$,
\[
\|\hat{\Pi}_{-1}^Lg_1\|_n\leq C\left( \frac{1}{\psi}\sqrt{\frac{s_1}{d_1}}+\frac{\sqrt{s_1}\eta}{\phi}\right)\|g_1\|_n.
\] 

In particular, if we choose $\eta$ as in \eqref{EqChoiceLambdaEta} and if the second inequality in \eqref{EqThmLCond} holds, then we have on $\EventEta\cap \EventNu$,
\[
\|\hat{\Pi}_{-1}^Lg_1\|_n\leq \frac{C}{\psi}\sqrt{\frac{s_1}{d_1}}\|g_1\|_n.
\] 
\end{proposition}

A proof of Proposition \ref{prem} is given in Appendix \ref{AppendixEvalBasis} in the Supplementary Material. The next proposition, proved in Appendix \ref{Proofempangle} in the Supplementary Material, studies an empirical version of~$\rho_0$.
\begin{proposition}\label{empangle} Suppose that Assumptions (B1), (B2) and (B4) hold. Then we have on $\EventEta\cap \EventNu$ with $
\eta\geq \sqrt{d/n}$,
\[
\sup_{ g_1\in P_1:\| g_1\|_n\leq 1}\|\hat{\Pi}_{-1}^Lg_1\|_n^2\leq \rho_0^2+C\left(\frac{s_1\nu}{\psi^2}+\frac{s_1\sqrt{d_1}\eta}{\psi\phi}+\frac{s_1 d_1\eta^2}{\phi^2} \right).
\]
In particular, if we choose $\eta$ and $\nu$ as in \eqref{EqChoiceLambdaEta} and \eqref{EqChoiceNu}, and if the first inequality in \eqref{EqThmLCond} holds with $c$ sufficiently small, then  we have on $\EventEta\cap \EventNu$,
\[
\forall g_1\in P_1,\qquad \|\hat{\Pi}_{-1}^Lg_1\|_n \leq (1+\rho_0)/2\|g_1\|_n.
\]
\end{proposition}

\subsection{End of proof of Theorem \ref{ourchoicesthm}}
Let $\phi_{11},\dots,\phi_{1d_1}$ be the orthonormal basis of $P_1$ with respect to the empirical inner product obtained by applying the Gram-Schmidt orthogonalization to the basis $b_{11},\dots,b_{1d_1}$. Clearly, this basis is still local in the sense that $\phi_{1,k(t_1+1)+1},\dots,\phi_{1,k(t_1+1)+t_1+1}$ is a basis of the functions in $P_1$ which are zero outside the interval $I_{1k}$. The existence of this basis is guaranteed on the event $\EventNu$.

We begin with proving an auxiliary result on the supremum norm of functions in $P_1$. The space of piecewise polynomials has a number of important properties, among which we need that
\begin{equation}\label{eq:lpinfty}
\|g_j\|_{\infty}^2\leq (t_j+1)^2m_j\int_0^1 g_j^2(x_j)dx_j
\end{equation}
for each $g_j\in B_j$ (see, e.g., \cite[Equation (7)]{BM2}). Combining \eqref{eq:lpinfty} with Assumption (B1), we have $\|g_1\|_\infty\leq (t_1+1)\sqrt{(m_1/c_1)}\|g_1\|$ for all $g_1\in P_1$. Hence, if $\EventNu$ holds (recall that $\nu\leq 1/2$), then we have 
\begin{equation}\label{eq:boundONB}
\| \phi_{1k}\|_\infty\leq  C \sqrt{d_1}.
\end{equation}
Now, let $g_1=\sum_{k=1}^{d_1}\alpha_k \phi_{1k}$. Since on each interval $I_{1k'}$ there are at most $t_1+1$ non-zero basis functions, we have $\|g_1\|_\infty\leq (t_1+1)\max_{1\leq k \leq d_1}\|\phi_{1k}\|_{\infty}|\alpha_k|$, and inserting \eqref{eq:boundONB} leads to
\begin{equation}\label{eq:asd}
\| g_1\|_\infty\leq C\sqrt{d_1}\|\alpha\|_\infty
\end{equation}
with $C$ depending only on $c_1$ and $t_1$.

We now verify properties (A1)--(A4) on the event $\EventLambda\cap \EventEta\cap \EventNu$ under the choices and assumptions made in Theorem~\ref{ourchoicesthm}, with $\xi: V_{-1} \to \mathbb{R}^+$ given by $\xi(g_{-1}) = \sum_{j=2}^q \|g_j\|_n$.

(A1): By Proposition \ref{empangle}, we have on $\EventEta\cap \EventNu$ that $\|\hat \Pi_{-1}^Lg_1\|_{n}\leq (1+\rho_0)/2\|g_1\|_{n}$ for every $g_1\in P_1$, meaning that (A1) holds with $\rho_1=(1+\rho_0)/2$.

(A2): Since $(\hat \Pi_1 - \hat\Pi_{-1}^L \hat \Pi_1)^T \mathbf{g}_{-1}\in V_1$, inequality \eqref{eq:asd} implies
\[
    \|(\hat \Pi_1 - \hat \Pi_{-1}^L \hat \Pi_1)^T \mathbf{g}_{-1}\|_{n,\infty} \leq C\sqrt{d_1}\max_{1\leq k \leq d_1}|\langle (\hat \Pi_1 - \hat \Pi_{-1}^L \hat \Pi_1)^T \mathbf{g}_{-1} , \phi_{1k}\rangle_n|.
\]
Applying that $\phi_{1k}$ is a linear combination of at most $t_1+1$ basis functions $(b_{1l})$ and that on $\mathcal{E}_{\nu,1}$ we have $\|\sum_{k=1}^{d_1}\alpha_kb_{1k}\|_n^2\geq \|\sum_{k=1}^{d_1}\alpha_kb_{1k}\|^2/2\geq (c_1/2)\|\alpha\|_2^2$ for each $\alpha\in\mathbb{R}^{d_1}$, we get 
\begin{align*}
    \|(\hat \Pi_1 - \hat \Pi_{-1}^L \hat \Pi_1)^T \mathbf{g}_{-1}\|_{n,\infty} &\leq C\sqrt{d_1}\max_{1\leq k \leq d_1}|\langle (\hat \Pi_1 - \hat \Pi_{-1}^L \hat \Pi_1)^T \mathbf{g}_{-1} , b_{1k}\rangle_n| \\
    &= C\sqrt{d_1}\max_{1\leq k \leq d_1}|\langle g_{-1} , b_{1k} - \hat \Pi_{-1}^L b_{1k}\rangle_n|.
\end{align*}
Applying Lemma \ref{KKT}, we get
$
    \|(\hat \Pi_1 - \hat \Pi_{-1}^L \hat \Pi_1)^T g_{-1}\|_{n,\infty} \leq C\sqrt{d_1} \eta \xi(g_{-1})
$.
Thus (A2) is satisfied with $\Delta_1 = C \sqrt{d_1}\eta$.

(A3): Applying Proposition \ref{prem}, we have on $\EventEta\cap \EventNu$,
that $
\|\hat \Pi_{-1}^L\hat\Pi_1 \mathbf{e}_i\|_{n} \leq (C/\psi)\sqrt{s_1/d_1}\|\hat\Pi_1 \mathbf{e}_i\|_{n}$.
Applying \eqref{eq:boundONB}, we get with $k$ such that $X_1^i \in I_{1k}$ that
$
n\|\hat\Pi_1 \mathbf{e}_i\|_{n}=(\sum_{l=1}^{t_1+1}\phi_{1,k(t_1+1)+l}^2(X_1^i))^{1/2}\leq C\sqrt{d_1}
$.
Thus property (A3) is satisfied on $\EventEta\cap \EventNu$ with $\Delta_2 = (C/\psi) \sqrt{s_1}$.

(A4): First, on $\mathcal{E}_{\nu,1}$, the inequality 
\begin{equation}\label{EqSER1}
|\langle y,\phi_{1k}\rangle_n|\leq \|y\|_{n,\infty} \|\phi_{1k}\|_{n,\infty}|\{i:X_i\in I_k\}|/n\leq C\|y\|_{n,\infty}/\sqrt{d_1}
\end{equation}
holds for each $1\leq k\leq d_1$. Combining this with the identity 
\begin{equation}\label{EqSER2}
\hat\Pi_1 y=\sum_{k=1}^{d_1}\langle y,\phi_{1k}\rangle_n \phi_{1k}
\end{equation}
and \eqref{eq:asd}, we see that (A4) holds on $\mathcal{E}_{\nu,1}$ with $\Delta_3$ being a constant.

Finally, since Conditions (A1)--(A4) hold on the event $\EventLambda\cap \EventEta\cap \EventNu$, Proposition \ref{mthmlowbounds} yields that Conditions (A1)--(A4) hold with probability $\geq 1 - n^{-a}$.

Plugging these values of $\Delta_1,\Delta_2,\Delta_3,\rho_1$, in combination with Assumption (B3) and Proposition \ref{lassoest1} into Theorem \ref{mainthm} (applied conditional on the design, on the event  $\EventLambda\cap \EventEta\cap \EventNu$, with $\delta = n^{-a}$), we get, with probability at least $1 - 2n^{-a}$, the inequality
\begin{align}
\|\hat \f_1^{\operatorname{(pre)}} - &\hat \f_1^{\operatorname{(oracle,pre)}}\|_{n,\infty} \label{EqDelta_n^*pre}\\ &\leq  \frac{C\sqrt{s_1}}{(1-\rho_0)\psi}\bigg(\sigma\sqrt{\frac{a\log n }{n}}+(\sqrt{s_1}/\psi)(d_1^{-r_1} + s_0 d_2^{-r_2})\nonumber\\
& \qquad +  \sqrt{d_1}\eta\lambda ^{-1} ( d_1^{-2r_1} + s_0^2 d_2^{-2r_2}) + C\sqrt{d_1}\eta s_0 \lambda/\phi^2\bigg)\nonumber.
\end{align}
Combining the choices of $\lambda$ and $\eta$ in \eqref{EqChoiceLambdaEta} with the second condition in \eqref{EqThmLCond}, we get $\eta\leq \sqrt{Ld/n}$ and also $\sigma\sqrt{d/n}\leq \lambda\leq \sigma\sqrt{(d\wedge L)/n}$. The claim follows from plugging these values into \eqref{EqDelta_n^*pre}.\qed

\noindent \textsc{Acknowledgements.}
Financial support by Deutsche Forschungsgemeinschaft (DFG) through the Research Training Group 1953 is gratefully acknowledged. Research of the second author was prepared within the framework of a subsidy granted to the HSE by the Government of the Russian Federation for the implementation of the Global Competitiveness Program.

\bibliographystyle{plain}
\bibliography{lit}

\appendix

\section{Proofs for the results in Section \ref{sec:math}}\label{AppendixProofs}

\subsection{Proofs of Propositions \ref{lassoest1} and \ref{lassoest}}\label{AppendixLasso}

The proofs of Propositions \ref{lassoest1} and \ref{lassoest} follow the same line of argument which leads to \cite[Theorem 6.1]{BRT}. For completeness, we give the proof of Proposition \ref{lassoest}; the proof of Proposition \ref{lassoest1} is analogous.
\begin{lemma}\label{LBTR} Let $l\leq d_1$. If $\mathcal{A}_l\cap\mathcal{E}_{\nu,1}$ holds, then for all $g\in P_{-1}$, we have
\begin{align*}
&\|\hat{\Pi}_{-1}^Lb_{1l}-\Pi_{-1}b_{1l}\|^2_n+(1/2)\operatorname{pen}_{\eta}(\hat{\Pi}_{-1}^Lb_{1l}-g)\\
&\leq \|g-\Pi_{-1}b_{1l}\|_n^2+4\eta\sum_{j\in J(g)}\|(\hat{\Pi}_{-1}^Lb_{1l})_j-g_j\|_n,
\end{align*}
where $J(g) = \{j:\|g_j\|\neq 0\}$ and 
\[\mathcal{A}_l=\Big\{2\max_{j=2,\dots,q}\sup_{0\neq g_j\in V_j} \frac{|\langle b_{1l}-\Pi_{-1}b_{1l},g_j\rangle_n|}{\|g_j\|_n} \leq\eta\Big\}.
\]
\end{lemma}
\begin{proof}
The proof of this lemma is essentially the same as the proof of \cite[Lemma B.1]{BRT}, yet we provide it in our notation for the sake of completeness. By definition of $\hat{\Pi}_{-1}^Lb_{1l}$, we have for any $g\in P_{-1}$ the so-called basic inequality
\begin{equation*}
\|b_{1l}-\hat{\Pi}_{-1}^Lb_{1l}\|_n^2+\operatorname{pen}_{\eta}(\hat{\Pi}_{-1}^Lb_{1l})\leq \|b_{1l}-g\|_n^2+\operatorname{pen}_{\eta}(g),
\end{equation*}
which is equivalent to 
\begin{align*}
&\|\hat{\Pi}_{-1}^Lb_{1l}-\Pi_{-1}b_{1l}\|^2_n+\operatorname{pen}_{\eta}(\hat{\Pi}_{-1}^Lb_{1l})\\
&\leq \|g-\Pi_{-1}b_{1l}\|^2_n+2\langle b_{1l}-\Pi_{-1}b_{1l},\hat{\Pi}_{-1}^Lb_{1l}-g\rangle_n+\operatorname{pen}_{\eta}(g).
\end{align*}
If $\mathcal{A}_l$ holds, then we have
\begin{align*}
&\|\hat{\Pi}_{-1}^Lb_{1l}-\Pi_{-1}b_{1l}\|^2_n\\
&\leq \|g -\Pi_{-1}b_{1l}\|^2_n+(1/2)\operatorname{pen}_{\eta}(\hat{\Pi}_{-1}^Lb_{1l}-g)+\operatorname{pen}_{\eta}(g)-\operatorname{pen}_{\eta}(\hat{\Pi}_{-1}^Lb_{1l}).
\end{align*}
Adding the term $(1/2)\operatorname{pen}_{\eta}(\hat{\Pi}_{-1}^Lb_{1l}-g)$ to both sides, we obtain on $\mathcal{A}_l$,
\begin{align*}
&\|\hat{\Pi}_{-1}^Lb_{1l}-\Pi_{-1}b_{1l}\|^2_n+(1/2)\operatorname{pen}_{\eta}(\hat{\Pi}_{-1}^Lb_{1l}-g)\\
&\leq \|g-\Pi_{-1}b_{1l}\|^2_n+\eta\sum_{j=2}^q\left( \|(\hat{\Pi}_{-1}^Lb_{1l})_j-g_j\|_n+\|g_j\|_n-\|(\hat{\Pi}_{-1}^Lb_{1l})_j\|_n\right).
\end{align*}
Now, $\|(\hat{\Pi}_{-1}^Lb_{1l})_j-g_j\|_n+\|g_j\|_n-\|(\hat{\Pi}_{-1}^Lb_{1l})_j\|_n=0$ for $j\notin J(g)$, and the claim follows.
\end{proof}
\begin{proof}[Proof of Proposition \ref{lassoest}] Assume that $\mathcal{A}_l\cap \mathcal{E}_{\phi,J_l}$ holds and fix an element $g\in P_{J_l}$. We consider separately the cases that
\begin{equation}\label{eq:c1}
4\eta\sum_{j\in J_l}\|(\hat{\Pi}_{-1}^Lb_{1l})_j-g_j\|_n\leq \|g-\Pi_{-1}b_{1l}\|^2_n
\end{equation}
and
\begin{equation}\label{eq:c2}
\|g-\Pi_{-1}b_{1l}\|^2_n<4\eta\sum_{j\in J_l}\|(\hat{\Pi}_{-1}^Lb_{1l})_j-g_j\|_n.
\end{equation}
In the case of \eqref{eq:c1}, Lemma \ref{LBTR} yields
\begin{align*}
&\|\hat{\Pi}_{-1}^Lb_{1l}-\Pi_{-1}b_{1l}\|^2_n+(1/2)\operatorname{pen}_{\eta}(\hat{\Pi}_{-1}^Lb_{1l}-g)\leq 2\|g-\Pi_{-1}b_{1l}\|_n^2,
\end{align*}
and clearly Proposition \ref{lassoest} holds in this case.
In the case of \eqref{eq:c2},  Lemma \ref{LBTR} yields
\begin{align*}
2\|\hat{\Pi}_{-1}^Lb_{1l}-\Pi_{-1}b_{1l}\|^2_n+\operatorname{pen}_{\eta}(\hat{\Pi}_{-1}^Lb_{1l}-g)
&\leq 16\eta\sum_{j\in J_l}\|(\hat{\Pi}_{-1}^Lb_{1l})_j-g_j\|_n\\
&\leq 16\eta\sqrt{s_1}\sqrt{\sum_{j\in J_l}\|(\hat{\Pi}_{-1}^Lb_{1l})_j-g_j\|_n^2}\\
&\leq 16\eta\sqrt{3}\sqrt{s_1}\|\hat{\Pi}_{-1}^Lb_{1l}-g\|_n/\phi,
\end{align*}
where the first inequality implies the third inequality on $\mathcal{E}_{\phi,J_l}$. From $2xy \leq x^2 + y^2$, we get the inequalities $16\sqrt{3}uv \leq 192u^2 + v^2$ and $16\sqrt{3}uv \leq 48u^2 + 4v^2$, each of which we invoke to get
\begin{align*}
&2\|\hat{\Pi}_{-1}^Lb_{1l}-\Pi_{-1}b_{1l}\|^2_n+\operatorname{pen}_{\eta}(\hat{\Pi}_{-1}^Lb_{1l}-g)\\
&\leq 16\eta\sqrt{3}\sqrt{s_1}\|\hat{\Pi}_{-1}^Lb_{1l}-g\|_n/\phi\\
&\leq 16\eta\sqrt{3}\sqrt{s_1}\big( \|\hat{\Pi}_{-1}^Lb_{1l}-\Pi_{-1}b_{1l}\|_n+\|\Pi_{-1}b_{1l}-g\|_n\big) /\phi\\
&\leq \|\hat{\Pi}_{-1}^Lb_{1l}-\Pi_{-1}b_{1l}\|^2_n+4\|\Pi_{-1}b_{1l}-g\|_n^2+(240/\phi^2)s_1\eta^2,
\end{align*}
and the claim follows.
\end{proof}

\subsection{Proof of Lemma \ref{KKT}}\label{proofLemmaKKT}
The nonparametric Lasso estimator $\hat{\Pi}_{-1}^Lb_{1l}$ evaluated at the observations can be written as $\sum_{j=2}^q \X_j\hat{\gamma}^L_j$ with
\begin{equation*}
(\hat{\gamma}^L_2,\dots,\hat{\gamma}^L_q) \in \operatorname{argmin}\bigg\{ \Big\| \mathbf{b}_{1l} - \sum_{j=2}^q \X_j\gamma_j \Big\|_{n,2}^2 + 2 \eta \sum_{j=2}^q \| \X_j\gamma_j \|_{n,2} \bigg\},
\end{equation*}
where $\mathbf{X}_j=(b_{jk}(X_j^i))_{1\leq i\leq n,1\leq k\leq d_j}$.
The KKT conditions give that for $j=2,\dots,q$ (see, e.g., \cite[Chapter 4 and Appendix D]{Giraud}),
\[
2\X_j^T(\mathbf{b}_{1l} - \sum_{k=2}^q \X_k\hat{\gamma}_k^L)/n = 2\eta \X_j^T \hat{\kappa}_j/\sqrt{n},
\]
where 
\[
\hat{\kappa}_j=\X_j\hat{\gamma}_{j}/\|\X_j\hat{\gamma}_j\|_2\text{ for }\X_j\hat{\gamma}_j \neq 0\quad\text{and}\quad\|\hat{\kappa}_j\|_2\leq 1\text{ for }\X_j\hat{\gamma}_j = 0.
\]
For $g_j = \sum_{l = 1}^{d_j}\alpha_{jl}b_{jl}$, we conclude that
\begin{align*}
 \langle g_j  ,b_{1l} - \hat{\Pi}_{-1}^Lb_{1l}) \rangle_n & = \alpha_j^T\X_j^T(\mathbf{b}_{1l} - \hbox{$\sum_{k=2}^q$} \X_k\hat{\gamma}_k) / n \\
				& =  \eta \alpha_j^T\X^T_j \hat{\kappa}_j/\sqrt{n}   \leq \eta \|g_j\|_n \|\hat{\kappa}_j\|_2   \leq \eta \|g_j\|_n,
\end{align*}
and the claim follows from summing over $j=2,\dots,q$.\qed

\subsection{Proof of Proposition \ref{prem}}\label{AppendixEvalBasis}
We start with three lemmas:
\begin{lemma}\label{lp1} Suppose that Assumption (B1) holds. Then, for each $j=2,\dots,q$ and $k=1,\dots,d_1$, we have
\begin{equation*}
\|\Pi_{P_j}b_{1k}\|\leq C\frac{1}{\sqrt{d_1}},
\end{equation*}
where $\Pi_{P_j}$ denotes the orthogonal projection from $L^2(\mathbb{P}^X)$ onto $P_j$.
\end{lemma}
\begin{proof}
Since $P_j\subset B_j$, we have $\|\Pi_{P_j}b_{1k}\|^2\leq \|\Pi_{B_j}b_{1k}\|^2$, where $\Pi_{B_j}$ denotes the orthogonal projection from $L^2(\mathbb{P}^X)$ onto $B_j$. Now, let $(g_{jl})_{1\leq l \leq d_j}$ be the be the orthonormal basis of $B_j$ with respect to the $L^2(\mathbb{P}^{X_j})$-norm obtained by applying the Gram-Schmidt orthogonalization to the basis $b_{11},\dots,b_{1d_1}$. Then we have
\begin{equation}\label{eq:to1}
\|\Pi_{B_j}b_{1k}\|^2=\sum_{l=1}^{m_j(t_j+1)}\langle g_{jl},b_{1k}\rangle^2.
\end{equation}
Moreover, using \eqref{eq:lpinfty} and Assumption (B1), we have
\begin{equation*}
\|g_{jl}\|_\infty\leq \frac{(t+1)\sqrt{m_j}}{\sqrt{c_1}}.
\end{equation*}
Thus, using the second part of Assumption (B1) and again \eqref{eq:lpinfty}, we get
\begin{align*}
\langle g_{jl},b_{1k}\rangle&=\int_{[0,1]^2}g_{jl}(x_j)b_{1k}(x_1)p_{1j}(x_1,x_j)dx_1dx_j\\
&\leq \frac{1}{c_1}\int_{[0,1]^2}g_{jl}(x_j)b_{1k}(x_1)dx_1dx_j\\
&\leq \frac{1}{c_1}\frac{\|g_{jl}b_{1k}\|_\infty}{m_jm_1}\leq \frac{1}{c_1^{3/2}}\frac{(t+1)\sqrt{m_j}(t_1+1)\sqrt{m_1}}{m_jm_1}.
\end{align*}
Plugging this estimate into \eqref{eq:to1}, we get
\begin{equation*}
\|\Pi_{B_j}b_{1k}\|^2\leq \frac{1}{c_1^3}\frac{(t+1)^3(t_1+1)^2}{m_1}=\frac{1}{c_1^3}\frac{(t+1)^3(t_1+1)^3}{d_1},
\end{equation*}
and the claim follows.
\end{proof}

\begin{lemma}\label{lp2} Let $J\subseteq\{ 2,\dots,q\}.$ Moreover, let $0<\psi_J\leq 1$ be a number such that
\begin{equation}\label{eq:hk}
\Big\| \sum_{j\in J}g_j \Big\|^2 \geq \psi_{J}^2  \sum_{j\in J} \|g_j\|^2  
\end{equation}
for  all $g_j\in P_j$, $j\in J$. Then for each $f\in  L^2(\mathbb{P}^X)$, we have
\begin{equation*}
\|\Pi_{P_J} f\|^2 \leq \psi_J^{-2} \sum_{j\in J} \| \Pi_{P_j}f\|^2.
\end{equation*}
\end{lemma}

\begin{proof}  
For $j\in J$, let $(\psi_{jk})_{k\in\{1,\dots,d_2\}}$ be a basis of $P_j$. Let
\[
\Psi_j=\left( \langle \psi_{jk}, \psi_{jk'}\rangle\right)_{(k,k')\in \{1,\dots,d_2\}\times \{1,\dots,d_2\}}
\]
and 
\[
\Psi=\left( \langle \psi_{jk}, \psi_{j'k'}\rangle\right)_{((j,k),(j',k'))\in (J\times \{1,\dots,d_2\}) \times (J \times \{1,\dots,d_2\})}.
\]
Then \eqref{eq:hk} can be rewritten as
\begin{equation}\label{eq:ld}
y^T\Psi y\geq \psi_J^2\sum_{j\in J}y_j^T\Psi_jy_j
\end{equation}
for all $y=(y_j)_{j\in J}\in\mathbb{R}^{d_2|J|}$.
Now, one can show that
\begin{equation*}
\|\Pi_{P_j} f\|^2 = x_j^T \Psi_j^{-1} x_j,
\end{equation*}
where $x_j=(\langle \psi_{jk},f \rangle)_{k\in\{1,\dots,d_2\}}$, and that
\begin{equation*}
\|\Pi_P f\|^2 = x^T \Psi^{-1} x,
\end{equation*}
where $x=(x_j)_{j\in J}$.
Applying \cite[Lemma 2.1]{E} and \eqref{eq:ld}, we conclude that
\begin{equation*}
\|\Pi_P f\|^2 = x^T \Psi^{-1} x\leq \psi_J^{-2}\sum_{j\in J}x_j^T\Psi_j^{-1}x_j=\psi_J^{-2}\sum_{j\in J}\|\Pi_{P_j} f\|^2,
\end{equation*}
and the claim follows. 
\end{proof}

\begin{lemma}\label{pr} Suppose that Assumptions (B1) and (B4) hold. Then, for $k=1,\dots,d_1$, we have
\[
\|\Pi_{J_k}b_{1k}\|\leq \frac{C}{\psi}\sqrt{\frac{s_1}{d_1}}.
\]
In particular, for every $g_1\in V_1$ satisfying $\operatorname{supp}(g_{1})\subseteq I_{1k'}$ for some $k'\in\{1,\dots,m_1\}$, we have 
\begin{equation*}
\|\Pi_{-1}g_1\|\leq C\bigg(\frac{1}{\psi}\sqrt{\frac{s_1}{d_1}}+\sqrt{\frac{s_1d}{n}}\bigg)  \|g_1\|.
\end{equation*}
\end{lemma}

\begin{proof}
Since $|J_k|\leq s_1$, Lemmas \ref{lp1} and \ref{lp2} give the first claim (recall the definition of $\psi$)

Moreover, by Assumption (B4) and the triangle inequality, we have $\| \Pi_{-1}b_{1k}-\Pi_{J_k}b_{1k}\|\leq C_1\sqrt{s_1d/n}$. Hence,
\begin{equation}\label{EqEvBF}
\|\Pi_{-1}b_{1k}\|\leq \|\Pi_{J_k}b_{1k}\|+\| \Pi_{-1}b_{1k}-\Pi_{J_k}b_{1k}\|\leq \frac{C}{\psi}\sqrt{\frac{s_1}{d_1}}+C_1\sqrt{\frac{s_1d}{n}}.
\end{equation}

Now, suppose that $g_1\in P_1$ satisfies $\operatorname{supp}(g_{1})\subseteq I_{1k'}$ for some $k'\in\{1,\dots,m_1\}$. Setting for brevity
\[
\sum_{a\in(k')}=\sum_{a=k'(t_1+1)+1}^{k'(t_1+1)+t_1+1},
\]
we can write
\begin{equation*}
g_1=\sum_{a\in(k')}\alpha_{a}b_{1a}.
\end{equation*}
By the Cauchy-Schwarz inequality and \eqref{EqEvBF}, we have
\begin{align*}
\|\Pi_{-1}Lg_1\|_n&\leq \sum_{a\in (k')}\alpha_a\|\Pi_{-1} b_{1a}\|_n\leq \sqrt{t_1+1}\|\alpha\|_2\left( \frac{C}{\psi}\sqrt{\frac{s_1}{d_1}}+C_1\sqrt{\frac{s_1d}{n}}\right),
\end{align*}
and the second claim follows from inserting $\|\alpha\|_2\leq c_1^{-1/2} \|g_1\|$.
\end{proof}

\begin{proof}[Proof of Proposition \ref{prem}]
By the triangular inequality, we have
\begin{align*}
\| \hat{\Pi}^L_{-1}b_{1k}\|_n\leq \|\Pi_{J_k}b_{1k}\|_n+\| \Pi_{-1}b_{1k}-\Pi_{J_k}b_{1k}\|_n+ \| \hat{\Pi}^L_{-1}b_{1k}-\Pi_{-1}b_{1k}\|_n.
\end{align*}
From now on, suppose that $\EventEta\cap \EventNu$ holds with $\eta\geq \sqrt{d/n}$ and $\nu\leq 1/2$. Then, using Lemma \ref{pr}, the first term can be bounded by
\begin{equation}\label{eq:lk2}
\|\Pi_{J_k}b_{1k}\|_n\leq {\sqrt{2}}\|\Pi_{J_k}b_{1k}\|\leq \frac{\sqrt{2}C}{\psi}\sqrt{\frac{s_1}{d_1}}.
\end{equation}
Moreover, using the decomposition of Assumption (B4), the second term can be bounded by
\begin{align}\label{eq:lk1}
&\| \Pi_{-1}b_{1k}-\Pi_{J_k}b_{1k}\|_n\leq \sum_{j=2}^q\|p_j\|_n\leq \sqrt{2}\sum_{j=2}^q\|p_j\|\leq \sqrt{2}C_1\sqrt{s_1}\sqrt{\frac{d}{n}}\leq C\sqrt{s_1}\eta\nonumber.
\end{align}
Finally, using Proposition \ref{lassoest}, we have $\| \hat{\Pi}^L_{-1}b_{1k}-\Pi_{-1}b_{1k}\|_n\leq (C/\phi)\sqrt{s_1}\eta$. Hence,
\begin{equation}\label{EqEvBFEmp}
\| \hat{\Pi}^L_{-1}b_{1k}\|_n\leq \frac{C}{\psi}\sqrt{\frac{s_1}{d_1}}+\frac{C}{\phi}\sqrt{s_1}\eta.
\end{equation}
Applying the triangular inequality, the Cauchy-Schwarz inequality, and \eqref{EqEvBFEmp} we get (similar as in the proof of Lemma \ref{pr})
\begin{align*}
\|\hat{\Pi}_{-1}^Lg_1\|_n\leq \sqrt{t_1+1}\|\alpha\|_2\left( \frac{C}{\psi}\sqrt{\frac{s_1}{d_1}}+\frac{C}{\phi}\sqrt{s_1}\eta\right),
\end{align*}
with $g_1=\sum_{a\in(k')}\alpha_{a}b_{1a}$.
If $\mathcal{E}_{\nu,1}$ holds ({recall that $\nu\leq 1/2$}), then we have 
\begin{equation}\label{eq:encpn}
\|\alpha\|_2^2\leq(1/c_1) \|g_1\|^2\leq (1/(c_1(1-\nu))\|g_1\|_n^2\leq (2/c_1)\|g_1\|_n^2,
\end{equation}
and the claim follows.
\end{proof}

\subsection{Proof of Proposition \ref{empangle}}\label{Proofempangle}
Let 
\[
g_1=\sum_{k=1}^{d_1}\alpha_kb_{1k}\in P_1.
\]
By Assumption (B2), we have $\|\Pi_{-1}g_1\|^2\leq \rho_0^2\| g_1\|^2$. Thus on $\mathcal{E}_{\nu,1}$,
\[
\|\Pi_{-1}g_1\|^2\leq \rho_0^2\| g_1\|_n^2/(1-\nu)\leq (1+2\nu)\rho_0^2\| g_1\|_n^2,
\]
where we used that $\nu\leq 1/2$.
Hence,
\begin{align}
\| \hat{\Pi}_{-1}^Lg_1\|_n^2&=\| \Pi_{-1}g_1\|^2+\| \hat{\Pi}_{-1}^Lg_1\|_n^2-\| \Pi_{-1}g_1\|^2\nonumber\\
&\leq \rho_0^2\| g_1\|_n^2+2\nu\| g_1\|_n^2+\| \hat{\Pi}_{-1}^Lg_1\|_n^2-\| \Pi_{-1}g_1\|^2\label{eq:tz1},
\end{align}
and it remains to consider the last two terms. Now,
\begin{equation*}
\| \hat{\Pi}_{-1}^Lg_1\|_n^2=\sum_{k=1}^{d_1}\sum_{l=1}^{d_1}
\alpha_k\alpha_l\langle \hat{\Pi}_{-1}^Lb_{1k},\hat{\Pi}_{-1}^Lb_{1l} \rangle_n,
\end{equation*}
\begin{equation*}
\| \Pi_{-1}g_1\|^2=\sum_{k=1}^{d_1}\sum_{l=1}^{d_1}
\alpha_k\alpha_l\langle \Pi_{-1}b_{1k},\Pi_{-1}b_{1l} \rangle,
\end{equation*}
and thus
\begin{align}
&\| \hat{\Pi}_{-1}^Lg_1\|_n^2-\| \Pi_{-1}g_1\|^2\nonumber\\
&=\sum_{k=1}^{d_1}\sum_{l=1}^{d_1}
\alpha_k\alpha_l( \langle \hat{\Pi}_{-1}^Lb_{1k},\hat{\Pi}_{-1}^Lb_{1l} \rangle_n-\langle \Pi_{J_k}b_{1k},\Pi_{J_l}b_{1l} \rangle_n)\label{eq:o1}\\
&+\sum_{k=1}^{d_1}\sum_{l=1}^{d_1}
\alpha_k\alpha_l\left( \langle \Pi_{J_k}b_{1k},\Pi_{J_l}b_{1l} \rangle_n-\langle \Pi_{J_k}b_{1k},\Pi_{J_l}b_{1l} \rangle\right) \label{eq:o2}\\
&+\sum_{k=1}^{d_1}\sum_{l=1}^{d_1}
\alpha_k\alpha_l\left( \langle \Pi_{J_k}b_{1k},\Pi_{J_l}b_{1l} \rangle-\langle \Pi_{-1}b_{1k},\Pi_{-1}b_{1l}\rangle\right)\label{eq:o3}.
\end{align}
First, consider the term \eqref{eq:o1}. Using the identity 
\[
\langle a',b' \rangle_n-\langle a,b \rangle_n= \langle a'-a,b'-b \rangle_n+\langle a,b'-b \rangle_n+\langle a'-a,b \rangle_n,
\]
we get
\begin{align*}
&\langle \hat{\Pi}_{-1}^Lb_{1k},\hat{\Pi}_{-1}^Lb_{1l} \rangle_n-\langle \Pi_{J_k}b_{1k},\Pi_{J_l}b_{1l} \rangle_n\\
&= \langle \hat{\Pi}_{-1}^Lb_{1k}-\Pi_{J_k}b_{1k},\hat{\Pi}_{-1}^Lb_{1l}-\Pi_{J_l}b_{1l} \rangle_n\\
&+\langle \Pi_{J_k}b_{1k},\hat{\Pi}_{-1}^Lb_{1l}-\Pi_{J_l}b_{1l} \rangle_n+\langle \hat{\Pi}_{-1}^Lb_{1k}-\Pi_{J_k}b_{1k},\Pi_{J_l}b_{1l} \rangle_n.
\end{align*}
Plugging in the formulas \eqref{eq:lk2} and  Proposition \ref{lassoest}, we get on $\EventEta\cap \EventNu$,
\begin{equation*}
|\langle \hat{\Pi}_{-1}^Lb_{1k},\hat{\Pi}_{-1}^Lb_{1l} \rangle_n-\langle \Pi_{J_k}b_{1k},\Pi_{J_l}b_{1l} \rangle_n|\leq C\left( \frac{1}{\phi\psi}\frac{s_1\eta}{\sqrt{d_1}}+\frac{s_1\eta^2}{\phi^2} \right).
\end{equation*}
Hence, if $\EventEta\cap \EventNu$ holds, then the term \eqref{eq:o1} can be bounded by
\begin{align*}
&C\sum_{k=1}^{d_1}\sum_{l=1}^{d_1}
|\alpha_k||\alpha_l|\left( \frac{1}{\phi\psi}\frac{s_1\eta}{\sqrt{d_1}}+\frac{s_1\eta^2}{\phi^2} \right) \leq C\left( \frac{1}{\phi\psi}\frac{s_1\sqrt{d_1}\eta}{\sqrt{d_1}}+\frac{s_1d_1\eta^2}{\phi^2} \right)\|\alpha\|_2^2,
\end{align*}
where we applied the Cauchy-Schwarz inequality in the last step.
Next, the last term in \eqref{eq:o3} can be bounded similarly. As above, we have
\begin{align*}
&\langle \Pi_{-1}b_{1k},\Pi_{-1}b_{1l} \rangle-\langle \Pi_{J_k}b_{1k},\Pi_{J_l}b_{1l} \rangle\\
&= \langle \Pi_{-1}b_{1k}-\Pi_{J_k}b_{1k},\Pi_{-1}b_{1l}-\Pi_{J_l}b_{1l} \rangle\\
&+\langle \Pi_{J_k}b_{1k},\Pi_{-1}b_{1l}-\Pi_{J_l}b_{1l} \rangle+\langle \Pi_{-1}b_{1k}-\Pi_{J_k}b_{1k},\Pi_{J_l}b_{1l} \rangle
\end{align*}
and thus, using Assumption (B4), Lemma \ref{pr}, and $\eta\geq\sqrt{d/n}$,
\begin{equation*}
|\langle \Pi_{-1}b_{1k},\Pi_{-1}b_{1l} \rangle-\langle \Pi_{J_k}b_{1k},\Pi_{J_l}b_{1l} \rangle|\leq \frac{C}{\psi} \frac{s_1\eta}{\sqrt{d_1}}+C_1^2s_1\eta^2.
\end{equation*}
Hence, the term \eqref{eq:o3} can be bounded by
\begin{equation*}
\left( \frac{C}{\psi} s_1\sqrt{d_1}\eta+C_1^2s_1d_1\eta^2\right)\|\alpha\|_2^2.
\end{equation*}
Finally, consider the middle term \eqref{eq:o2}. 
If $\EventNu$ holds, then 
\begin{equation*}
4|\langle \Pi_{J_k}b_{1k},\Pi_{J_l}b_{1l} \rangle_n-\langle \Pi_{J_k}b_{1k},\Pi_{J_l}b_{1l} \rangle|\leq 4\nu\|\Pi_{J_k}b_{1k}\|\|\Pi_{J_l}b_{1l}\|\leq C\frac{s_1\nu}{\psi^2},
\end{equation*} 
where we applied Lemma \ref{pr} in the last inequality. Hence, \eqref{eq:o2} can be bounded by $(C/\psi^2)\nu s_1\|\alpha\|_2^2$. We conclude that
\begin{equation*}
\| \hat{\Pi}_{-1}^Lg_1\|_n^2\leq \rho_0^2\| g_1\|_n^2+2\nu\| g_1\|_n^2+C\left(\frac{s_1\nu}{\psi^2}+\frac{s_1\sqrt{d_1}\eta}{\psi\phi}+\frac{s_1 d_1\eta^2}{\phi^2} \right)\|\alpha\|_2^2,
\end{equation*}
and the first claim follows from inserting \eqref{eq:encpn}. Inserting the choices in \eqref{EqChoiceLambdaEta} and \eqref{EqChoiceNu} into the first claim and using the first inequality in \eqref{EqThmLCond} with $c$ small enough we get $\|\hat{\Pi}_{-1}^L\hat\Pi_1\|_{\operatorname{op}}^2 \leq \rho_0^2+(1/4)(1-\rho_0)^2$ from which the second claim follows by taking the square root.\qed

\subsection{Proof of Proposition \ref{mthmlowbounds}}\label{AppendixBadEvents}
The proof of Proposition \ref{mthmlowbounds} is relatively standard.  References in which similar tools are used are, e.g., \cite{Giraud}, \cite{MGB} and \cite{BouchLugMass}, which we follow closely. 
\label{AppendixE}
\subsubsection{Nonparametric Lasso events}\label{nle}
Setting
\[\mathcal{A}_0=\Big\{2\max_{j=1,\dots,q}\sup_{0\neq g_j\in V_j} \frac{|\langle \epsilon,g_j\rangle_n|}{\|g_j\|_n} \leq \lambda\Big\},\]
we have the following standard result (see, e.g. \cite[Chapter 4]{Giraud}):
\begin{proposition}\label{gaussconc} For $x> 0$ and 
\begin{equation*}
\lambda=2\sigma\sqrt{\frac{d}{n}}+2\sigma\sqrt{\frac{2x+2\log q}{n}},
\end{equation*}
we have $\mathbb{P}\left( \mathcal{A}_0^c\right)\leq \exp(-x)$.
\end{proposition}
\begin{proof}
Let us outline the main steps in the proof (since the proofs in the next section follow similar steps). First, applying the Cauchy-Schwarz inequality, one can show that
\begin{equation*}
\mathbf{E}_\epsilon\left[ \sup_{0\neq g_j\in V_j} \frac{|\langle \boldsymbol{\epsilon},g_j\rangle_n|}{\|g_j\|_n}\right]=\mathbf{E}_\epsilon\left[ \sup_{0\neq g_j\in V_j} \frac{\langle \boldsymbol{\epsilon},g_j\rangle_n}{\|g_j\|_n}\right] \leq \sigma\sqrt{\frac{d_j}{n}}.
\end{equation*}
Combining this with the Gaussian concentration inequality (see, e.g., \cite[Theorem 5.6]{BouchLugMass}), we obtain
\begin{equation*}
\mathbf{P}_\epsilon\left(  \sup_{0\neq g_j\in V_j} \frac{\langle \boldsymbol{\epsilon},g_j\rangle_n}{\|g_j\|_n}>\sigma\sqrt{\frac{d_j}{n}}+\sigma\sqrt{\frac{2x}{n}}\right) \leq \exp(-x),
\end{equation*}
where $\mathbf{P}_\epsilon$ can be also replaced by $\mathbb{P}$. Finally, we apply the following lemma which is a consequence of the union bound:
\begin{lemma}\label{maxl}
Suppose that $Z_1,\dots,Z_q$ are random variables satisfying 
\begin{equation*}
\mathbb{P}\left(Z_j\geq m_j +R_j\sqrt{\frac{x}{n}}+K_j\frac{x}{n} \right) \leq \exp(-x).
\end{equation*}
Then for $m=\max_j m_j$, $R=\max_j R_j$, and $K=\max_j K_j$, we have
\begin{equation*}
\mathbb{P}\left(\max_j Z_j\geq m +R\sqrt{\frac{x+\log q}{n}}+K\frac{x+\log q}{n} \right) \leq \exp(-x).
\end{equation*}
\end{lemma}
\end{proof}
For the nonparametric Lasso event $\mathcal{A}_l$ defined in Section \ref{AppendixLasso}, we have the following result:
\begin{proposition}\label{cinlp} Suppose that Assumptions (B1) and (B4) hold. For $x>0$, let 
\begin{equation*}
\eta=C\left( \sqrt{\frac{d(x+\log d_1+\log q)}{n}}+\frac{\sqrt{s_1}d (x+\log d_1+\log q)}{\psi n} \right),
\end{equation*}
for some large enough constant $C$ (given explicitly in the proof).
Then
\begin{equation*}
\mathbb{P}\left(\mathcal{E}_{1,\nu}\cap \bigcup_{l=1}^{d_1}\mathcal{A}_l^c\right)\leq \exp(-x).
\end{equation*}
\end{proposition}
 The main arguments of the proof are the same as in the previous section. The main difference is the replacement of the Gaussian concentration inequality with the following refinement of Talagrand's inequality obtained by Bousquet (see, e.g., \cite[Equation (5.50)]{M}).
\begin{thm}\label{talineq} Consider $n$ independent and identically distributed
random variables $X^1,\dots,X^n$ taking values in some measurable space $(S,\mathcal{B})$. Let $\mathcal{G}$ be a countable family of real-valued measurable functions on $(S,\mathcal{B})$ that are uniformly bounded by some constant $b$. Let 
\[Z=\sup_{g\in\mathcal{G}}\left|\frac{1}{n}\sum_{i=1}^ng(X^i)-\mathbb{E}\left[g(X^i)\right]\right|
\] and $v=\sup_{g\in\mathcal{G}}\mathbb{E}\left[ g^2(X^1)\right]$. Then for each $x>0$, 
\begin{equation*}
\mathbb{P}\left( Z \geq 2\mathbb{E}\left[Z\right]+\sqrt{\frac{2vx}{n}}+\frac{4}{3}\frac{bx}{n}\right)\leq \exp\left(-x\right).
\end{equation*}
\end{thm}
We introduce
\[
\varphi=\max_j\sup_{0\neq g_j\in P_j}\frac{1}{\sqrt{d_j}}\frac{\|g_j\|_\infty}{\|g_j\|}.
\]  
Using \eqref{eq:lpinfty}, one can show that $\varphi\leq \sqrt{2 (t+1)/c_1}$, i.e.~$\varphi$ can be absorbed into $C$.
In order to apply Talagrand's inequality, we start with two lemmas. 
\begin{lemma}\label{all2} Suppose that Assumptions (B1) and (B4) hold. For $l=1,\dots,d_1$ and $j=2,\dots,q$, we have
\begin{equation*}
\mathbb{E}\left[ \sup_{0\neq g_j\in V_j} \frac{\langle b_{1l}-\Pi_{-1}b_{1l},g_j\rangle_n}{\|g_j\|}\right] \leq \sqrt{\frac{\varphi^2}{c_1}}\sqrt{\frac{d_j}{n}}.
\end{equation*}
\end{lemma}

\begin{proof}
Let $\psi_{j1},\dots,\psi_{jd_j}$ be an orthonormal basis of $P_j$ with respect to $\|\cdot\|$. Then, using the Cauchy-Schwarz inequality, we have
\begin{align*}
&\sup_{g_j\in V_j:\|g_j\|\leq 1} \langle b_{1l}-\Pi_{-1}b_{1l},g_j\rangle_n\\
=&\sup_{a\in\mathbb{R}^{d_j}:\|a\|_2\leq 1}\sum_{k=1}^{d_j}a_k\langle b_{1l}-\Pi_{-1}b_{1l},\psi_{jk}\rangle_n \leq\left(  \sum_{k=1}^{d_j}\langle b_{1l}-\Pi_{-1}b_{1l},\psi_{jk}\rangle_n^2\right)^{1/2}.
\end{align*}
Hence, by the Cauchy-Schwarz inequality,
\begin{align*}
&\mathbb{E}\left[\sup_{g_j\in V_j:\|g_j\|\leq 1} \langle b_{1l}-\Pi_{-1}b_{1l},g_j\rangle_n\right] 
\\&\leq \left( \mathbb{E}\left[\sum_{k=1}^{d_j}\langle b_{1l}-\Pi_{-1}b_{1l},\psi_{jk}\rangle_n^2 \right] \right)^{1/2} \\
&=\left( \frac{1}{n}\sum_{k=1}^{d_j}\mathbb{E}\left[ \left( (b_{1l}-\Pi_{-1}b_{1l})(X)\psi_{jk}(X_j)\right) ^2\right] \right)^{1/2}.
\end{align*}
Since $\|b_{1l}-\Pi_{-1}b_{1l}\|\leq \|b_{1l}\|\leq  1/\sqrt{c_1}$, this is bounded by (cf. \cite[Lemma 1]{BM})
\begin{equation*}
\frac{1}{\sqrt{c_1n}}\left\|\sum_{k=1}^{d_j}\psi_{jk}^2\right\|_{\infty}^{1/2}\leq \sqrt{\frac{\varphi^2}{c_1}}\sqrt{\frac{d_j}{n}},
\end{equation*}
and the claim follows.
\end{proof}

\begin{lemma}\label{all1} Suppose that Assumption (B1) holds. For $j=2,\dots,q$ and each $g_j\in P_j$ with $\|g_j\|\leq 1$, we have 
\begin{equation*}
\|(b_{1l}-\Pi_{-1}b_{1l})g_j\| \leq \varphi\sqrt{d/c_1},\qquad
\|(b_{1l}-\Pi_{-1}b_{1l})g_j\|_\infty\leq \frac{C}{\psi}\sqrt{s_1}d.
\end{equation*}

\end{lemma}

\begin{proof}
By the definition of $\varphi$ and the bound $\|b_{1l}-\Pi_{-1}b_{1l}\|\leq 1/\sqrt{c_1}$, we get
$\|(b_{1l}-\Pi_{-1}b_{1l})g_j\|\leq\|g_j\|_\infty\|b_{1l}-\Pi_{-1}b_{1l}\|\leq \varphi\sqrt{d/c_1}$. This gives the first claim. For the second claim note that $\|g_j\|_\infty\leq \varphi\sqrt{d}$ and $\|b_{1l}\|_\infty\leq \varphi\sqrt{d_1}\|b_{1l}\|\leq \varphi\sqrt{d/c_1}$. Thus the claim follows if we can show that
\begin{equation}\label{all1helper}
\|\Pi_{-1}b_{1l}\|_\infty\leq \left(\varphi/(\sqrt{c_1}\psi)+C_1\varphi\right) \sqrt{s_1d}. 
\end{equation}
Letting $\Pi_{J_l}b_{1l}=\sum_{j\in J_l} g_j$, $g_j\in P_j$ we get
\begin{align*}
\|\Pi_{J_l}b_{1l}\|_\infty\leq \sum_{j\in J_l}\|g_j\|_\infty &\leq \varphi\sqrt{d}\sum_{j\in J_l}\|g_j\|\leq \varphi\sqrt{s_1d}\left( \sum_{j\in J_l}\|g_j\|^2\right)^{1/2}
\end{align*}
and thus by the definition of $\psi$,
\begin{align*}
\|\Pi_{J_l}b_{1l}\|_\infty \leq \frac{\varphi\sqrt{s_1d}}{\psi}\|\Pi_{J_l}b_{1l}\| \leq \frac{\varphi\sqrt{s_1d}}{\sqrt{c_1}\psi}.
\end{align*}
Moreover, by Assumption (B4) and the inequality $d\leq n$, we have
\begin{align*}
\|\Pi_{J_k}b_{1l}-\Pi_{-1}b_{1l}\|_\infty\leq \sum_{j=2}^q\|p_j\|_\infty
&\leq \varphi\sqrt{d}\sum_{j=2}^q\|p_j\|\leq C_1\varphi\sqrt{s_1d}.
\end{align*}
Hence, \eqref{all1helper} follows from the last two inequalities and the triangle inequality. This completes the proof.
\end{proof}
Combining Lemmas \ref{all2} and \ref{all1} with Theorem \ref{talineq}, we get for $x>0$,
\begin{multline*}
\mathbb{P}\left( \sup_{0\neq g_j\in V_j} \frac{\langle b_{1l}-\Pi_{-1}b_{1l},g_j\rangle_n}{\|g_j\|} \geq\right.\\
 \left.\left( 2\sqrt{\frac{\varphi^2d}{c_1n}}+\sqrt{\frac{2\varphi^2dx}{c_1n}}+\frac{4C \sqrt{s_1}d x}{3\psi n}\right)\right) 
\leq \exp\left(-x\right).
\end{multline*}
Applying Lemma \ref{maxl} and the bound
\begin{align*}
&\mathbb{P}\left(\mathcal{E}_{\nu,1}\cap \max_j\sup_{0\neq g_j\in V_j} \frac{\langle b_{1l}-\Pi_{-1}b_{1l},g_j\rangle_n}{\|g_j\|_n}>y\right)\\
&\leq  \mathbb{P}\left( \max_j\sup_{0\neq g_j\in V_j} \frac{\langle b_{1l}-\Pi_{-1}b_{1l},g_j\rangle_n}{\|g_j\|}>\sqrt{1-\nu}y\right),
\end{align*}
where $y>0$, we conclude that (recall that $\nu\leq 1/2$)
\begin{multline*}
\mathbb{P}\left(\mathcal{E}_{\nu,1}\cap \max_j\sup_{0\neq g_j\in V_j} \frac{\langle b_{1l}-\Pi_{-1}b_{1l},g_j\rangle_n}{\|g_j\|_n} \geq\right.\\
 \left.C\left( \sqrt{\frac{d}{n}}+\sqrt{\frac{d(x+\log q)}{n}}+\frac{ \sqrt{s_1}d (x+\log q)}{\psi n}\right)\right) 
\leq \exp\left(-x\right).
\end{multline*}
Thus we have established an upper bound for $\mathbb{P}(\mathcal{E}_{\nu,1}\cap\mathcal{A}^c_l)$. Replacing $x$ by $x+\log d_1$, this gives the upper bound for the probability of the union of the $\mathcal{A}_l^c$ stated in  Proposition \ref{cinlp}.

\subsubsection{Compatibility condition events}
Compatibility condition events have been considered in \cite{MGB}, and their analysis can be applied well in our setting. By a slight abuse of notation, we define for $(g_1,\dots,g_q)\in (P_1,\dots,P_q)$,
\[
\pen(g)=\sum_{j=1}^q\|g_j\|
\]
For $\mu\leq 1/2$, we define
\begin{equation*}
\mathcal{E}_\mu=\left\lbrace \sup_{g_j\in P_j,g=\sum_{j}g_j} \frac{\left|\|g\|_n^2-\|g\|^2\right|}{\pen(g)^2}\leq \mu\right\rbrace 
\end{equation*}
(with the convention $0/0=0$). The following result is Theorem 2 of \cite{MGB}.
\begin{proposition}\label{sparsityboundmu} 
Suppose that $\mu\leq 1/2$ and
\begin{equation}\label{eq:eqmbg}
\frac{1152\max(s_0,s_1)\mu}{\phi^2}\leq 1.
\end{equation}
Then on $\mathcal{E}_\mu$, the theoretical compatibility condition stated in Assumption (B5) implies their empirical counterparts from Section \ref{SecEvents}, i.e. we have  $\mathcal{E}_\mu\subseteq \mathcal{E}_{\phi,J_k}$ for $k=0,\dots,d_1$.
\end{proposition}
\begin{proof}
We only prove the case $k=0$, since the remaining claims follow analogously.
Suppose that 
\begin{equation*}
\sum_{j=1}^q\|g_j\|_n\leq 8\sum_{j\in J_0}\|g_j\|_n.
\end{equation*}
If $\mathcal{E}_\mu$ holds, then $g$ satisfies
\begin{equation*}
\sum_{j=1}^q\|g_j\|\leq 8\sqrt{\frac{1+\mu}{1-\mu}}\sum_{j\in J_0}\|g_j\|\leq 8\sqrt{3}\sum_{j\in J_0}\|g_j\|
\end{equation*}
and thus we can apply the first implication of (B5). Moreover, we have on $\mathcal{E}_\mu$,
\begin{align*}
\pen(g)^2&\leq 192s_0\sum_{j\in J_0}\|g_j\|^2\leq 384s_0\sum_{j\in J_0}\|g_j\|_n^2.
\end{align*}
Hence on $\mathcal{E}_\mu$,
\begin{align*}
\sum_{j\in J_0}\|g_j\|_n^2 \leq (3/2)\sum_{j\in J_0}\|g_j\|^2&\leq (3/2)\|g\|^2/\phi^2\\
&\leq (3/2)\left( \|g\|_n^2+\mu\pen(g)^2\right)/\phi^2 \\
&\leq (3/2)\|g\|_n^2/\phi^2+ \frac{1152s_0\mu}{2\phi^2}\sum_{j\in J_0}\|g_j\|_n^2.
\end{align*}
By Assumption \eqref{eq:eqmbg}, we conclude that on $\mathcal{E}_\mu$,
\begin{equation*}
\sum_{j\in J_0}\|g_j\|_n^2\leq 3\|g\|_n^2/\phi^2,
\end{equation*}
and the claim follows.
\end{proof}
It remains to bound the probability of $\mathcal{E}_\mu$, where we again follow \cite{MGB} closely.
\begin{proposition}\label{cicce} For $x>0$ and
\begin{equation*}
\mu=64\left( \frac{\varphi^2 d}{\sqrt{n}}+\frac{\varphi^2d(x+\log q)}{3n} + \sqrt{\frac{\varphi^2d(x+\log q)}{n}}\right),
\end{equation*}
wehave $\mathbb{P}\left(\mathcal{E}_\mu^c\right)\leq \exp(-x)$.
\end{proposition}
We start with the following lemma:
\begin{lemma}\label{dget} We have
\begin{align*}
&\mathbb{E}\left[\sup_{g_j\in P_j,g=\sum_{j}g_j} \frac{\left|\|g\|^2-\|g\|^2_n\right|}{\pen(g)^2}\right]\\
&\leq \frac{32\varphi^2d}{\sqrt{n}}+\frac{16\varphi^2d(1+\log q)}{3n} + 16\sqrt{\frac{\varphi^2d(1+\log q)}{n}}.
\end{align*}
\end{lemma}

\begin{proof}[Proof of Lemma \ref{dget}]
By the symmetrization and contraction principle (see, e.g., \cite[Chapter 11.3]{BouchLugMass}), we have 
\begin{align*}
&\mathbb{E}\left[\sup_{g_j\in P_j,g=\sum_{j}g_j} \frac{\left|\|g\|^2-\|g\|^2_n\right|}{\pen(g)^2}\right]\\
&\leq 2\mathbb{E}\left[\sup _{g_j\in P_j,g=\sum_{j}g_j}\frac{\left|1/n\sum_{i=1}^n\sigma^ig^2(X^i)\right|}{\pen(g)^2}\right]\\
&\leq 8\varphi\sqrt{d}\mathbb{E}\left[\sup _{g_j\in P_j,g=\sum_{j}g_j}\frac{\left|1/n\sum_{i=1}^n\sigma^ig(X^i)\right|}{\pen(g)}\right]
\end{align*}
where $(\sigma^i)$ is a Rademacher sequence independent of $X^i$. Note that for the contraction principle, we used that
\[
\|g\|_\infty \leq \sum_{j=1}^p\|g_j\|_\infty\leq \varphi\sqrt{d}\pen(g).
\]
We continue, writing
\begin{align*}
&\mathbb{E}\left[\sup _{g_j\in P_j,g=\sum_{j}g_j}\frac{\left|1/n\sum_{i=1}^n\sigma^ig(X^i)\right|}{\pen(g)}\right]\\
&\leq \mathbb{E}\left[\sup _{g_j\in P_j,g=\sum_{j}g_j}\sum_{j=1}^p\frac{\|g_j\|}{\pen(g)}\frac{\left|1/n\sum_{i=1}^n\sigma^ig_j(X_j^i)\right|}{\|g_j\|}\right]\\
&\leq \mathbb{E}\left[\max_j\sup_{0\neq g_j\in P_j}\frac{\left|1/n\sum_{i=1}^n\sigma^ig_j(X_j^i)\right|}{\|g_j\|}\right].
\end{align*}
Now, following the proof of Lemma \ref{all2}, letting $\psi_{j1},\dots\psi_{jd_j}$ be an orthonormal basis of $P_j$ with respect to $\|\cdot\|$, we have that
\begin{align*}
\mathbb{E}\left[ \sup_{0\neq g_j\in P_j}\frac{\left|1/n\sum_{i=1}^n\sigma^ig_j(X_j^i)\right|}{\|g_j\|}\right] &  \leq \left( \frac{1}{n}\sum_{k=1}^{d_j} \mathbb{E}\left[ (\sigma^1 \psi_{jk}(X_j))^2\right]  \right)^{1/2} \leq \sqrt{\frac{\varphi^2d}{n}}.
\end{align*}
We now apply \cite[Lemma 13]{MGB} to get the result.
\end{proof}
\begin{proof}[Proof of Proposition \ref{cicce}] 
The claim follows from Theorem \ref{talineq}, Lemma \ref{dget}, and the use of the bounds
\begin{equation*}
\left\|\frac{g^2}{\pen(g)^2}\right\|_\infty\leq \varphi^2d,\qquad
\left\|\frac{g^2}{\pen(g)^2}\right\|\leq \varphi\sqrt{d}.
\end{equation*}
\end{proof}

\subsubsection{Empirical norm approximation events}\label{empnormapprox} We have
\begin{proposition}\label{empnormconcev} Suppose that Assumption (B1) holds. For $x>0$, let 
\begin{equation*}
\nu=\frac{C}{\psi}\sqrt{\frac{s_1d(x+2\log (2d)+\log q)}{n}},
\end{equation*}
for some large enough constant $C$. If $\nu\leq 1/2$, then we have $\mathbb{P}\left(\EventNu^c\right) \leq \exp(-x)$.
\end{proposition}
\begin{proof}
We first consider the second part of the event $\EventNu$. In the proof of Lemma \ref{all1}, we have shown that
\begin{equation}\label{EqLInftyL2}
\|\Pi_{J_l}b_{1l}\|_\infty  \leq \frac{\varphi\sqrt{s_1d}}{\psi}\|\Pi_{J_l}b_{1l}\|.
\end{equation}
Applying this and  Bernstein's inequality (see, e.g, \cite[Equation (2.10)]{BouchLugMass}), we get
\begin{multline*}
\mathbb{P}\left( \left|\langle \Pi_{J_k}b_{1k},\Pi_{J_l}b_{1l}\rangle_n-\langle \Pi_{J_k}b_{1k},\Pi_{J_l}b_{1l}\rangle\right|> \nu \|\Pi_{J_k}b_{1k}\|\|\Pi_{J_l}b_{1l}\|\right)\\
\leq 2\exp\left( -\frac{\psi^2}{C}\frac{n\nu^2}{s_1d}\right),
\end{multline*}
where we also used the fact that $\nu\leq 1/2\leq 1$. By the union bound and the choice of $\nu$, we get
\begin{multline*}
\mathbb{P}\left( \left|\langle \Pi_{J_k}b_{1k},\Pi_{J_l}b_{1l}\rangle_n-\langle \Pi_{J_k}b_{1k},\Pi_{J_l}b_{1l}\rangle\right|\right.\\
\left.> \nu \|\Pi_{J_k}b_{1k}\|\|\Pi_{J_l}b_{1l}\|\text{ for all }k,l=1,\dots,d_1\right)\\
\leq 2d_1^2\exp\left( -\frac{\psi^2}{C}\frac{n\nu^2}{s_1d}\right)\leq(1/2)\exp(-x).
\end{multline*}
We now turn to the second part, namely the event $\mathcal{E}_{\nu,1}$. Applying
\cite[Theorem 7.3]{R} (see also \cite[Theorem 7]{Wahl}) and the union bound, we
have:
\begin{equation*}
\mathbb{P}\left( \mathcal{E}_{\nu,1}^c\right) \leq 2dq\exp\left(
-\frac{1}{14\varphi^2}\frac{n\nu^2}{d}\right).
\end{equation*}
Hence, by the definition of $\nu$, we have $\mathbb{P}\left( \mathcal{E}_{\nu,1}^c\right) \leq (1/2)\exp(-x)$, and the claim follows.
\end{proof}
\subsubsection{End of proof of Proposition \ref{mthmlowbounds}}
Let $\lambda$, $\eta$ and $\nu$ be as defined in Propositions \ref{gaussconc}, \ref{cinlp} and \ref{empnormconcev}. Using the definitions of the events and Proposition \ref{sparsityboundmu}, we have
\begin{align*}
\mathbb{P}\left((\EventLambda\cap\EventEta\cap\EventNu)^c\right)
&= \mathbb{P}\left(\EventLambda^c\cup\EventEta^c\cup\EventNu^c\right)\\
&\leq\mathbb{P}\left(\mathcal{A}_0^c\cup \bigcup_{l=1}^{d_1}\mathcal{A}_l^c\cup\EventNu^c\cup\mathcal{E}_\mu^c\right)\\
&\leq \mathbb{P}\left(\mathcal{A}_0^c\right) +\mathbb{P}\left(\mathcal{E}_{\nu,1}\cap\bigcup_{l=1}^{d_1}\mathcal{A}_l^c\right)+\mathbb{P}\left(\EventNu^c\right)+\mathbb{P}\left(\mathcal{E}_\mu^c\right).
\end{align*}
Applying Propositions \ref{gaussconc}, \ref{cinlp}, \ref{cicce}, and \ref{empnormconcev}, this can be bounded by $4\exp(-x)$, and the claim follows from inserting $x=a\log n+\log 4$. \qed

\subsection{Proof of Theorem \ref{altchoicesthm}}
We give the additional arguments for the proof of Theorem \ref{altchoicesthm}.
\subsubsection{Additional events}
For $0<\nu<1$, let $\mathcal{E}_{\nu,J_1}$ be the event on which
\[
\forall g_{J_1}\in P_{J_1}\quad (1-\nu)\|g_{J_1}\|^2\leq \|g_{J_1}\|_n^2\leq (1+\nu)\|g_{J_1}\|^2
\]
with $J_1$ from (B4'). For $l=1,\dots,d_1$, let
\begin{equation*}
\mathcal{A}_l'= \Big\{2\max_{j=2,\dots,q}\sup_{g_j\in V_j:\|g_j\|\leq 1} |\langle b_{1l}-\Pi_{J_1}b_{1l},g_j\rangle_n-\langle b_{1l}-\Pi_{J_1}b_{1l},g_j\rangle| \leq\eta\Big\}.
\end{equation*}
The following proposition follows from the same line of arguments as Proposition \ref{cinlp}.
\begin{proposition}\label{PropThm3Conc1} Suppose that Assumptions (B1) and (B4') hold. For $x>1$, let 
\begin{equation*}
\eta=C\left( \sqrt{\frac{d(x+\log d_1+\log q)}{n}}+\frac{\sqrt{s_1}d (x+\log d_1+\log q)}{\psi n} \right),
\end{equation*}
for some large enough constant $C$ (given explicitly in the proof).
Then
\begin{equation*}
\mathbb{P}\left(\mathcal{E}_{1,\nu}\cap \bigcup_{l=1}^{d_1}\mathcal{A}_l'^c\right)\leq \exp(-x).
\end{equation*}
\end{proposition}
Applying
\cite[Theorem 7.3]{R} (see also \cite[Theorem 7]{Wahl}) in combination with \eqref{EqLInftyL2}, we get
\begin{proposition}\label{PropThm3Conc2}
We have
\[
\mathbb{P}\left(\mathcal{E}_{\nu,J_1}^c\right)\leq 2^{3/4}ds_1\exp\Big(-c\psi^2\frac{n\nu^2}{s_1d}\Big).
\]
Hence, with 
\begin{equation}\label{EqChoiceNu'}
\nu=\frac{C}{\psi}\sqrt{\frac{s_1d(a\log n+\log q)}{n}}< 1
\end{equation}
we have $\mathbb{P}(\mathcal{E}_{\nu,J_1}^c)\leq n^{-a}$.
\end{proposition}
\subsubsection{Fundamental properties of $\hat{\Pi}_{-1}^M$}
\begin{proposition}\label{PropJMConstr} 
Suppose that (B1), (B2), and (B4') hold. Let $\hat \Pi_{-1}^M$ be the solution of \eqref{EqConstrJMInfNorm} with $\Delta_1'$ and $\Delta_2'$ from Theorem \ref{altchoicesthm}. Suppose that $\nu$  from \eqref{EqChoiceNu'} satisfies $\nu\leq 1/2$. If $\mathcal{E}_{\nu,J_1}\cap \mathcal{E}_{\nu,1}\cap\bigcap_{l=2}^q\mathcal{A}'_l$ holds with $\nu$ from \eqref{EqChoiceNu'}, then we have 
\[
\sup_{g_1\in P_1:\|g_1\|_n\leq 1}\|\hat \Pi_{-1}^Mg_1\|_n \leq\rho_0+(1-\rho_0)/4+\frac{C}{\psi}\sqrt{\frac{s_1d(a\log n+\log q)}{n}}.
\]
In particular, if \eqref{eq:ThmCond} holds with $c$ sufficiently small, then on $\mathcal{E}_{\nu,J_1}\cap \mathcal{E}_{\nu,1}\cap\bigcap_{l=2}^q\mathcal{A}'_l$,
\[
\forall g_1\in P_1,\qquad\|\hat \Pi_{-1}^Mg_1\|_n\leq (1+\rho_0)/2\|g_1\|_n.
\]
\end{proposition}
\begin{proof}
On the event $\mathcal{E}_{\nu,J_1}$, the map $\Pi_{J_1}$ can also be considered as a map from $V_1$ to  $V_{-1}\subseteq\mathbb{R}^n$, meaning that we can consider $M=\Pi_{J_1}$ in \eqref{EqConstrJMInfNorm}. Then, for $g_1\in P_1$, we have on $\mathcal{E}_{\nu,J_1}\cap \mathcal{E}_{\nu,1}$, 
\begin{align*}
\|\Pi_{J_1}g_1\|_n\leq \sqrt{1+\nu}\|\Pi_{J_1}g_1\|
&\leq \sqrt{1+\nu}(\|\Pi_{-1}g_1\|+\|\Pi_{J_1}g_1-\Pi_{-1}g_1\|)\\
&\leq \sqrt{1+\nu}\big(\rho_0+(1-\rho_0)/4\big)\|g_1\|\\
&\leq \frac{\sqrt{1+\nu}}{\sqrt{1-\nu}}\big(\rho_0+(1-\rho_0)/4\big)\|g_1\|_n,
\end{align*}
where we applied (B2) and (B4)' in the third inequality.  Since $\nu\leq 1/2$ and since $\rho_0+(1-\rho_0)/4\leq 1$, we get for $g_1\in P_1$ with $\|g_1\|_n\leq 1$,
\[
\|\Pi_{J_1}g_1\|_n\leq \rho_0+(1-\rho_0)/4+\frac{C}{\psi}\sqrt{\frac{s_1d(a\log n+\log q)}{n}}.
\]
Hence, it suffices to show that $\Pi_{J_1}$ satisfies the constraints (i) and (ii) from \eqref{EqConstrJMInfNorm}. First, by Lemma \ref{pr}, we have $\|\Pi_{J_1}b_{1k}\|_n\leq \sqrt{1+\nu}\|\Pi_{J_1}b_{1k}\|\leq \Delta_2'$ for all $k=1,\dots,d_1$, provided that $C$ in the definition of $\Delta_2'$ is chosen large enough. It remains to show that on the event $\mathcal{E}_{\nu,1}\cap\bigcap_{l=2}^q\mathcal{A}_l$, (i) holds with $\Delta_1'$ from Theorem \ref{altchoicesthm}. First, proceeding as in the proof of Theorem \ref{ourchoicesthm}, we have on  $\mathcal{E}_{\nu,1}$,
\begin{align}
    &\|(\hat \Pi_1 - \Pi_{J_1} \hat \Pi_1)^T \mathbf{g}_{j}\|_{n,\infty}\label{EqPropJMConstr} \\
    &\leq  C\sqrt{d_1}\max_{1\leq k \leq d_1}|\langle g_{j} , b_{1k} - \Pi_{J_1} b_{1k}\rangle_n|\nonumber\\
    &\leq  C\sqrt{d_1}\max_{1\leq k \leq d_1}|\langle g_{j} , b_{1k} - \Pi_{J_1} b_{1k}\rangle_n-\langle g_{j} , b_{1k} - \Pi_{J_1} b_{1k}\rangle|\nonumber\\ 
    &\ \ \ +C\sqrt{d_1}\max_{1\leq k \leq d_1}|\langle g_{j} ,\Pi_{-1} b_{1k} - \Pi_{J_1} b_{1k}\rangle|\nonumber
\end{align}
for all $g_j\in V_j$ and all $j=2,\dots,q$, where we used the identity $\langle g_{j} ,\Pi_{-1} b_{1k} -  b_{1k}\rangle=0$ in the last inequality. By (B4'), we have 
\begin{align}
|\langle g_{j} ,\Pi_{-1} b_{1k} - \Pi_{J_1} b_{1k}\rangle|&\leq \|g_{j}\|\|\Pi_{-1} b_{1k} - \Pi_{J_1} b_{1k}\|\label{EqUseB4'}\\
&\leq C_1\sqrt{\frac{d}{n}}\|g_{j}\|\leq (1+\nu)^{1/2}C_1\sqrt{\frac{d}{n}}\|g_{j}\|_n.\nonumber
\end{align}
Combining \eqref{EqPropJMConstr} and \eqref{EqUseB4'}, we get on the event $\mathcal{E}_{\nu,1}\cap\bigcap_{l=2}^q\mathcal{A}_l$,
\[
 \sup_{g_j\in V_j:\|g_j\|_n\leq 1}\|(\hat \Pi_1 - \Pi_{-1} \hat \Pi_1)^T \mathbf{g}_{j}\|_{n,\infty}\leq C\sqrt{d_1}\eta,\quad j=2,\dots,q,
\]
and the claim follows.
\end{proof}

\subsubsection{End of proof of Theorem \ref{altchoicesthm}}
We now verify properties (A1)--(A4) under the choices and assumptions made in Theorem \ref{altchoicesthm}.  Let the function $\xi: V_{-1} \to \mathbb{R}^+$ be given by $\xi(g_{-1}) = \sum_{j=2}^q \|g_j\|_n$.

(A1): By Proposition \ref{PropJMConstr}, (A1) holds with $\rho_1=(1+\rho_0)/2$.

(A2): By the triangular inequality and \eqref{EqConstrJMInfNorm}, we have
\begin{align*}
\|(\hat \Pi_1 - \hat\Pi_{-1}^M \hat \Pi_1)^T \mathbf{g}_{-1}\|_{n,\infty}&\leq \sum_{j=2}^q\|(\hat \Pi_1 - \hat\Pi_{-1}^M\hat \Pi_1)^T \mathbf{g}_{j}\|_{n,\infty}\leq \Delta_1'\xi(\mathbf{g}_{-1}).
\end{align*}
Thus (A2) is satisfied with $\Delta_1 = \Delta_1'=C\sqrt{d_1}\eta$.

(A3): Plugging in the value of $\Delta_2'$, we have
\[
\|\hat{\Pi}_{-1}^Mg_1\|_n\leq \frac{C}{\psi}\sqrt{\frac{s_1}{d_1}}\|g_1\|_n
\]
for all $g_1\in P_1$ satisfying $\operatorname{supp}(g_{1})\subseteq I_{1k'}$ for some $k'\in\{1,\dots,m_1\}$. Combining this with $n\|\hat\Pi_1 \mathbf{e}_i\|_{n,2}\leq C\sqrt{d_1}$ (see, e.g. the proof of Theorem \ref{ourchoicesthm}) we get that (A3) holds in $\EventNu$ with $\Delta_2 = (C/\psi) \sqrt{s_1}$.

(A4): In the proof of Theorem \ref{ourchoicesthm} it is shown that (A4) is satisfied with $\Delta_3$ depending only on $c_1$ and $t_1$.

The claim now follows by proceeding as in the proof of Theorem  \ref{ourchoicesthm}, by plugging these values of $\Delta_1,\Delta_2,\Delta_3$ and $\rho_1$ into Theorem \ref{mainthm}, setting $\delta = n^{-a}$ in \eqref{T1bound}, and using also Propositions \ref{PropThm3Conc1} and \ref{PropThm3Conc2}.\qed

\section{Proofs and additional material for  Section~\ref{sec:resmooth}} \label{AppendixD}

\subsection{Proof of Theorem \ref{theo:splinepoly}}

For the proof we will make use of the arguments given after the statement of \eqref{eqn:smoothcontin}. For least-squares piecewise polynomial fitting and least-squares splines we will show that \eqref{eqn:smoothundersmooth} hold with $\DeltaOracle =0$ and  that \eqref{eqn:smoothcontin} holds with a universal constant $C=C_*$. This will give the statement of the theorem. The claim $\DeltaOracle =0$ follows directly from ${\mathcal S} \hat {\mathbf{f}}_1^{\text{(oracle,pre)}}=\hat {{f}}_1^{\text{(oracle)}}$, because for two nested linear subspaces iterative application of orthogonal projections onto the two subspaces is equal to a direct orthogonal projection onto the smaller subspace. Note that in the second step of our estimation procedure we project onto the function subspaces of splines or of piecewise polynomials on a grid that is coarser than the grid used in the preliminary estimation step. Thus both spaces are subspaces of the space of piecewise polynomials on the finer grid.

In order to show the small error robustness \eqref{eqn:smoothcontin} of  least-squares piecewise polynomial fitting and spline smoothing it suffices to show that on $\EventNu$, we have $\|{\mathcal S}y\|_{\infty}\leq C\|y\|_{n,\infty}$ for every $y\in\mathbb{R}^n$.

In the case of least-squares piecewise polynomials fitting this property has been derived in the proof of Theorem~\ref{ourchoicesthm}. In fact, it follows from \eqref{EqSER2} in combination with \eqref{eq:asd} and \eqref{EqSER1} by replacing $P_1$ with $P_1^*$, where $P_1^*\subseteq P_1$ is the space of piecewise polynomials with parameters $t_1$ and $m^*$, where $m^*$ is a multiple of $m_1$.

For spline smoothing the arguments are similar. An essential argument now is that spline smoothing preserves the sup norm. More precisely, the least-squares spline smoother of a function that is absolutely bounded by $1$ results in a function that is absolutely bounded by a constant $C$; see e.g.~\cite{dB2}.
For least-squares spline fitting of discrete data one can proceed as in the proofs of Lemmas 6.1--6.3 in \cite{ZSD}.
For this purpose define an  $(m^* + t_1) \times (m^* + t_1)$ matrix with elements $G_{jk}= n^{-1} \sum_{i=1}^n N_j(X_1^i)N_k(X_1^i)$, where $N_j$ are the normed B-spline basis functions for $j,k= 1,...,m^* + t_1$. The central argument in the proof of Lemma 6.3 in \cite{ZSD} is that the elements $\alpha_{jk}$ of the inverse of $G$ can be bounded by $c^* \lambda_{\text{max}} \gamma^{|i-j|}$ with $c^*> 0$ and $0 < \gamma < 1$ depending only on $t_1$ and on a bound of  $\lambda_{\text{max}}/\lambda_{\text{min}}$,
where $\lambda_{\text{min}}$ and $\lambda_{\text{max}}$ are the smallest or largest eigenvalues of $G$, respectively. Because we have on the event $\EventNu$ that $C^{-1} {m^*}^{-1} \leq \lambda_{\text{min}} \leq \lambda_{\text{max}}\leq C {m^*}^{-1}$, one gets that on $\EventNu$, $\| G^{-1} \|_{\infty} =  \max_{1\leq j\leq m^* + t_1} \sum _{1\leq k\leq m^* + t_1} |\alpha_{jk}| \leq C m^*$. Using that $\mathcal{S}y=\sum_{1\leq k\leq m^* + t_1}(\sum_{1\leq j\leq m^* + t_1}\alpha_{jk}\langle N_j,y\rangle_n)N_k$ on $\EventNu$, the claim now follows from using that $n^{-1} \sum_{i=1}^n |N_j(X_1^i) y_i| \leq C  {m^*}^{-1} \|y\|_{n,\infty} $ on $\EventNu$. \qed

\subsection{Proof of \eqref{ratesplineC}} \label{subsec:resmoothsplines} In this subsection we will show that if $\gamma_0$ and $\gamma_1$ are small enough one can always choose $d_1, d_2$ such that our two step procedure can achieve the optimal  rate $\beta = r_1/(2r_1+1)$ as long as $r_2 > 2r_1 /(2r_1 +1)$ holds. We have to choose $d_1, d_2$ such that $\DeltaThm = o(n^{-\beta})$ with $\DeltaThm$ given in Theorem \ref{ourchoicesthm}.

If we choose $d_1 = n^{\delta_1}$ and $d_2 = n ^{\delta_2}$ for some $\delta_1, \delta_2>0$.
 we need with $\beta = r_1 /(2r_1 +1),$ that 
\begin{eqnarray} \label{eqn:resmoothbound1} && \gamma_1 + (\delta_1 \vee \delta_2)/2 + (\delta_1 \vee \gamma_1)/2 - 1/2 < 0\\
 \label{eqn:resmoothbound2} && (\gamma_0\vee\gamma_1) + (\delta_1 \vee \delta_2) - 1/2 < 0\\
 \label{eqn:resmoothbound3} && (\gamma_1 -1)/2 < -\beta, \\
 \label{eqn:resmoothbound4} && \gamma_1 -\delta_1 r_1 < -\beta, \\
 \label{eqn:resmoothbound5}&& \gamma_0 +  \gamma_1- r_2 \delta_2 < -\beta, \\
 \label{eqn:resmoothbound6}&& \gamma_1/2 + (1/2 -2r_1) \delta_1
 < -\beta, \\
  \label{eqn:resmoothbound7}&& \gamma_1/2 +  \delta_1/2  + 2 \gamma_0 - 2 r_2 \delta_2
 < -\beta, \\
 \label{eqn:resmoothbound8}&& \gamma_0 + \gamma_1/2 + (\delta_1 \vee \delta_2) + \delta_1/2-1
 < -\beta.
\end{eqnarray} With the choice $\delta_1 = 1 /(2r_1 +1)+ \delta$ all inequalities can be achieved with $\delta$ and $\gamma_1>0$ small enough as long as
\begin{eqnarray*} &&  \delta_2 > \frac {r_1 + (2r_1 +1) \gamma_0}{r_2(2r_1 +1)}, \\
&& r_1 > \frac {1 + 2 \gamma_0} { 2 (1-2 \gamma_0)}, \\
&& \delta_2 < \frac 1 2 - \gamma_0.
\end{eqnarray*}
Thus a choice of $\delta_2$ is possible if $$ \frac {r_1 + (2r_1 +1) \gamma_0}{r_2(2r_1 +1)} <  \frac 1 2 - \gamma_0$$ which is equivalent to $$r_2 > 2 \frac {r_1 + \gamma_0(2 r_1 +1)} {(2r_1 +1)(1-2 \gamma_0)}.$$
Thus our theory can be applied up to $s_0= O(n^{\gamma_0})$ active additive nuisance components with $\gamma_0 <1/2$ as long as enough smoothness $r_2$ can be assumed for the nuisance components. In particular, for $\gamma_0$ small enough we get that only $r_2 > 2 r_1 /(2 r_1 +1)$ is required. Thus in this case  less smoothness  is required for the nuisance components than for the first component of interest.

\subsection{Proof of Theorem \ref{theo:locpoly}} 
As in the proof of Theorem \ref{theo:splinepoly} we will again make use of the arguments given after the statement of equation \eqref{eqn:smoothcontin}. Using similar arguments as in that proof one can easily check 
that small error robustness \eqref{eqn:smoothcontin} holds for local polynomials on the event  $\EventNu$. Here we apply again that up to a factor orthogonal projections onto polynomials preserve the sup norm under weak conditions on the design. Compare also the arguments given in the first part of the proof of Theorem 3 in \cite{HKM}. It remains 
to prove  \eqref{eqn:smoothundersmooth} for some sequence $\DeltaOracle $.

With arguments similar to the ones used in the first part of the proof of Theorem 3 in \cite{HKM} one can show that for the proof of \eqref{eqn:smoothundersmooth} with $\DeltaOracle =C (d_1^{-\rho} + (d_1 h)^{-1} (nh)^{-1/2} \sqrt{ \log(n)})$ it suffices to prove that
\begin{eqnarray*}&&  \|\tilde r ^{j,\text{(oracle,lpol)}}-\tilde{ \hat r} ^{j,\text{(oracle,lpol)}}\|_\infty\\ && \qquad \leq C (d_1^{-\rho} + (d_1 h)^{-1} (nh)^{-1/2}( \sqrt{ \log(n)}+\sqrt{z}))\end{eqnarray*} uniformly  for all $h$ with $c_1 n^{-\eta_1} \leq h \leq c_2 n^{-\eta_2}$
with probability greater than or equal to $1 - 2n ^{-a} - \exp(-z)$. Here 
\begin{eqnarray*}
\tilde r ^{j,\text{(oracle,lpol)}} (x)&=& n^{-1} \sum_{i=1}^n w_{i,j}(x) Z^i, \\
\tilde{ \hat r} ^{j,\text{(oracle,lpol)}} (x)&=& n^{-1} \sum_{i=1}^n w_{i,j}(x) \hat Z^i, \\
w_{i,j}(x) &=& h^{-j} (X_1^i-x)^j K_h (X_1^i-x)
\end{eqnarray*}
with $\hat Z_i = \hat f_1 ^{\text{(oracle,pre)}}(X_1^i)$.
Fix $j\in \{0,...,k\}$ and define for $i=1,...,n$, $x \in [0,1]$, $l=1,...,m_1$
 $$\bar w_{i,j}(x)= \sum_{t=0}^{k} a_t^l (X_1^i-x)^t$$ if $X_1^i \in \left (\frac {l-1} {m_1}, \frac {l} {m_1}\right ]$ with
\begin{eqnarray*}(a_0^l,...,a_{k}^l)&= &\arg \min \sum_{v=1} ^nI \left [X_1^v \in \left (\frac {l-1} {m_1}, \frac {l} {m_1}\right ]\right ] 
\\ && \times \left [ w_{v,j}(x)  - a_0 - ...- a_{k} (X_1^v-X_1^i)^{k}\right ]^2 .
\end{eqnarray*}
With this notation we get that 
\begin{eqnarray*} 
&&\tilde r ^{j,\text{(oracle,lpol)}} (x)- \tilde{ \hat r} ^{j,\text{oracle,lpol}} (x)= n^{-1} \sum_{i=1}^n w_{i,j}(x) (Z^i-  \hat Z^i)\\
&& \qquad = n^{-1} \sum_{i=1}^n (w_{i,j}(x)- \bar  w_{i,j}(x)) (Z^i-  \hat Z^i)\\
&& \qquad = n^{-1} \sum_{i=1}^n (w_{i,j}(x)- \bar  w_{i,j}(x)) Z^i\\
&& \qquad = n^{-1} \sum_{i=1}^n (w_{i,j}(x)- \bar  w_{i,j}(x)) \epsilon^i\\
&& \qquad \qquad  + n^{-1} \sum_{i=1}^n (w_{i,j}(x)- \bar  w_{i,j}(x)) f_1(X_1^i)
\\
&& \qquad = n^{-1} \sum_{i=1}^n (w_{i,j}(x)- \bar  w_{i,j}(x)) \epsilon^i\\
&& \qquad \qquad  + n^{-1} \sum_{i=1}^n w_{i,j}(x)(M^i-\hat M^i),
\end{eqnarray*}
where $M^i= f_1(X_1^i)$ and $$\hat M^i=  \sum_{t=0}^{k} b_t^l (X_1^i-x)^t$$ if $X_1^i \in \left (\frac {l-1} {m_1}, \frac {l} {m_1}\right ]$ with
\begin{eqnarray*}(b_0^l,...,b_{k}^l)&= &\arg \min \sum_{v=1} ^nI \left [X_1^v\in \left (\frac {l-1} {m_1}, \frac {l} {m_1}\right ]\right ] \\ && \times \left [ M^i  - b_0 - ...- b_{k} (X_1^v-X_1^i)^{k}\right ]^2 .\end{eqnarray*}
The theorem now follows from 
\begin{eqnarray*}&&\sup_{1 \leq i \leq n} | \hat M^i - M^i| \leq C d_1 ^{- \rho*},\\
&& n^{-2} \sum_{i=1}^n (w_{i,j}(x)- \bar w_{i,j}(x))^2\\&& \qquad \leq n^{-2} \sum_{i=1}^n (w_{i,j}(x)- \tilde w_{i,j}(x))^2 \leq C \frac 1 {d^2_1 h^2} \frac 1 {n h} ,
\end{eqnarray*}
and standard bounds for the supremum of Gaussian random variables.
Here $$\tilde w_{i,j}(x) = \frac {\frac 1 n  \sum_{v=1} ^nI \left [X_1^v \in \left (\frac {l-1} {m_1}, \frac {l} {m_1}\right ]\right ] 
 w_{v,j}(x)}{\frac 1 n  \sum_{v=1} ^nI \left [X_1^v \in \left (\frac {l-1} {m_1}, \frac {l} {m_1}\right ]\right ] } $$
if $X_1^i \in \left (\frac {l-1} {m_1}, \frac {l} {m_1}\right ]$. \qed

\subsection{Implementation of Nadaraya-Watson smoothing in the additive model} \label{subsec:NaWa}

For our two-step procedure we now discuss implementation of Nadaraya-Watson smoothing in the second step. We do this in an asymptotic setting where the number of observations converges to infinity. First we consider theory for a fixed deterministic  bandwidth (sequence) $h$. We assume that $f_1$ has $\rho^*=2$ continuous derivatives and that the density $p_1$ of $X_1^i$ has one continuous derivative.  In the oracle model the Nadaraya-Watson estimator is given by
$$\hat f_1 ^{\text{(oracle,NaWa)}}(x)= \sum_{i=1} ^n \omega_{h,i}(x) Z^i$$
with $$\omega_{h,i}(x) = K_h(X_1^i-x) /  \sum_{j=1} ^n K_h(X_1^j-x).$$
As above, in the oracle model we observe  $Z^i = f_1(X_1^i) + \varepsilon^i$ with two independent  i.i.d. samples, $X_1^i:i=1,...,n$ and   $\varepsilon^i:i=1,...,n$. Here $\varepsilon^i$ has a zero mean normal distribution with variance $\sigma^2$. Conditionally given $X_1^i:i=1,...,n$, the estimator $\hat f_1 ^{\text{(oracle,NaWa)}}(x)$ is a Gaussian process with expectation $f_{1,h}(x) =\sum_{i=1} ^n \omega_{h,i}(x)  f_1(X_1^i) $ and covariance 
$\sigma^2\sum_{i=1} ^n \omega_{h,i}(x)\omega_{h,i}(x')$. Thus with a consistent estimator $\hat \sigma$ of $\sigma$ and 
$\overline{ \omega^2_h}(x)=\sum_{i=1} ^n \omega^2_{h,i}(x)$ we get that 
\begin{eqnarray}\label{eqn:coverNaWa} &&\left [\hat f_1 ^{\text{(oracle,NaWa)}}(x)-k_\alpha \hat \sigma \sqrt {\overline{ \omega^2_h}(x)},\hat f_1 ^{\text{(oracle,NaWa)}}(x)  + k_\alpha \hat \sigma \sqrt {\overline{ \omega^2_h}(x)}\right ]\end{eqnarray}
is a pointwise confidence interval for $f_{1,h}(x)$ with asymptotic coverage $1- \alpha$ if $k_\alpha$ is the $(1 - \alpha/2)$-quantile of a standard normal distribution. Furthermore, it is a uniform confidence band for  $f_{1,h}(x): x \in [0,1] $ with asymptotic coverage $1- \alpha$ if $k_\alpha$ is the $(1 - \alpha)$-quantile of the maximal absolute value of a mean zero Gaussian process on $[0,1]$ with covariance  $\sum_{i=1} ^n \omega_{h,i}(x)\omega_{h,i}(x')$ and if it holds that $\hat \sigma - \sigma = o_P( (\log n)^{-1})$. The latter condition is needed because the absolute maximum of the Gaussian process has a mean of order $(\log n)^{1/2}(nh)^{-1/2}$ but standard deviation of smaller order $(\log n)^{-1/2}(nh)^{-1/2}$. If $h^{5} n$ or $h^{5} n \log n$ , respectively, converges to zero we get that the difference $f_{1,h}(x)-f_{1}(x)$ becomes asymptotically negligible and we have that \eqref{eqn:coverNaWa} is an asymptotically valid confidence interval for $f_1(x) $ or confidence band for $f_1(x): h \leq x \leq 1-h$, respectively.

We now consider our two-step estimator $\hat f_1 ^{\text{(NaWa)}}$ with Nadaraya-Watson smoothing used in the second step: 
$$\hat f_1 ^{\text{(NaWa)}}(x)= \sum_{i=1} ^n \omega_{h,i}(x) \hat Y^i,$$
where $\hat Y^i$ is the $i$th element of $\hat \f_1^{\operatorname{(pre)}}$. 

We suppose now that the assumptions of 
Theorems \ref{ourchoicesthm} or \ref{altchoicesthm} and \ref{theo:locpoly}  hold with $\DeltaThm (nh)^{1/2} \to 0$ and  $(d_1 h)^{-1} \sqrt{ \log(n)}\to 0$. Then for all $x \in (0,1)$ it can be easily checked that \eqref{eqn:coverNaWa} with $\hat f_1 ^{\text{(oracle,NaWa)}}$ replaced by $\hat f_1 ^{\text{(NaWa)}}$ is a confidence interval with asymptotic coverage $1 - \alpha$ for the estimation of $f_{1,h}(x)$. And again if $h^{5} n \to 0$, it also covers $f_{1}(x)$ with asymptotic coverage $1 - \alpha$. For the construction of the confidence interval we need a consistent estimator $\hat \sigma$ of $ \sigma$. Now, this estimator should only depend on $(X_1^i, \hat Y^i): i=1,...,n$. We will propose such an estimator below.

If we make the assumptions that $\DeltaThm \sqrt{ \log(n)} (nh)^{1/2} \to 0$,   $(d_1 h)^{-1} { \log(n)}\to 0$ and $\hat \sigma - \sigma = o_P( (\log n)^{-1})$ we get that \eqref{eqn:coverNaWa} with $\hat f_1 ^{\text{(oracle,NaWa)}}$ replaced by $\hat f_1 ^{\text{(NaWa)}}$ is a uniform confidence band for  $f_{1,h}(x): x \in [0,1] $ with asymptotic coverage $1- \alpha$. And under the additional assumption that $h^{5} n \log n\to 0$ we have that it is a valid confidence band for $f_1(x): h \leq x \leq 1-h$. 

The uniform confidence band can be used to test for the validity of a  hypothesis $f_1 \in {\mathcal{F}}$ on the function $f_1$, where ${\mathcal{F}} $ is a parametric or nonparametric class of functions. Then we can use the test statistic
$$\inf _ {f \in {\mathcal{F}}} \sup_{h \leq x \leq 1-h} |\hat f_1 ^{\text{(NaWa)}}(x) - f(x)|$$ or $$\inf _ {f \in {\mathcal{F}}_h} \sup_{0 \leq x \leq 1} |\hat f_1 ^{\text{(NaWa)}}(x) - f(x)|$$ with the critical value $k_\alpha \hat \sigma \sqrt {\overline{ \omega^2_h} }$. Here we put $ {\mathcal{F}}_h = \{ \sum_{i=1} ^n \omega_{h,i}(x)  f(X_1^i): f \in  {\mathcal{F}}\}$. This is a test with asymptotic level $\alpha$.

We now propose an estimator $\hat \sigma^2$ of the error variance $\sigma^2$ based on $(X_1^i, \hat Y^i): i=1,...,n$. For some constant $0 < a < 1$, a bandwidth $g$ and a weight function $w$ that vanishes in a neighborhood of  the boundaries of the support of $X_1^i$  we define
$$ \hat \sigma ^2 = \frac {\int { \left ( \sum_{i=1} ^n (\omega_{g,i}(x) - \omega_{ag,i}(x)) \hat Y ^i\right ) ^2w(x) \mathrm d x}} {\overline{ \omega^2_{g,a}}}$$
with $$ \overline{ \omega^2_{g,a}}= \sum_{i=1} ^n \int {  (\omega_{g,i}(x) - \omega_{ag,i}(x)) ^2w(x)\mathrm d x}.$$
We now argue that this estimator is consistent for suitable choices of $g$. Note that 
\begin{eqnarray*} &&\int { \left ( \sum_{i=1} ^n (\omega_{g,i}(x) - \omega_{ag,i}(x)) \hat Y ^i\right ) ^2w(x) \mathrm d x} \\
&& \qquad  = \int { \left ( \sum_{i=1} ^n (\omega_{g,i}(x) - \omega_{ag,i}(x)) \varepsilon ^i + R_i\right ) ^2w(x) \mathrm d x} ,
\end{eqnarray*}
where $R_i $ are random variables that fulfill
$$\sup_{1 \leq i \leq n} | R_i| = O_P( \DeltaThm + g^2 + d_1^{-2} + (d_1 g) ^{-1} (ng) ^{-1/2} \sqrt { \log n} ).$$
For $g $ of the order $n^{-1/5-\delta}$ for some $\delta > 0$ small enough, one gets under appropriate conditions on $d_1$ and $\DeltaThm$ that the right hand side of this equation is of order $o_P( (ng) ^{-1/2}(\log n) ^{-1/2})$. This can be used to show that 
\begin{eqnarray*} &&\int { \left ( \sum_{i=1} ^n (\omega_{g,i}(x) - \omega_{ag,i}(x)) \hat Y ^i\right ) ^2w(x) \mathrm d x} \\
&& \qquad  = \sum_{i=1} ^n \int {  (\omega_{g,i}(x) - \omega_{ag,i}(x)) ^2w(x)\mathrm d x} (\varepsilon^i ) ^2 + o_P( (ng) ^{-1}).
\end{eqnarray*} 
Consistency of $\hat \sigma^2 $ follows now because $ \overline{ \omega^2_{g,a}}$ is of order $(ng) ^{-1}$. Above, we needed for the validity of  the uniform confidence bands rates of convergence for the variance estimator. Such rates can be achieved again by choices of $g$ of the order $n^{-1/5-\delta}$ for some $\delta > 0$ small enough, now under slightly stronger conditions on $d_1$ and $\DeltaThm$. An alternative estimator of $\hat \sigma^2$ is given by 
$$ \hat \sigma ^2 = \frac {\sum_{j=1} ^n \left (\hat Y ^j -  \sum_{i=1} ^n \omega_{g,i}(X_1^j )  \hat Y ^i\right ) ^2w(X_1^j)} {\widehat{ \omega^2_{g,*}}},$$
where $\widehat{ \omega^2_{g,*}}$ is equal to the conditional expectation, given $X_1 ^i: 1 \leq i \leq n$, of $\sum_{j=1} ^n \left (\hat Z ^j -  \sum_{i=1} ^n \omega_{g,i}(X_1^j )  \hat Z ^i\right ) ^2w(X_1^j)$
in an oracle model where $f_1$ is equal to zero and $\sigma^2=1$. Note that this conditional expectation is explicitly given and it only depends on $X_1 ^i: 1 \leq i \leq n$. Consistency of this estimator for appropriate choices of $g$ can be shown by similar arguments as for the last variance estimator.

We now briefly discuss data-adaptive bandwidth selection. In the next subsection we will discuss local bandwidth choice. For global bandwidth choice, one can make use of plug-in estimates of the asymptotic MISE-optimal bandwidth that is given by  $h^* =$  $ n^{-1/5}$  $ (\int (2p_1'f_1' + f_1 '')^2(x) w(x) \mathrm d x) ^{-1/5} $  $(\int u^2 K(u) \mathrm d u) ^{-2/5}$ $  (\int  K^2(u) \mathrm d u) ^{1/5} (\int p_1(x) w(x) \mathrm d x ) ^{1/5} \sigma^{2/5}$.
Estimates of $f_1 '$ and $f_1''$ were studied in Theorem \ref{theo:locpoly}. These estimators can be plugged in into the formula of the asymptotic optimal bandwitdth. Consistency of the resulting bandwidth selector immediately follows from Theorem \ref{theo:locpoly} under suitable conditions.   Alternative procedures are given by penalized least-squares criterions, as e.g. generalized cross validation, Akaike's information criterion, and Shibata's model selector. A discussion of these methods leads to arguments similar to our discussion of variance estimation. We omit a discussion of the performance of these procedures in the additive model. 

Up to now, the whole discussion of this subsection assumed that the error variables have a Gaussian distribution. {All the arguments in this subsection would also apply, after small modifications, to the case of non-Gaussian errors if it is the case that the theory of this paper could be extended to other error distributions. Clearly, for non-Gaussian errors the Nadaraya-Watson estimator no longer has a conditional Gaussian distribution, so we would need central limit theorems to fill this gap in the discussions above. 

\subsection{Small error robustness of Lepski method for pointwise kernel estimation} \label{subsec:Lepski}

In this subsection we will discuss small error robustness \eqref{eqn:smoothcontin} for kernel smoothing with data-adaptive local bandwidth selection. We will compare a kernel estimator at a fixed point with local bandwidth for responses in the oracle model and in the additive model. In the oracle model \eqref{eqn:ormodel} the bandwidth is based on the oracle responses and these responses are also used for the calculation of the kernel estimator. In the distorted model the responses are replaced by new values that may differ from the responses in the oracle model by a difference absolutely bounded by some constants $\Delta'_n$. Then in the distorted model the new responses are also used in both steps: in the calculation of the data-adaptive bandwidth and in the calculation of the kernel estimator that makes use of this bandwidth.  We will carry out this comparison for a large class of adaptation rules. In particular this class contains Lepski's bandwidth selector. We will study this in an asymptotic setting in which the number $n$ of observations converges to infinity. Our main result is that with probability tending to one the estimators do not differ by a value that is larger than the maximal distortion of the responses. 

For a kernel estimator $\hat f_{h,1} ^{\operatorname{(oracle,kernel)}}(x)$ of the function value $f_1(x)$ in the oracle model we consider a Lepski-type rule for data-adaptive choice of the bandwidth $h$:
\begin{eqnarray}&&\label{eq:hhat} \hat h = \max\{ h \in {\mathcal H}_n: | \hat f_{h,1} ^{\operatorname{(oracle,kernel)}}(x) - \hat f_{\eta,1} ^{\operatorname{(oracle,kernel)}}(x)| \leq \psi_n (h,\eta)\\ \nonumber && \qquad \text{ for all } \eta \leq h \text{ with } \eta \in  {\mathcal H}_n\}.\end{eqnarray}
Here $ {\mathcal H}_n$ is a finite set of bandwidths with $\#  {\mathcal H}_n \leq C (\log n) ^{\kappa_1}$ for some constants $C, \kappa_1>0$ and $\psi_n$ is a function defined on $\{(h,\eta): h,\eta \in {\mathcal H}_n, \eta \leq h\}$. We allow $\psi_n$ to be random. In particular,  it may depend on the values of $X_1^i:i=1,...,n$, but we assume that $\psi_n$ is independent of $\varepsilon^i: i=1,...,n$. We fix $x \in (0,1)$ and write
$$\hat \mu (h) = \hat f_{h,1} ^{\operatorname{(oracle,kernel)}}(x) = \frac {\sum_{i=1}^n K\left (h^{-1} (X_1^i-x)\right)Z^i} {\sum_{i=1}^n K\left (h^{-1} (X_1^i-x)\right)},$$
where $Z^i = f_1(X_1^i) + \varepsilon^i$ with i.i.d. standard normal $\varepsilon^i$ that are independent of $X_1^i: i=1,...,n$. We now discuss how much the value of $\hat \mu (h) $ can change if the responses $Z^i$ are replaced by distorted values. For simplicity of notation we only consider the case in which each $Z^i$ is replaced by $\hat Y^i$, where  $\hat Y^i$ is the $i$th element of $\hat \f_1^{\operatorname{(pre)}}$. 
Define 
$$\tilde \mu (h) =  \hat f_{h,1} ^{\operatorname{(kernel)}}(x) =\frac {\sum_{i=1}^n K\left (h^{-1} (X_1^i-x)\right) \hat Y^i} {\sum_{i=1}^n K\left (h^{-1} (X_1^i-x)\right)}.$$
This defines the Lepski bandwidth choice based on the distorted responses:
\begin{eqnarray} \label{eq:htilde}&& \tilde h = \max\{ h \in {\mathcal H}_n: | \tilde \mu (h) -\tilde \mu (\eta)| \leq \psi_n (h,\eta)\\ \nonumber && \qquad \text{ for all } \eta \leq h \text{ with } \eta \in  {\mathcal H}_n\}.\end{eqnarray}

In Theorem \ref{theo:locpoly} we have shown that \begin{eqnarray*} \|\tilde \mu (h) -\hat \mu (h)\|_\infty &\leq& \Delta_n'\end{eqnarray*} 
with \begin{eqnarray} \label{eq:defdelta}
\Delta_n' = C \big [\DeltaThm + d_1^{-\rho^*} + (d_1 h)^{-1} (nh)^{-1/2}\sqrt{ \log(n)}\big ]\end {eqnarray} uniformly  for all $h$ with $c_- n^{-\eta_1} \leq h \leq c_+ n^{-\eta_2}$
with probability greater than or equal to $ 1 - 6 n^{-a} $. 

In the following proposition we will prove that under suitable assumptions  \begin{eqnarray*} 
|\tilde \mu (\tilde h) - \hat \mu (\hat h)| \leq \Delta_n'\end {eqnarray*}
with probability tending to 1. 

This shows that kernel estimators under Lepski bandwidth selection are small error robust. We note that this holds for a general class of adaptive  bandwidth selectors because we only assumed on the function $\psi_n$ that it is independent of $\varepsilon^i:i=1,...,n$.
For choices of the function $\psi_n$ for pointwise kernel estimation in a white noise model we refer to \cite{LepSpok} and to \cite{LepMamSpok},
where pointwise bandwidth selection is used to achieve optimal global rates in function spaces with inhomogeneous smoothness. Applications of the Lepski method in nonparametric regression are discussed in \cite{Gaif}.

\begin{proposition} For a finite set $ {\mathcal H}_n$ of bandwidths with $\#  {\mathcal H}_n \leq C (\log n) ^{\kappa_1}$ for some constants $C, \kappa_1>0$ and for a random function $\psi_n$ defined on $\{(h,\eta): h,\eta \in {\mathcal H}_n, \eta \leq h\}$ 
that
is independent of $\varepsilon^i: i=1,...,n$ define the bandwidth selector $\hat h$ in the oracle model, see
\eqref{eq:hhat} and define the  bandwidth selector $ \tilde h$ in the additive model, see \eqref{eq:htilde}. Assume that the assumptions of Theorem \ref{theo:locpoly} apply for all $h \in {\mathcal H}_n$
and that the following inequality holds
$$(\log n) ^{2 \kappa_1 + \kappa_2} \Delta_n' \sqrt { n h_{n,max}} \to 0,$$
where  $h_{n,max}$ is the maximal value in $ {\mathcal H}_n$ and where  $\kappa_2>0$ is chosen 
such that
$$\min \{ \eta^{-1} h : \eta < h  \in  {\mathcal H}_n\} \geq 1 + C^* (\log n) ^{- \kappa_2}$$
for some constant $C^*>0$. Then it holds that
\begin{eqnarray} \label{claim:Lepski}
|\tilde \mu (\tilde h) - \hat \mu (\hat h)| =  |\hat f_{\tilde h,1} ^{\operatorname{(kernel)}}(x) -  \hat f_{\hat h,1} ^{\operatorname{(oracle,kernel)}}(x)| \leq \Delta_n',\end {eqnarray} where $\Delta_n'$ was defined in \eqref{eq:defdelta}.

\end{proposition}

\begin{proof}
For constants $C_1, C_2,C_3>0$ we define the event ${\mathcal E}_n$
\begin{eqnarray*}&&{\mathcal E}_n =\left. \Big \{ (ng)^{-1}{\sum_{i=1}^n K\left (g^{-1} (X_1^i-x)\right) ^2} \geq C_1,\right .\\ && \qquad
 (ng)^{-1}\sum_{i=1}^n K\left (g^{-1} (X_1^i-x)\right)\leq C_2  \text { for all } g \in {\mathcal H}_n \text{ and } \\
&& \qquad  n^{-1} (\log n) ^{2\kappa_2} h \sum_{i=1}^n \left (h^{-1} K\left (h^{-1} (X_1^i-x)\right) -\eta^{-1}K\left (\eta^{-1} (X_1^i-x)\right) \right )^2 \\ && \qquad \left .
\geq C_3\text { for all } \eta < h \in {\mathcal H}_n\right. \Big  \}. \end{eqnarray*}
One can show under the assumptions of the proposition that for suitable choices of $C_1, C_2,C_3$
 the probability of the event ${\mathcal E}_n$ converges to 1.
 
For the proof of \eqref{claim:Lepski} we show that
$\hat h_- = \hat h_+$ with probability tending to one, where 
\begin{eqnarray*}&& \hat h_- = \max\{ h \in {\mathcal H}_n: | \hat \mu (h) -\hat \mu (\eta)| \leq (\psi_n (h,\eta) - 2 \Delta_n')^+\\ && \qquad \text{ for all } \eta \leq h \text{ with } \eta \in  {\mathcal H}_n\}, \\
&& \hat h_+ = \max\{ h \in {\mathcal H}_n: | \hat \mu (h) -\hat \mu (\eta)| \leq \psi_n (h,\eta) +2 \Delta_n'\\ && \qquad \text{ for all } \eta \leq h \text{ with } \eta \in  {\mathcal H}_n\}.\end{eqnarray*}
This implies $\hat h = \tilde h $ with probability tending to one because of $\hat h_-\leq \hat h \leq \hat h_+ $ and $\hat h_-\leq \tilde h \leq \hat h_+ $ with probability tending to one. Note that 
\begin{eqnarray} \label{eq:mubound}&&\max_{ h \in {\mathcal H}_n} | \hat \mu (h) -\tilde \mu (h)| \leq \Delta_n'\end{eqnarray}
with probability tending to one.
Furthermore, from the fact that  $\hat h = \tilde h $ with probability tending to one, we conclude, again by using \eqref{eq:mubound}, that  \eqref{claim:Lepski} holds with probability tending to one. 

It remains to show that $\hat h_- = \hat h_+$ with probability tending to one. For a proof of this claim note that for $h  \in  {\mathcal H}_n$
\begin{eqnarray*}&& \mathbb{P}\left(  \hat h_+ = h,  \hat h_- < h\right )\\&& \leq \mathbb{P}\left( | \hat \mu (h) - \hat \mu (\eta)| \leq \psi_n(h,\eta) + 2 \Delta_n', | \hat \mu (h) - \hat \mu (\eta)| > \psi_n(h,\eta) - 2 \Delta_n'\right .\\ && \qquad \left . \text{ for an } \eta \in {\mathcal H}_n \text{ with } \eta < h\right )\\ &&
\leq C' (\log n) ^{\kappa_1} \Delta_n' \sqrt {n h_{n,max}} (\log n) ^{\kappa_2}
\end {eqnarray*}
for some $C' > 0$.
Here, in the last inequality we applied $\#  {\mathcal H}_n \leq C (\log n) ^{\kappa_1}$ and the fact that for $ \eta < h  \in  {\mathcal H}_n$ conditionally given  $X_1^i:i=1,...,n$, $\sqrt {n h }(\log n) ^{\kappa_2} (\hat \mu (h) - \hat \mu (\eta))$ has a normal distribution with the variance bounded from below by a constant on the event ${\mathcal E}_n$. The latter fact implies that  for all intervals $I$ of length $\delta$, we have on ${\mathcal E}_n$ that  $$ \mathbb{P}\left(  \sqrt {nh }(\log n) ^{\kappa_2} (\hat \mu (h) - \hat \mu (\eta))\in I | X_1^i:i=1,...,n \right ) \leq C'' \delta$$
for some $C'' > 0$.

We conclude that \begin{eqnarray*}&& \mathbb{P}\left(  \hat h_+ \not =  \hat h_- \right  )\\&& 
\leq C''' (\log n) ^{2\kappa_1} \Delta_n' \sqrt {n h_{n,max}} (\log n) ^{\kappa_2} + 
\mathbb{P}\left( {\mathcal E}_n^c \right  )\end {eqnarray*}
for some $C''' > 0$. The right hand side converges to zero by assumption. This shows that $\hat h_- = \hat h_+$ with probability tending to one and concludes our proof of \eqref{claim:Lepski}.
	\end{proof}

\subsection{Asymptotic minimax estimation in additive models} \label{subsec:minimax}

We conclude this section by discussing a minimax theorem. To simplify notation we formulate this theorem asymptotically. For the first additive component we assume  that 
\begin{equation} \label{condsmooth}
f_1 \in \mathcal{W}  = \left \{ g: [0,1] \to \mathbb{R}: \int_0^1 g^{(r_1^*)}(x)^2 \ dx \leq C_{\mathcal{W} } \right \},
\end{equation}
where $r_1^* \geq 1$ and $C_{\mathcal{W}} > 0$. By the Sobolev embedding theorem this implies that 
for all $f_1 \in \mathcal{W}$ there is a $g_1^*\in V_1$ satisfying with a constant $C(C_{\mathcal{W}})$ that only depends on $C_{\mathcal{W}}$
\begin{equation*}
\|f_1-g^*_1\|_\infty\leq C(C_{\mathcal{W}})d_1^{-r_1}
\end{equation*} 
with $r_1=r_1^* - 1/2$, so that the first part of Assumption (B3) is satisfied.

We now define a class ${\mathcal F}_n= {\mathcal F}_n(r_1^*,C_{\mathcal{W} }, c_1,C_0,C_1,r_2,\phi,\rho_0,\gamma_0,C^0,\gamma_1,C^1)$ of tuples $(f_1,\dots,f_q,p)$ of additive components $f_1,\dots,f_q$ and  densities $p$ of $(X_1,\dots,X_q)$, where it is assumed that these functions fulfill Assumptions (B1)-(B5) with constants $c_1$, $C_0>C(C_{\mathcal{W}})$, $C_1$, $r_2$, $r_1 =r_1^*-1/2$, $\phi$, $\rho_0$ and $s_0$, $s_1$ with $s_0 \leq C^0 n^{\gamma_0}$ and $s_1 \leq C^1 n^{\gamma_1}$, $q \leq \omega_n$ and where $f_1 \in \mathcal{W}$. Here $\omega_n$ is a fixed sequence with $\log \log \omega_n = o(\log n)$. 
We now state our minimax theorem. 
\begin{thm}\label{theo:minimax} Suppose that for some constants $r_2$, $r_1^*$, $\gamma_0$, and $\gamma_1$, the constraints
\eqref{eqn:resmoothbound1}--\eqref{eqn:resmoothbound8}, stated in the Subsection \ref{subsec:resmoothsplines} of the Supplementarial Material, hold with $\beta = r_1^*/(2r_1^* +1)$ and $r_1 = r_1^* - 1/2$. 
Then there exists an estimator $\hat f_1$ in the additive model with 
\[
n^{2r_1^*/(2 r_1^* + 1)} \kappa(p_1)^{-1} \mathbb{E}_\epsilon \left [\int_0^1 ( \hat f_1 (x) - f_1(x))^2 \ dx\right] = 1 + o_P(1)
\]
uniformly over $(f_1,\dots,f_q,p) \in   {\mathcal F}_n(r_1^*,C_{\mathcal{S} }, c_1,C_0,C_1,r_2,\phi,\rho_0,\gamma_0,C^0,\gamma_1,C^1)$ for positive constants $C_{\mathcal{W} },c_1,C_0,C_1,\phi,C^0,C^1> 0$ and $0 \leq \rho_0 < 1$. Here
\[
\kappa(p_1) = \left \{ (2r_1^* +1 ) C_{\mathcal{W} } \left ( \frac {\sigma^2 r_1^*}{\pi (r_1^* +1)} \int_0^1 p_1^{-1}(x)\ dx\right) ^{2 r_1^*} \right\}^{ 1/(2r_1^* +1)}
\] 
and $\mathbb{E}_\epsilon$ denotes the conditional expectation, given $X_1,\dots,X_q$.
\end{thm}

The theorem states that, under our assumptions, the asymptotic minimax risk for estimators in the oracle model can be achieved in the additive model. 
 We conjecture that the minimax result continues to hold under weaker sparsity conditions because we use in the proof 
  $L_\infty$ bounds between the pilot estimators but for the stated minimax theorem only $L_2$ bounds are needed.
The proof of Theorem \ref{theo:minimax} makes use of similar arguments as used in the proofs of Theorems \ref{theo:splinepoly} and \ref{theo:locpoly} and proceeds then as in 
 the  proof of Theorem 6 in  \cite{HKM}. The minimax estimator can be chosen as two-step estimator according to the construction presented in this paper. The value $\kappa(p_1)$ is the asymptotic minimax risk in the oracle model, which has been established in \cite{Efro}; see also the discussion in \cite{HKM}, where a minimax theorem for additive models was proved for the case in which $q$ is fixed.

\section{Discussion of Assumption (B4). Proof of Lemma \ref{lemB4check}. } \label{sectdiscussB4}

	In this subsection we will prove Lemma \ref{lemB4check}. For simplicity we will do this only for the case $t_1=...=t_q =0$. The following arguments also apply for other choices of  $t_1$, ..., $t_q$ by slightly more complex arguments. Without loss of generalitiy, we assume that 
	 \begin{equation} \label{eq:asslem1}|p_{il}(x_i,x_l) - p_i(x_i) p_l(x_l) | \leq c^* \rho^{|i-l|}\end{equation} for $1 \leq i,l \leq q$, with $0 < \rho < 1$ and a constant $c^*>0$ that is small enough. 
	 
In the case  that $t_1=...=t_q =0$ holds the basis functions $b_{j,k}$ are equal to $m_j^{1/2}$ on an interval of length $m_j^{-1}$ and equal to 0 outside of the interval. For the projections $\Pi_{J_k}b_{1,k}$ and $\Pi_{-1}b_{1,k}$ we have that they are equal to
\begin{eqnarray*}
\Pi_{J_k}b_{1,k} (x) &=& \sum_{j\in J_k} \beta^* _j(x_j), \\
\Pi_{-1}b_{1,k} (x) &=& \sum_{j=2} ^q\beta _j(x_j),
\end{eqnarray*}
where $\beta^* _j$ and $\beta _j$ are functions that are piecewise constant on the intervals $I_{jk}$, fulfill $\int_0^1 \beta^* _j(x_j)  p_j(x_j)  \ \mathrm{d} x_j =0$ and $\int_0^1 \beta _j(x_j)  p_j(x_j)  \ \mathrm{d} x_j =0$, and minimize
\begin{eqnarray*}
\int ( b_{1,k} (x_1) - \sum_{j\in J_k} \beta^* _j(x_j))^2 p^n(x) \ \mathrm{d} x, \\
\int ( b_{1,k} (x_1) - \sum_{j=2} ^q \beta _j(x_j))^2 p^n(x) \ \mathrm{d} x
\end{eqnarray*}
in this class of functions. Here, $$p^n(x) =\prod_{j=1}^q m_j \int _{\prod_{j=1}^q I_{j,k_j}}p (u) \ \mathrm{d} u$$ for $x \in \prod_{j=1}^q I_{j,k_j}$
with $p$ equal to the density of $(X_1,...,X_q)$. This implies
\begin{eqnarray*}
 \beta^* _l(x_l) &=& \alpha_l(x_l) - \sum_{j\in J_k, j\not =l} \int \beta^* _j(x_j)   \frac {p^n_{jl}(x_j, x_l)}{p^n_l(x_l)} \ \mathrm{d} x_j \mbox { for } l \in J_k, \\
 \beta_l(x_l) &=& \alpha_l(x_l) - \sum_{j\in \{1,...,q\}\backslash \{l\}} \int \beta _j(x_j)   \frac {p^n_{jl}(x_j, x_l)}{p^n_l(x_l)} \ \mathrm{d} x_j \mbox { for } l =2,...,q, 
 \end{eqnarray*}
where $$\alpha_l(x_l) =  \int  b_{1,l} (x_1) \frac {p^n_{1l}(x_1, x_l)}{p^n_l(x_l)} \ \mathrm{d} x_1-  \int  b_{1,l} (x_1) p^n_{1}(x_1) \ \mathrm{d} x_1 .$$
Here we have used the notation 
\begin{eqnarray*}p_{j,l}^n(x_j,x_l) &=& m_j  m_l \int _{  I_{j,k_j} \times  I_{l,k_l}}p_{j,l}(u_j,u_l)  \ \mathrm{d} u_j\ \mathrm{d} u_l  \end{eqnarray*} for $(x_j,x_l) \in I_{j,k_j} \times  I_{l,k_l}$ and 
\begin{eqnarray*}p_{j}^n(x_j) &=& m_j  \int _{  I_{j,k_j} }p_{j}(u_j)  \ \mathrm{d} u_j \end{eqnarray*}  for $ x_j\in I_{j,k_j}$. Note that these integral equations remain valid if we replace 
$ \frac {p^n_{jl}(x_j, x_l)}{p^n_l(x_l)}$ by $$k_{n,jl}(x_j, x_l)= \frac {p^n_{jl}(x_j, x_l)}{p^n_l(x_l)}-p^n_j(x_j)$$ for $j \not = l$ and $k_{n,jl}(x_j, x_l)=0$ for $j=l$. With this modified kernel we rewrite the integral equations as
\begin{eqnarray}  \label{addassB4add1}
 \beta^*  &=& \alpha^* - {\mathcal K}^* \beta^*, \\ \label{addassB4add2}
\beta  &=& \alpha - {\mathcal K} \beta, 
 \end{eqnarray}
 with $ \beta^* =( \beta^* _l)_{l\in J_k}$, $ \beta =( \beta _l)_{l\in \{2,...,q\}}$. 
 Furthermore,
 ${\mathcal K}^*$ and ${\mathcal K}$ are integral operators with
 $ ({\mathcal K}^{*}g)_j (x_j) = \sum_{l\in J_k} \int k_{n,jl}(x_j,x_l)g_l(x_l) \mathrm d x_l$ for a tuple $g(x) = (g_l(x_l))_{l \in J_k
}$ of functions and $j\in J_k$ and with
 $ ({\mathcal K}g)_j (x_j) = \sum_{l=2} ^{q} \int k_{n,jl}(x_j,x_l)g_l(x_l) \mathrm d x_l$ for a tuple $g(x) = (g_l(x_l))_{l=2} ^q$ of functions and $j=2,...,q$.

 From (B1) and \eqref{eq:asslem1} we get  for all  $c_*$ that for $1 \leq i,l \leq q$, $i \not = l$
\begin{equation} \label{addassB43} |k_{n,il}(x_i,x_l)  | \leq c_* \rho^{|i-l|} |i-l|^{-1/2}\end{equation}
 for a new choice of $0 < \rho < 1$  if $c^*$ has been chosen small enough. In particular we can choose $c^*$ such that \eqref{addassB43} holds with $$c_* =  \frac {1- \rho^2}{5-\rho^2}.$$

 For the proof of the lemma we have to show that  \eqref{addassB43} implies that our assumption (B4) holds with $s_1= C \log n$ for $C$ large enough and with $J_k=\{2,...,s_1\}$. 
 
 For the proof of this claim note first that  \eqref{addassB43} implies that for $i=2,...,q$
 \begin{eqnarray} \label{addassB44} |\alpha_i(x_i)  | &\leq  & \left |\int  b_{1,i} (x_1)k_{n,1i}(x_1, x_i) \ \mathrm{d} x_1 \right |
 \\ \nonumber 
&\leq  &  c_* \rho^{|i-1|} |i-1|^{-1/2} \int\left |  b_{1,i} (x_1) \right | \ \mathrm{d} x_1
  \\ \nonumber 
&\leq  & 
  C^* m_1^{-1/2} \rho^{i} \end{eqnarray}
 with some constant $C^*$ depending only on $\rho$. We denote by ${\mathcal K}^r$ the operator consisting of $r$ iterative applications of ${\mathcal K}$.
 We now argue that for the kernel $k^r _{n,jl}$ of the operator ${\mathcal K}^r$ it holds that
  \begin{equation} \label{addassB45}|k^r _{n,jl}(x_j,x_l) | \leq c_* \rho^{|j-l|} 2^{-r+1} \min \{1,  |j-l|^{-1/2}\},\end{equation}
  where $c_*$,  $\rho$  are as in \eqref{addassB43}. 
  We establish \eqref{addassB45} by induction. Suppose \eqref{addassB45} holds for $r=r_0$. For $r=r_0+1$ we get for $j \not = l$ that 
  \begin{eqnarray*}
|k^{r _0+1} _{n,jl}(x_j,x_l) |&=& \left |\sum _{i \in \{1,...,q\}\backslash \{l\}}\int k^{r _0} _{n,ji}(x_j,x_i) k_{n,il}(x_i,x_l) \mathrm{d} x_i\right | \\
&\leq& \sum _{2 \leq i < j} (c_*)^2 \rho^{j+l-2i} (j-i)^{-\nu}(l-i)^{-\nu}2^{-r_0+1} \\
&& +  (c_*)^2 \rho^{l-j} (l-j)^{-{1/2}}2^{-r_0+1}\\
&&+ \sum _{j < i <l} (c_*)^2 \rho^{l-j} (i-j)^{-{1/2}}(l-i)^{-{1/2}}2^{-r_0+1}\\
&& +  \sum _{l < i \leq q} (c_*)^2 \rho^{2i-j-l} (i-j)^{-{1/2}}(i-l)^{-{1/2}}2^{-r_0+1}\\
&\leq& (c_*)^2 \rho^{l-j}2^{-r_0+1}(l-j)^{-{1/2}} (2 \rho^2 (1-\rho^2)^{-1} + 1+4)\\
&=& c_* \rho^{l-j}(l-j)^{-{1/2}}2^{-r_0}   ,
\end{eqnarray*}
where we supposed without loss of generality that $j<l$ and where we used with $k = l-j$ that for $k>0$
\begin{eqnarray*} &&\sum_{i=1}^k i^{-1/2} (k-i)^{-1/2} = k^{-1} \sum_{i=1}^k( i/k)^{-1/2} ((k-i)/k)^{-1/2}\\
&& \qquad \leq \int_0^1 x^{-1/2} (1-x)^{-1/2} \mathrm d x \\
&&  \qquad =2 \sqrt {2} \int_0^{1/2} x^{-1/2}  \mathrm d x \\
&&  \qquad =2 \sqrt {2}\left . \left[2 x^{1/2}\right] \right |_0^{1/2}  \\
&& \qquad = 4.
\end{eqnarray*}

This shows \eqref{addassB45} for $j \not = l$. For  $j  = l$ one gets that
\begin{eqnarray*}
|k^{r _0+1} _{n,jj}(x_j,x_j^\prime) |&=& \left |\sum _{i \in \{2,...,q\}\backslash \{j\}}\int k^{r _0} _{n,ji}(x_j,x_i) k_{n,ij}(x_i,x_j^\prime) \mathrm{d} x_i\right | \\
&\leq& \sum _{2 \leq i < j} (c_*)^2 \rho^{2(j-i)} (j-i)^{-1}2^{-r_0+1} \\
&& +  \sum _{j < i \leq q} (c_*)^2 \rho^{2(i-j)} (i-j)^{-1}2^{-r_0+1}\\
&\leq& (c_*)^2 2^{-r_0+1} 2 \rho^2 (1-\rho^2)^{-1} \\
&\leq& c_* 2^{-r_0} 4  \rho^2 (5-\rho^2)^{-1}\\
&\leq& c_* 2^{-r_0} .
\end{eqnarray*}
This concludes the proof of  \eqref{addassB45}.

From \eqref{addassB44} and \eqref{addassB45} one gets by direct calculations for $r \geq 0$

\begin{eqnarray} \label{addassB46} |({\mathcal K}^r \alpha)_i(x_i)  |&=& \left |\sum_{j=2} ^q  \int k^r _{n,ij}(x_i,x_j)  \alpha_j(x_j) \mathrm d x_j \right |
\\ \nonumber 
&\leq & \sum_{j=2} ^q  c_* \rho^{|i-j|} 2^{-r+1} \min \{1,  |i-j|^{-1/2}\}\int  |\alpha_j(x_j)| \mathrm d x_j
\\ \nonumber 
&\leq & \sum_{j=2} ^q  c_* \rho^{|i-j|} 2^{-r+1}  C^* m_1^{-1/2} \rho^{j} 
\\ \nonumber 
&\leq &
 C^{**} m_1^{-1/2}\rho^{i} 2^{-r}
 \end{eqnarray}
with a constant $C^{**}$ depending only on $\rho$. We now apply $\beta_i(x_i) = ((I +{ \mathcal K} ) ^{-1} \alpha )_i (x_i) = \sum^\infty_{r=1} ((-{ \mathcal K} ) ^{r} \alpha )_i (x_i)$ which follows from \eqref{addassB4add2}.
Thus, by help of  \eqref{addassB46}, we have that
\begin{eqnarray} \label{addassB47} |\beta_i(x_i)  |& \leq  & \left |\sum^\infty_{r=1} \sum_{j=1} ^q  \int k^r _{n,ij}(x_i,x_j)  \alpha_j(x_j) \mathrm d x_j \right |
\\ \nonumber 
&\leq & C^{**} m_1^{-1/2}\rho^{i} \sum^\infty_{r=1}2^{-r}
\\ \nonumber 
&\leq &
C^{**} m_1^{-1/2}\rho^{i}.\end{eqnarray}

We now define tuples of functions 
 $\beta^{**} =(\beta^{**} )_{j=s_1+1}^{q}$,   and $\Delta \beta = (\Delta \beta)_{j=2}^{s_1}$  and an operator ${\mathcal K}^{**}$ acting on  tuples of functions by 
$\beta^{**} _j (x_j) = \beta _j (x_j)$ for $ {j=s_1+1,...,q}$, $ ({\mathcal K}^{**}g)_j (x_j) = \sum_{l=s_1+1} ^q \int k_{n,jl}(x_j,x_l)g_l(x_l) \mathrm d x_l$ for a tuple $g(x) = (g_l(x_l))_{l=s_1+1} ^q$ of functions and $j=2,...,s_1$ and $(\Delta \beta)_j(x_j) = \beta _j (x_j)- \beta^{*} _j (x_j)$ for $j=2,...,s_1$.

For the proof of Assumption (B4) we have to check that 
 \begin{eqnarray} \label{addassB4claim}&&\sum_{j=2} ^{s_1}\left ( \int |(\Delta \beta)_j(x_j)|^2 \mathrm d x_j\right ) ^{1/2} + \sum_{j=s_1+1} ^{q} \left ( \int |\beta_j(x_j)| \mathrm d x_j \right ) ^{1/2}   \\ \nonumber &&\qquad \leq C_1 \sqrt{s_1} \sqrt {d/n}\end{eqnarray}
 for some constant $C_1$. We will show now that for the left hand side of \eqref{addassB4claim} a bound holds of the form $C^{***} m_1^{-1/2} \rho^{s_1} $ with a constant $C^{***}$ depending only on $\rho$. 
 Thus if $m_1$, $d$ grow polynomially in $n$ we have that under Assumption \eqref{addassB43} our Condition (B4) holds for choices of $s_1$ of order $C \log n$ with $C$ large enough. 
 
 We now develop the bound for the left hand side of \eqref{addassB4claim}. We first note that $\Delta \beta$ can be calculated by solving the linear integral equation
  \begin{equation} \label{addassB4help1} \Delta \beta = - {\mathcal K} ^{**} \beta^{**} -  {\mathcal K} ^{*} \Delta \beta.\end{equation}
  Note that \eqref{addassB4help1} follows by taking the differences of the elements of both sides of \eqref{addassB4add1} minus the elements in $J_k$ of both sides of \eqref{addassB4add2}.
For the intercept $ - {\mathcal K} ^{**} \beta^{**}$ we get from our assumption \eqref{addassB43} and \eqref{addassB47} that for $j=2,...,s_1$
 \begin{eqnarray*} | - ({\mathcal K} ^{**} \beta^{**})_j (x_j) | &\leq& \left | \sum_{l=s_1+1} ^q \int k_{n,jl}(x_j,x_l) \beta_l(x_l) \mathrm d x_l \right |
 \\ &\leq& \sum_{l=s_1+1} ^q c_* \rho^{l-j} 2^{-r+1}  (l-j)^{-1/2} C^{**} m_1^{-1/2}\rho^{l}
  \\ &\leq&
  C^{+} m_1^{-1/2}\rho^{2 s_1 -j}\end{eqnarray*}
with a constant $C^{+}$ depending only on $\rho$. Now, $\Delta \beta$ can be calculated as above it was done for $\beta$. Note that 
$$\Delta \beta = ( I+  {\mathcal K} ^{*}) ^{-1} (- ({\mathcal K} ^{**} \beta^{**}))= \sum_{r=0} ^\infty ( {\mathcal K} ^{*})^r  (- ({\mathcal K} ^{**} \beta^{**})).$$
With similar arguments as in the proof of \eqref{addassB45} one gets that   $$|{k}^{*r} _{n,jl}(x_j,x_l) | \leq c_* \rho^{|j-l|} 2^{-r+1} \min \{1,  |j-l|^{-1/2}\},$$
  where  ${k}^{*r} _{n,jl}$ is the kernel of the operator ${{\mathcal K}^*}^r$ and
$c_*$,  $\rho$  are as in \eqref{addassB43}. 
With this bound we get for $i=2,...,s_1$
 \begin{eqnarray} \label{addassB4help2} &&|(\Delta \beta)_i (x_i) | \leq  \left |\sum^\infty_{r=1} \sum_{j=2} ^{s_1}   \int {k}^{*r} _{n,ij}(x_i,x_j)  ( {\mathcal K} ^{**} \beta^{**} )_j(x_j) \mathrm d x_j \right |
 \\ \nonumber && \qquad \leq   \sum^\infty_{r=1} \sum_{j=2} ^{s_1}         c_* \rho^{|j-i|} 2^{-r+1} \min \{1,  |j-i|^{-1/2}\}          C^{+} m_1^{-1/2}\rho^{2 s_1 -j}
  \\ \nonumber && \qquad \leq   4 c_*  C^{+} m_1^{-1/2} \left ( \sum_{j=2} ^{i}         \rho^{i-j}       \rho^{2 s_1 -j} + \sum_{j=i+1} ^{s_1}         \rho^{j-i}       \rho^{2 s_1 -j} \right )
\\ \nonumber && \qquad \leq 
  C^{++} m_1^{-1/2}s_1\rho^{2 s_1 -i}.\end{eqnarray}
with a constant $C^{++}$ depending only on $\rho$. From \eqref{addassB47} and \eqref{addassB4help2} we get the desired bound for the left hand side of \eqref{addassB4claim}: 
 \begin{eqnarray} \label{addassB4claimhelp3}&&\sum_{j=2} ^{s_1}\left ( \int |(\Delta \beta)_j(x_j)|^2 \mathrm d x_j\right ) ^{1/2} + \sum_{j=s_1+1} ^{q} \left ( \int |\beta_j(x_j)| \mathrm d x_j \right ) ^{1/2} \\ \nonumber &&\qquad\leq C^{***} m_1^{-1/2} s_1\rho^{s_1}\end{eqnarray}
 with a constant $C^{***}$ depending only on $\rho$. For any constant $C_1$ we now have, with a choice $s_1= C \log n$ with $C$ large enough, that the right hand side of this inequality is smaller than $C_1 \sqrt{s_1} \sqrt {d/n}$. This shows that  under the mixing condition \eqref{addassB43}  assumption (B4) holds with $s_1= C \log n$ for $C$ large enough and concludes the proof of Lemma  \ref{lemB4check}.  \endproof

\section{Additional simulation details}\label{simdetails}

\subsection{The presmoothing estimator}

In each simulation, the presmoothing estimators $\hat f_1^{\operatorname{(pre)}},\dots,\hat f_q^{\operatorname{(pre)}}$ are constructed from a cubic B-spline basis with $d^{\operatorname{(pre)}} = \lfloor 2n^{1/2} \rfloor$ basis functions (chosen large to achieve undersmoothing) generated from knots placed at the quantiles
\[
0,0,0,0,1/(d^{\operatorname{(pre)}} - 3),\dots,(d^{\operatorname{(pre)}}-2)/(d^{\operatorname{(pre)}}-3),1,1,1,1
\]
of the data points $X_j^1,\dots,X_j^n$ for $j=1,\dots,q$. Each basis function is empirically centered in order that $\sum_{i=1}^q \hat f^{\operatorname{(pre)}}_j(X_j^i)=0$ for all $j=1,\dots,q$; we also use the centered response data $Y^i-n^{-1}\sum_{j=1}^n Y^j$, $i=1,\dots,n$. We choose the tuning parameters $\lambda$ and $\eta$ via $5$-fold crossvalidation for $10$ of the $500$ simulation runs under each setting, and then we use the median of the $10$ crossvalidation values of $\lambda$ and $\eta$ under each setting for the remaining $490$ simulations; this was to reduce computation time.
\par

\subsection{The resmoothing and oracle estimators}

The resmoothing estimator $\hat f_1^{\text{(spl)}}$ and the oracle estimator $\hat f_1^{(\operatorname{oracle,spl})}$ are constructed from a cubic B-spline basis with $d^{\operatorname{(re)}}$ and $d^{\operatorname{(oracle)}}$ basis functions, respectively, with knots placed at data quantiles as with the presmoothing estimator.  We choose $d^{\operatorname{(re)}}$ and $d^{\operatorname{(oracle)}}$ via $5$-fold crossvalidation as follows: 
\par
For the oracle estimator, we consider smoothing the data given by the pairs $(X_1^1,Z^1),\dots,(X_1^n,Z^n)$ when sorted such that $X_1^1 < \cdots < X_1^n$.
Let $\mathcal{F}_k = \{k,k + 5,k+10,\dots\}\cap\{1,\dots,n\}$ for $k=1,\dots,5$ and define
\[
\hat f_{1,d}^{\operatorname{(oracle,spl)},-k} = \arg\min_{g\in V_1^{d,-k}}\sum_{i \in\{1,\dots,n\}\setminus \mathcal{F}_k}[Z^i - g(X_1^i)]^2,
\]
where $V_1^{d,-k}$ is the space of cubic splines with knots at the quantiles 
\[
0,0,0,0,1/(d - 3),\dots,(d-2)/(d-3),1,1,1,1
\]
of $(X_1^i,i\in\{1,\dots,n\}\setminus \mathcal{F}_k)$.  We then consider for $d = 4,\dots,d^{\operatorname{(pre)}}$ the values of the crossvalidation prediction error
\[
\text{CV}(d)=\frac{1}{n}\sum_{k=1}^K\sum_{i \in \mathcal{F}_k}[ Z^i - \hat f_{1,d}^{\operatorname{(oracle,spl)},-k}(X_1^i) ]^2.
\]
If $d^*$ is the value of $d$ which minimizes $\text{CV}(d)$, then we choose 
\[
d^{\operatorname{(oracle)}} = \max\{ d \in\{4,\dots,d^{\operatorname{(pre)}}\} : \text{CV}(d) \leq \text{CV}(d^*) + \widehat{\text{se}}(\text{CV}(d^*))\},
\]
where $\widehat{\text{se}}(\text{CV}(d^*))$ is an estimate of the standard error of the crossvalidation prediction error at $d^*$. In essence, we find the  value of $d$ which minimizes the crossvalidation prediction error and then increase it somewhat in order to achieve a controlled amount of undersmoothing.  The intent is to undersmooth just enough to make the bias of $\hat f_1^{\operatorname{(oracle,spl)}}(x)$ small with respect to its standard error, which is required for the oracle confidence interval to be asymptotically correct for $f_1(x)$.  To find $d^{\operatorname{(re)}}$ for the resmoothed estimator, we do the same procedure but with the data given by the pairs $(Y^i - \hat f_{-1}^L(X_{-1}^i),X_1^i)$, $i=1,\dots,n$, sorted such that $X_1^n<\cdots<X_1^n$.

\subsection{The variance estimator}\label{simvarest}

In our simulations we used the variance estimator given by
\[
\hat \sigma_1^2 = \frac{1}{\hat \nu} \sum_{i=1}^{n-1}\left[\hat f_1^{\operatorname{(pre)}}(X_1^{i+1}) - \hat f_1^{\operatorname{(pre)}}(X_1^{i}) \right]^2 = \frac{1}{\hat \nu}\|D\hat \f_1^{\operatorname{(pre)}}\|_2^2
\]
after the data pairs  $(X^1,Y^1),\dots,(X^n,Y^n)$ have been reordered such that $X_1^1 <\cdots<X_1^n$, where $D$ is the $(n-1)\times n$ matrix with entries defined by
\[
D_{ij} = \left\{ \begin{array}{rl}
                -1    &  \text{if } i = j\\
                1  & \text{if } i = j - 1\\
                0 & \text{otherwise}
                \end{array}\right. \quad 
\]
for $i=1,\dots,n-1$ and $j=1,\dots,n$, and 
\[
\hat \nu = \operatorname{tr}[(D(I-A)^{-1}(\hat \Pi_1 - A))^TD(I-A)^{-1}(\hat \Pi_1 - A)].
\]
The estimator $\hat \sigma_1^2$ is defined in analogy to the variance estimator in the oracle model given by
\[
[\hat \sigma^{\operatorname{(oracle)}}]^2 = \frac{1}{2(n-1)}\sum_{i=1}^{n-1}[Z^{i+1} - Z^{i}]^2,
\]
when the pairs $(X_1^1,Z^1),\dots,(X_1^n,Z^n)$ are sorted such that $X_1^1 <\cdots<X_1^n$. The random variables $Z^1,\dots,Z^n$ are defined in \eqref{eqn:ormodel}.

\subsection{Additional simulation output} \label{additionaloutput}
\FloatBarrier

We see in Table \ref{table_cov_diff} that our resmoothed estimator achieved less-than-nominal coverage under some settings. In order to investigate whether this was due to biased point estimation we considered the proportion of times our pointwise confidence intervals contained the expected value of the point estimators.  That is, we recorded the proportion of times over the $500$ simulated data sets that the pointwise confidence intervals contained the value $\mathbf{E}\hat f_{j}^{\operatorname{(spl)}}(x)$, where this was obtained as the average value of the point estimates. These proportions are presented in Table \ref{tab:covofmean} (we consider also the sample size $n=500$).  We see that the coverages are all very close to the nominal coverage of $95\%$.  This suggests that the instances of less-than-nominal coverage of our resmoothing estimator are primarily due to bias in the resmoothed estimator.  Some of this bias comes from the sparsity penalization used in the construction of the presmoothing estimator; however, some of the bias is the unavoidable nonparametric estimation bias from which the oracle estimator also suffers.

\begin{table}[ht]
\centering
\begin{tabular}{crr|rrrrrrrrrr}
  \multicolumn{3}{c}{ }& \multicolumn{ 2 }{c}{ $f_1$ } & \multicolumn{ 2 }{c}{ $f_2$ } & \multicolumn{ 2 }{c}{ $f_3$ } & \multicolumn{ 2 }{c}{ $f_4$ } & \multicolumn{ 2 }{c}{ $f_5$ } \\$n$ & $q$ & $\zeta$ & -1.0 & 0.5 & -1.0 & 0.5 & -1.0 & 0.5 & -1.0 & 0.5 & -1.0 & 0.5 \\ 
  \hline
100 & 50 & 0.0 & 95 & 94 & 93 & 93 & 95 & 93 & 94 & 93 & 94 & 94 \\ 
   &  & 0.1 & 94 & 95 & 93 & 94 & 95 & 93 & 96 & 94 & 94 & 95 \\ 
   &  & 0.3 & 94 & 96 & 93 & 92 & 94 & 96 & 94 & 93 & 95 & 93 \\ 
   &  & 0.5 & 94 & 95 & 94 & 94 & 91 & 95 & 96 & 92 & 93 & 96 \\ 
   & 150 & 0.0 & 94 & 94 & 92 & 92 & 95 & 94 & 93 & 93 & 95 & 93 \\ 
   &  & 0.1 & 95 & 94 & 92 & 94 & 95 & 93 & 94 & 96 & 95 & 95 \\ 
   &  & 0.3 & 94 & 94 & 91 & 96 & 92 & 94 & 93 & 93 & 94 & 96 \\ 
   &  & 0.5 & 95 & 95 & 92 & 95 & 94 & 96 & 95 & 92 & 95 & 93 \\ 
   \hline
500 & 50 & 0.0 & 93 & 94 & 94 & 95 & 93 & 94 & 96 & 96 & 95 & 95 \\ 
   &  & 0.1 & 93 & 95 & 97 & 95 & 95 & 92 & 96 & 96 & 95 & 95 \\ 
   &  & 0.3 & 96 & 94 & 96 & 95 & 94 & 95 & 95 & 96 & 94 & 94 \\ 
   &  & 0.5 & 93 & 96 & 95 & 96 & 93 & 93 & 94 & 96 & 93 & 94 \\ 
   & 150 & 0.0 & 95 & 94 & 95 & 94 & 95 & 94 & 95 & 96 & 95 & 94 \\ 
   &  & 0.1 & 94 & 94 & 97 & 95 & 95 & 95 & 95 & 96 & 94 & 96 \\ 
   &  & 0.3 & 93 & 95 & 95 & 97 & 94 & 93 & 96 & 96 & 93 & 93 \\ 
   &  & 0.5 & 93 & 94 & 95 & 94 & 90 & 93 & 95 & 95 & 93 & 95 \\ 
   \hline
1000 & 50 & 0.0 & 94 & 95 & 95 & 94 & 96 & 93 & 96 & 96 & 94 & 95 \\ 
   &  & 0.1 & 95 & 93 & 96 & 94 & 93 & 92 & 96 & 94 & 94 & 95 \\ 
   &  & 0.3 & 96 & 95 & 94 & 95 & 95 & 95 & 96 & 97 & 92 & 93 \\ 
   &  & 0.5 & 94 & 93 & 95 & 96 & 93 & 94 & 95 & 96 & 94 & 94 \\ 
   & 150 & 0.0 & 94 & 94 & 95 & 96 & 96 & 94 & 94 & 98 & 94 & 94 \\ 
   &  & 0.1 & 95 & 95 & 95 & 95 & 95 & 93 & 95 & 95 & 94 & 95 \\ 
   &  & 0.3 & 96 & 94 & 95 & 96 & 92 & 94 & 96 & 96 & 95 & 93 \\ 
   &  & 0.5 & 94 & 95 & 94 & 96 & 93 & 95 & 94 & 94 & 94 & 95 \\ 
  \end{tabular}
	\caption{Coverage $(\times 100)$ of $\mathbf{E}\hat f_{j}^{\operatorname{(spl)}}(x)$ of confidence interval in \eqref{ciresmoothed} for $f_1,\dots,f_5$ at $x = -1.0,0.5$.} 
\label{tab:covofmean}
\end{table}

Table \ref{tab:oraclecov} reports the coverage of the confidence interval based on the oracle estimator given by
\begin{equation}\label{cioracle}
\hat f_1^{\operatorname{(oracle,spl)}}(x) \pm z_{\alpha/2}\| \w_{x_0}^{\operatorname{(oracle,spl)}}\|_2\hat \sigma^{\operatorname{(oracle)}},
\end{equation}
where $\w_{x_0}^{\operatorname{(oracle)}}$ is the $n \times 1$ vector such that 
\[
	\hat f_j^{\operatorname{(oracle,spl)}}(x_0) = (\w_{x_0}^{\operatorname{(oracle,spl)}})^T(\f_1+\error).
\]
We note that even the coverage of the oracle confidence intervals is in some cases less than the nominal coverage, which we may attribute to the bias of the oracle estimator. For example, the coverage of $f_4$ at $x = -1$ is under several settings somewhat less-than-nominal, as that function is quite steep at $x=-1$.

\begin{table}[ht]
\centering
\begin{tabular}{crr|rrrrrrrrrr}
  \multicolumn{3}{c}{ }& \multicolumn{ 2 }{c}{ $f_1$ } & \multicolumn{ 2 }{c}{ $f_2$ } & \multicolumn{ 2 }{c}{ $f_3$ } & \multicolumn{ 2 }{c}{ $f_4$ } & \multicolumn{ 2 }{c}{ $f_5$ } \\$n$ & $q$ & $\zeta$ & -1.0 & 0.5 & -1.0 & 0.5 & -1.0 & 0.5 & -1.0 & 0.5 & -1.0 & 0.5 \\ 
  \hline
100 & 50 & 0.0 & 96 & 96 & 94 & 94 & 93 & 92 & 91 & 92 & 96 & 96 \\ 
   &  & 0.1 & 94 & 94 & 91 & 93 & 94 & 93 & 88 & 90 & 95 & 96 \\ 
   &  & 0.3 & 94 & 94 & 93 & 93 & 95 & 94 & 90 & 90 & 95 & 95 \\ 
   &  & 0.5 & 93 & 94 & 92 & 93 & 92 & 94 & 91 & 91 & 96 & 94 \\ 
   & 150 & 0.0 & 94 & 94 & 91 & 93 & 92 & 94 & 92 & 90 & 94 & 96 \\ 
   &  & 0.1 & 95 & 95 & 93 & 93 & 93 & 93 & 90 & 90 & 95 & 96 \\ 
   &  & 0.3 & 96 & 95 & 94 & 91 & 92 & 90 & 91 & 89 & 97 & 94 \\ 
   &  & 0.5 & 94 & 96 & 93 & 93 & 93 & 93 & 93 & 90 & 96 & 95 \\ 
   \hline
500 & 50 & 0.0 & 94 & 94 & 95 & 94 & 91 & 95 & 93 & 93 & 95 & 96 \\ 
   &  & 0.1 & 95 & 96 & 95 & 95 & 93 & 94 & 91 & 92 & 94 & 95 \\ 
   &  & 0.3 & 95 & 96 & 95 & 95 & 94 & 95 & 90 & 92 & 96 & 97 \\ 
   &  & 0.5 & 96 & 95 & 94 & 95 & 93 & 94 & 91 & 95 & 95 & 94 \\ 
   & 150 & 0.0 & 95 & 94 & 93 & 93 & 92 & 94 & 90 & 91 & 96 & 96 \\ 
   &  & 0.1 & 95 & 96 & 93 & 97 & 95 & 95 & 92 & 94 & 95 & 97 \\ 
   &  & 0.3 & 94 & 96 & 95 & 94 & 93 & 94 & 92 & 92 & 95 & 96 \\ 
   &  & 0.5 & 94 & 95 & 90 & 92 & 93 & 94 & 93 & 93 & 95 & 97 \\ 
   \hline
1000 & 50 & 0.0 & 96 & 96 & 95 & 94 & 95 & 95 & 94 & 92 & 94 & 96 \\ 
   &  & 0.1 & 95 & 95 & 95 & 93 & 93 & 92 & 92 & 93 & 95 & 96 \\ 
   &  & 0.3 & 95 & 96 & 96 & 94 & 95 & 93 & 91 & 93 & 95 & 96 \\ 
   &  & 0.5 & 95 & 96 & 91 & 93 & 95 & 95 & 93 & 94 & 95 & 95 \\ 
   & 150 & 0.0 & 96 & 95 & 93 & 93 & 93 & 93 & 93 & 93 & 94 & 96 \\ 
   &  & 0.1 & 95 & 94 & 94 & 95 & 95 & 94 & 93 & 93 & 95 & 95 \\ 
   &  & 0.3 & 96 & 95 & 94 & 94 & 93 & 96 & 92 & 91 & 95 & 96 \\ 
   &  & 0.5 & 97 & 95 & 92 & 93 & 93 & 94 & 95 & 93 & 96 & 95 \\ 
  \end{tabular}
\caption{Coverage $(\times 100)$ of oracle confidence intervals for functions $f_1,\dots,f_5$ at $x = -1.0,0.5$.} 
\label{tab:oraclecov}
\end{table}

\end{document}